\documentclass[12pt]{article}
\usepackage[colorlinks=true,linkcolor=blue,citecolor=Red]{hyperref}
\usepackage{breakcites}

\usepackage{mathtools}
\DeclarePairedDelimiter\ceil{\lceil}{\rceil}
\DeclarePairedDelimiter\floor{\lfloor}{\rfloor}
\usepackage{subfigure}

\newcommand\independent{\protect\mathpalette{\protect\independenT}{\perp}}
\def\independenT#1#2{\mathrel{\rlap{$#1#2$}\mkern2mu{#1#2}}}

\usepackage{algorithm}
\usepackage{algpseudocode}

\usepackage{latexsym}
\usepackage{graphics}
\usepackage{amsmath}
\usepackage[algo2e]{algorithm2e}
\usepackage{xspace}
\usepackage{amssymb}
\usepackage{psfrag}
\usepackage{epsfig}
\usepackage{amsthm}
\usepackage{url}
\usepackage{pst-all}

\usepackage{color}

\definecolor{Red}{rgb}{1,0,0}
\definecolor{Blue}{rgb}{0,0,1}
\definecolor{Olive}{rgb}{0.41,0.55,0.13}
\definecolor{Yarok}{rgb}{0,0.5,0}
\definecolor{Green}{rgb}{0,1,0}
\definecolor{MGreen}{rgb}{0,0.8,0}
\definecolor{DGreen}{rgb}{0,0.55,0}
\definecolor{Yellow}{rgb}{1,1,0}
\definecolor{Cyan}{rgb}{0,1,1}
\definecolor{Magenta}{rgb}{1,0,1}
\definecolor{Orange}{rgb}{1,.5,0}
\definecolor{Violet}{rgb}{.5,0,.5}
\definecolor{Purple}{rgb}{.75,0,.25}
\definecolor{Brown}{rgb}{.75,.5,.25}
\definecolor{Grey}{rgb}{.5,.5,.5}

\newcommand{\Z}{\mathbb{Z}}

\newcommand{\R}{\mathbb{R}}
\newcommand{\N}{\mathbb{N}}
\newcommand{\fl}[1]{\left\lfloor #1 \right\rfloor}

\renewcommand{\Z}{\mathbb{Z}}
\newcommand{\Q}{\mathbb{Q}}
\newcommand{\h}[1]{\widehat{#1}}

\newcommand{\inner}[2]{\langle{#1},{#2}\rangle}

\renewcommand{\R}{\mathbb{R}}
\renewcommand{\Z}{\mathbb{Z}}

\newcommand{\sign}{\mathrm{sign}}

\newcommand{\ignore}[1]{\relax}

\newlength\myindent
\newtheorem{assumptions}{Assumption}
\newtheorem{theorem}{Theorem}[section]
\newtheorem{remark}[theorem]{Remark}
\newtheorem{lemma}[theorem]{Lemma}

\newtheorem{proposition}[theorem]{Proposition}

\newtheorem{claim}[theorem]{Claim}

\newtheorem{definition}[theorem]{Definition}

\makeatletter
\newenvironment{subtheorem}[1]{%
  \def\subtheoremcounter{#1}%
  \refstepcounter{#1}%
  \protected@edef\theparentnumber{\csname the#1\endcsname}%
  \setcounter{parentnumber}{\value{#1}}%
  \setcounter{#1}{0}%
  \expandafter\def\csname the#1\endcsname{\theparentnumber.\Alph{#1}}%
  \ignorespaces
}{%
  \setcounter{\subtheoremcounter}{\value{parentnumber}}%
  \ignorespacesafterend
}
\makeatother
\newcounter{parentnumber}

\makeatletter
\def\BState{\State\hskip-\ALG@thistlm}
\makeatother

\definecolor{Red}{rgb}{1,0,0}
\definecolor{Blue}{rgb}{0,0,1}
\definecolor{Olive}{rgb}{0.41,0.55,0.13}
\definecolor{Green}{rgb}{0,1,0}
\definecolor{MGreen}{rgb}{0,0.8,0}
\definecolor{DGreen}{rgb}{0,0.55,0}
\definecolor{Yellow}{rgb}{1,1,0}
\definecolor{Cyan}{rgb}{0,1,1}
\definecolor{Magenta}{rgb}{1,0,1}
\definecolor{Orange}{rgb}{1,.5,0}
\definecolor{Violet}{rgb}{.5,0,.5}
\definecolor{Purple}{rgb}{.75,0,.25}
\definecolor{Brown}{rgb}{.75,.5,.25}
\definecolor{Grey}{rgb}{.5,.5,.5}
\definecolor{Pink}{rgb}{1,0,1}
\definecolor{DBrown}{rgb}{.5,.34,.16}
\definecolor{Black}{rgb}{0,0,0}

\usepackage{float}

\usepackage{float}
\author{
{\sf David Gamarnik}\thanks{MIT; e-mail: {\tt gamarnik@mit.edu}. Research supported  by the NSF grants CMMI-1335155.}
\and
{\sf Eren C. K{\i}z{\i}lda\u{g}\thanks{MIT; e-mail: {\tt kizildag@mit.edu}}}
\and
{\sf Ilias Zadik}\thanks{NYU; e-mail: {\tt zadik@nyu.edu}. Research supported by a CDS Moore-Sloan Postdoctoral Fellowship.}
}

\begin{document}

\title{Inference in High-Dimensional Linear Regression\\ via Lattice Basis Reduction and Integer Relation Detection\footnote{Parts of this paper have been presented at 2018 Conference on Neural Information Processing Systems (NeurIPS) \cite{zadik2018high} and 2019 IEEE International Symposium on Information Theory (ISIT) \cite{kizildaggamarnik}. Several synthetic experiments can be found in 
\cite{zadik2018high}.}}
\date{}

\maketitle

\begin{abstract}
We focus on the high-dimensional linear regression (HDR) problem, where the algorithmic goal is to efficiently recover an unknown feature vector $\beta^*\in\R^p$ from its linear measurements, using a small number $n$ of measurements. Unlike most of the literature on this model, we make no sparsity assumption on $\beta^*$, but instead adopt a different regularization:
\begin{itemize}
    \item[(a)] In the noiseless setting, we assume $\beta^*$ consists of entries, which are either rational numbers with a common denominator $Q\in\mathbb{Z}^+$ (referred to as $Q$-rationality) or irrational numbers supported on a rationally independent set of bounded cardinality, known to learner; collectively called the {\em mixed-support} assumption. Using a novel combination of PSLQ integer relation and LLL lattice basis reduction algorithms, we propose an polynomial-time algorithm exactly recovering a $\beta^*\in\R^p$ enjoying mixed-support assumption, from its linear measurements $Y=X\beta^*\in\R^n$ for a large class of distributions for the random entries of $X$, even with one measurement $(n=1)$. We then apply these ideas and develop a polynomial-time, single-sample algorithm for the phase retrieval problem, where a $\beta^*\in\R^p$ is to be recovered from magnitude-only measurements $Y=|\langle X,\beta^*\rangle|$. 
    \item[(b)] In the noisy setting, we propose a polynomial-time, lattice-based algorithm, which recovers a $\beta^*\in\R^p$ enjoying $Q$-rationality, from its noisy measurements $Y=X\beta^*+W\in\R^n$, even with a single sample $(n=1)$. We further establish for large $Q$, and normal noise, this algorithm tolerated information-theoretically optimal level of noise, using insights provided from the capacity of Gaussian channels with power constraint. 
\end{itemize}
Our algorithms address the single-sample $(n=1)$ regime, where the sparsity-based methods such as LASSO and Basis Pursuit are known to fail, and therefore showing there is no {\em computational-statistical gap}; namely, not only it is information-theoretically possible to recover $\beta^*$ with one sample, the recovery can also be achieved in a computationally efficient way. Furthermore, our results will also reveal an algorithmic connection between the regression problem and the integer relation detection, randomized subset-sum, and shortest vector problems.

\end{abstract}

\tableofcontents

\section{Introduction}
 We consider the following high-dimensional linear regression model. Consider $n$ linear measurements of a vector $\beta^* \in \mathbb{R}^p$ through the measurement model, $Y=X\beta^*$ for some $X \in \mathbb{R}^{n \times p}$; or its noisy version $Y=X\beta^*+W\in\R^n$, where $W\in\R^n$ is independent additive noise. 
 Given the knowledge of $Y$ and $X$ the algorithmic goal is to efficiently infer $\beta^*$ using a small number $n$ of measurements. Throughout the paper we refer to $p$ as the number of features,  to $X$ as the measurement matrix, and to $n$ as number of measurements of samples, in short.

 We focus on the high-dimensional case where $n$ may be much smaller than $p$ and $p$ grows to infinity, a setting that has been very popular in the literature during the last years \cite{SignDen}, \cite{donoho2006compressed}, \cite{candes}, \cite{foucart2013mathematical}, \cite{wainwright2009sharp}. In this case, and under no additional structural assumption, the inference task becomes impossible (even in the absence of noise), as the underlying linear system becomes underdetermined. Most papers address this issue by imposing a \textit{sparsity assumption} on $\beta^*$, which refers to $\beta^*$ having only a limited number of non-zero entries compared to its dimension  \cite{donoho2006compressed}, \cite{candes}, \cite{foucart2013mathematical}, which is motivated from an empirical observation that many signals (such as medical image and photos) exhibit a sparse behaviour, when translated to an appropriate transform domain. During the past decades, the sparsity assumption led to a fascinating line of research in statistics and compressed sensing communities, with a wide-range of applications, from single-pixel camera to medical imaging and radar imaging, which established, among other results, that several polynomial-time algorithms, such as Basis Pursuit Denoising Scheme and LASSO, can efficiently recover a sparse $\beta^*$ with number of samples much smaller than the number of features  \cite{candes}, \cite{wainwright2009sharp}, \cite{foucart2013mathematical}. For example, it is established that if $\beta^*$ is constrained to have at most $k \leq p$ non-zero entries, $X$ has iid $N(0,1)$ entries, $W$ has iid $N(0,\sigma^2)$ entries and $n $ is of the order $k \log \left(\frac{p}{k}\right)$, then both of the mentioned algorithms can recover $\beta^*$, up to the level of the noise. Different structural assumptions than sparsity have also been considered in the literature. For example, a recent paper \cite{Bora} makes the assumption that $\beta^*$ lies near the range of an $L$-Lipschitz generative model $G : \mathbb{R}^k \rightarrow \mathbb{R}^p$ and it proposes an algorithm which succeeds with $n=O(k \log L)$ samples. Other assumptions that have also been considered in the literature include tree-sparsity and block-sparsity.

A downside of all of the above results is that they provide no guarantee in the case $n$ is much smaller than  $k \log \left(\frac{p}{k}\right)$. Consider for example the case where the components of a sparse $\beta^*$ are binary-valued, and $X,W$ follow the Gaussian assumptions described above.  Supposing that $\sigma$ is sufficiently small, it is a straightforward argument that even when $n=1$, $\beta^*$ is recoverable from $Y=\inner{X}{\beta^*}+W$ by a brute-force method with probability tending to one as $p$ goes to infinity (whp). On the other hand, for sparse and binary-valued $\beta^*$, the Basis Pursuit method in the noiseless case \cite{donoho2006counting} and the LASSO in the noisy case \cite{gamarnikzadik2,wainwright2009sharp} have been proven to fail to recover a binary $\beta^*$ with $n=o(k \log \left( \frac{p}{k} \right))$ samples. 
This failure to capture the complexity of the problem accurately enough for small sample sizes also lead to an algorithmic hardness conjecture for the regime $n=o(k \log \left(\frac{p}{k}\right))$ \cite{gamarnikzadik}, \cite{gamarnikzadik2}. While this conjecture still stands in the general case, as we show in this paper, in the special case where $\beta^*$ takes mixed rational-irrational values (where the rational entries share a common denominator $Q\in\Z^+$ and the irrational entries are supported on a bounded cardinality set), and the magnitude of the noise $W$ is sufficiently small or zero the statistical/computational gap can be closed and $\beta^*$ can be recovered even in the extreme case, $n=1$. 

In this paper, we make no sparsity assumption on the entries of $\beta^*$. Instead, we achieve the regularization based on the assumption that the entries of $\beta^*$ are supported on a set $\mathcal{S}$ consisting of irrational and rational elements, that are assumed to satisfy certain structural properties, which we call the mixed-support assumption. 
Specifically, we assume that all rational entries of $\beta^*$ have denominator equal to some fixed positive integer $Q \in \mathbb{Z}^+$, something we refer to as the $Q$\textit{-rationality} assumption. The assumption that we impose on the irrational part of the support  of $\beta^*$ is that, it consists of irrational numbers $\{a_1,a_2,\dots,a_{\mathcal{R}}\}$ where the set $\{a_1,\dots,a_{\mathcal{R}}\}$ is known to the learner, and enjoys the so-called rational independence property, which  states that, the only rational combination of $a_1,\dots,a_{\mathcal{R}}$ adding up to $0$ is the trivial one, namely, if for $q_i\in\mathbb{Q}$, $\sum_{i=1}^{\mathcal{R}} q_ia_i=0$ holds, then it follows that $q_i=0$, for every $i\in [R]$. 

An assumption that $\beta^*$ consists of both mixed real and rational-valued entries is, in fact, well-motivated from a practical point of view. Several examples are as follows: In \cite{HassibiGPS}, authors study the global navigation satellite systems (GPS), and propose a mixed real/integer model of form $Y=Ax+Bz+W$, with $A\in\R^{n\times p}$, $B\in\R^{n\times q}$, $x\in \R^p$, $z\in\mathbb{Z}^q$, and $W\in \R^n$ being noise. The vector $z\in\mathbb{Z}^q$ here corresponds to the integer multiples of a certain wavelength. Recasting this system as $Y=X\beta^*+W$ with $X=\begin{bmatrix} A & B\end{bmatrix}\in\R^{n\times (p+q)}$, and $\beta^*= \begin{bmatrix} x^T & z^T \end{bmatrix}^T\in\R^{p+q}$, we observe that $\beta^*$ consists of both real and rational entries, with rational entries enjoying $1-$rationality assumption. Yet another example where an assumption similar to mixed-support assumption is employed, is from the field of geodesy and earth sciences, see \cite{teunissen2007best} and references therein. Even the special case, where $\beta^*$ has only integer entries (namely, $\beta^*$ enjoys $1-$rationality) is well-motivated from practice. As noted earlier, this assumption appears frequently in the study of global navigation satellite systems (GPS) and communications, see \cite{HassibiInteger,Brunel, BornoThesis}, in addition to \cite{HassibiGPS}. Several examples corresponding to noiseless regression models with integer valued regression coefficients are also discussed in the book \cite{foucart2013mathematical}. 
In particular one application is the so-called Single-Pixel camera. In this model a vector $\beta$ corresponds to color intensities of an image for different
pixels and thus takes discrete values. The model assumes no noise, which is one of the assumptions we adopt in our model, though the corresponding regression matrix has i.i.d. $+1/-1$ Bernoulli entries, as opposed to a continuous distribution we will mostly assume. Two other applications involving noiseless regression models found in the same reference are MRI imaging and Radar detection. The rational independence assumption on the irrational entries of $\mathcal{S}$ is imposed so as to study the most general case possible, namely the models enjoying the mixed-support assumption with $\beta^*\in\R^p$ consisting of both rational, but in addition, irrational entries. This assumption, moreover, also has a probabilistic connection: Almost all supports (that is, the supports generated in a random way) are rationally independent: This follows from the fact that if $a_1,\dots,a_{\mathcal{R}}$ are random numbers generated from a continuous-valued distribution, say, in an independent way, then the support $\{a_1,\dots,a_{\mathcal{R}}\}$ will almost surely be rationally independent. 

In particular, the mixed-support assumption that we pursue herein should be perceived as a concrete theoretical step towards analyzing more general setups, i.e. models with $\beta^*$ having mixed real and rational entries, with possibly much different structural properties, and even perhaps with sparsity.

Noiseless regression model with integer valued regression coefficients were also important in the theoretical development of compressive
sensing methods. Specifically,   Donoho~\cite{donoho2006compressed} and 
Donoho and Tanner~\cite{donoho2005neighborliness},\cite{donoho2006counting},\cite{donoho2009observed}
consider a noiseless regression model of the form $AB$
where $A$ is a random (say Gaussian) matrix and $B$ is the unit cube $[0,1]^p$. One of the  goals of these papers was  to count number of extreme points of the
projected polytope $AB$ in order to explain the effectiveness of the linear programming based methods. The extreme points of this
polytope can only appear as projections of extreme points of $B$ which are all length-p binary vector, namely one deals with noiseless regression
model with binary coefficients -- an important special case of the model we consider in our paper. In the Bayesian setting, where the ground truth $\beta^*$ is sampled according to a discrete distribution \cite{DonohoAMP} proposes a low-complexity algorithm which provably recovers $\beta^*$ with $n=o(p)$ samples. This algorithm uses the technique of approximate message passing (AMP) and is motivated by ideas from statistical physics \cite{ZdeborovaCompressed}. Even though the result from \cite{DonohoAMP} applies to the general discrete case for $\beta^*$, it requires the matrix $X$ to be spatially coupled, a property that in particular does not hold for $X$ with iid standard Gaussian entries. Furthermore the required sample size for the algorithm to work is only guaranteed to be sublinear in $p$, a sample size potentially much bigger than the information-theoretic limit for recovery under sufficiently small or zero noise ($n=1$), a gap which is filled in this paper. In the present paper, where $\beta^*$ satisfies the $Q$-rationality assumption, we propose a polynomial-time algorithm which applies for a large class of continuous distributions for the iid entries of $X$, including the normal distribution, and provably works even when $n=1$.

As noted earlier, our approach focuses both on noiseless and noisy setups, where we adopt mixed-support assumption in the noiseless case, and $Q-$rationality assumption in the noisy case, and thus, we also study the noise tolerance in this  case, an establishing information-theoretical noise optimality in case of normal noise and large $Q$. Our techniques, however, do not transfer to $\beta^*\in\R^p$ under our mixed-support assumption with noise, as a certain structure utilized by one of the building blocks of our algorithms is fragile under noisy measurement model. 
We then apply our ideas to the so-called phase retrieval problem, a problem which has received a large amount of attention in the study of noiseless regression type models; and has applications in physics and crystallography. Here the coefficients
of the regression vector $\beta^*$ and the entries of the regression matrix $X$ are complex valued, but the observation vector $Y=X\beta^*$ is only
observed through absolute values, that is, $Y_i=|\langle X_i,\beta^*\rangle|$, where $X_i$ is the $i^{th}$ row of the measurement matrix $X$, and $\langle \cdot,\cdot\rangle$ is the Euclidean inner product. This model has many applications, including crystallography, see~\cite{candes2015phase}.
The aforementioned paper provides many references to phase retrieval model including the cases when the entries of $\beta^*$ have a finite support. In this  paper, we also study the phase retrieval problem, under a random measurement model $X$, and complex-valued $\beta^*$, which we assume to consist of complex numbers from a known support $\mathcal{S}$, where a certain set $\mathcal{S'}$, obtained from $\mathcal{S}$ enjoys rational independence. We develop an algorithm which provably recovers $\beta^*\in \mathbb{C}^p$ whp, even with a single observation ($n=1$), after polynomial in $p$ and $\mathcal{R}$ many operations over real numbers, where $\mathcal{R}$ is the size of the set, $\mathcal{S}$. 



The algorithms we propose are inspired by the algorithm introduced in \cite{lagarias1985solving} which solves, in polynomial time, a certain version of the so-called randomized Subset-Sum problem, arising in cryptography. To be more specific, consider the following NP-hard algorithmic problem. Given $p \in \mathbb{Z}^+$, and the assumption that a set $S\subseteq[p]$ exists, the algorithmic goal is to recover $S$, from knowledge of $x_1,\dots,x_p\in\mathbb{Z}^+$, and $y=\sum_{i\in S}x_i\in\mathbb{Z}^+$. 
Here, one can interpret $\{x_i\}_{i=1}^p$ as a public information and $Y$ to be the ciphertext, where the plaintext ${\bf e}=(e_1,\dots,e_p)\in\{0,1\}^p$ is encrypted using $Y=\sum_{k=1}^n x_ke_p$. Over 30 years ago, this problem received a lot of attention in the field of cryptography, based on the belief that the problem would be hard to solve in many ``real" instances. This would imply that several already built public key cryptosystems, called knapsack public key cryptosystems, could be considered safe from attacks \cite{lempel1979cryptology}, \cite{merkle1978hiding}. This belief though was proven wrong by several papers in the early 80s, see for example  \cite{shamir1982polynomial}. Motivated by this line of research, Lagarias and Odlyzko in \cite{lagarias1985solving}, and a year later Frieze in \cite{FriezeSubset}, using a cleaner and shorter argument, proved the same surprising fact: if $x_1,x_2,\ldots,x_p $ follow an iid uniform distribution on $[2^{cp^2}]:=\{1,2,3,\ldots,2^{cp^2}\}$ for $c>1/2$, then there exists a polynomial-in-$p$ time algorithm which solves the subset-sum problem whp as $p \rightarrow +\infty$. In other words, even though the problem is NP-hard in the worst-case, assuming a quadratic-in-$p$ number of bits for the coordinates of $x$, the algorithmic complexity of the typical such problem is polynomial in $p$. The successful efficient algorithm is based on an elegant application of a seminal algorithm in the computational study of lattices called the Lenstra-Lenstra-Lovasz (LLL) algorithm, introduced in \cite{lenstra1982factoring}. This algorithm receives as an input a basis $\{b_1,\ldots,b_m\} \subset \mathbb{Z}^m$ of a full-dimensional lattice $\mathcal{L}$ and returns in time polynomial in $m$ and $\max_{i=1,2,\ldots,m} \log \|b_i\|_{\infty}$ what is known as a 'reduced basis', from which one can, in particular, obtain a non-zero vector $\hat{z}$  in the lattice, such that $\|\hat{z}\|_2 \leq 2^{\frac{m}{2}} \|z\|_2$, for all $z \in \mathcal{L}\setminus \{0\}$, which therefore solves $2^{m/2}$-approximate shortest vector problem on $\mathcal{L}$, in a time polynomial in $m$ and $\max_i \log \|b_i\|_\infty$.
  
Besides its significance in cryptography, the result of \cite{lagarias1985solving} and \cite{FriezeSubset} enjoys an interesting linear regression interpretation as well. In particular, one can show that under the assumption that $x_1,x_2,\ldots,x_p$ are iid discrete uniform in $[2^{\frac{1}{2}(1+\epsilon)p^2}]$, there exists exactly one set $S^*$ with $y=\sum_{i \in S^*}x_i$ whp as $p$ tends to infinity, and moreover, this set can be recovered, with high probability, in polynomial time. Therefore if $\beta^*$ is the indicator vector of this unique set $S^*$, that is $\beta^*_i=1(i \in S)$ for $i\in[p]$, we have that $y=\sum_{i}x_i\beta^*_i=\inner{x}{\beta^*}$ where $x:=(x_1,x_2,\ldots,x_p)$. Furthermore using only the knowledge of $y$ and $x$ as input to the Lagarias-Odlyzko algorithm, we obtain a polynomial in $p$ time algorithm which recovers exactly $\beta^*$ whp as $p \rightarrow +\infty$. Written in this form, and given our earlier discussion on high-dimensional linear regression, this statement is equivalent to the statement that the noiseless high-dimensional linear regression problem with binary $\beta^*$ and $X$ generated with iid elements from $\mathrm{Unif}[2^{\frac{1}{2}(1+\epsilon)p^2}]$ is polynomial-time solvable even with one sample $(n=1)$, whp as $p$ grows to infinity. One of the main focuses of this paper is to extend this result to $\beta^*$ satisfying the $Q$-rationality assumption, continuous distributions on the iid entries of $X$ and non-zero noise levels. 

In addition to its connection with the subset-sum problem, our irrationality assumption on $\beta^*$ will also reveal an interesting algorithmic connection between the linear regression problem and the well-known integer relation detection problem, which we discuss next. Given a vector, ${\bf  x}=(x_1,x_2,\dots,x_n)\in\mathbb{R}^n$, an integral relation for ${\bf x}$ is a vector ${\bf m}=(m_1,m_2,\dots,m_n)\in \mathbb{Z}^n$ of integers such that $\sum_{k=1}^n m_kx_k = 0$, and ${\bf m}$ is not identically $0$. Given an input ${\bf x}$, the goal of the integer relation detection problem is then to find an integer relation ${\bf m}$ for a given input ${\bf x}\in\R^n$. Note that, such a relation is not necessarily unique. Finding such a relation has been studied for a very long time, by many researchers, the earliest of whom is Euclid. 
In fact, for the simplest case of finding a relation for a vector with only two entries $x_1$ and $x_2$, one may apply Euclidean algorithm to continued fraction expansion of $x_1/x_2$. The case $n\geq 3$ was attempted by Euler, Perron, Minkowski, Bernstein, Brun, and many others; however none of them was provably working for $n\geq 3$. The first successful algorithm that works for $n\geq 3$ is devised by Ferguson and Forcade \cite{ferguson1979generalization}. In successive years, Ferguson provided an alternative variant of \cite{ferguson1979generalization} in \cite{ferguson1987noninductive}. After the introduction of the seminal Lenstra-Lenstra-Lov\'asz (LLL) lattice basis reduction algorithm \cite{lenstra1982factoring}, Hastad et al. \cite{hastad1989polynomial} devised the HJLS algorithm, which is based on the LLL lattice basis reduction algorithm \cite{lenstra1982factoring}, and provably recovers an integer relation for a given ${\bf x}\in\mathbb{R}^n$ in time polynomial in the dimension $n$ of ${\bf x}$, and polynomial in $\log\|{\bf m}\|$, where ${\bf m}\in\mathbb{Z}^n\setminus \{{\bf 0}\}$ is a relation for ${\bf x}$ with the smallest $\|{\bf m}\|$. The HJLS algorithm is the first algorithm proven to have this property. However, it has a drawback that it is numerically unstable; and this issue has further been resolved in PSLQ integral relation algorithm, which was introduced in Ferguson and Bailey \cite{ferguson1998polynomial}, whose analysis and further refinement was provided by Ferguson, Bailey and Arno \cite{ferguson1999analysis}.

PSLQ integer relation detection algorithm (where PS stands for partial-sum of squares, and LQ stands for the LQ decomposition) has been included in the list "Top Ten Algorithms of the Century" by the January/February 2000 issue of {\em Computing in Science and Engineering}, which is jointly published by American Institute of Physics and the IEEE Computer Society \cite{bailey2000integer}. The fundamental property of this algorithm that we will make use of in this paper is that, PSLQ algorithm provably returns a relation for a given ${\bf x}\in  \mathbb{R}^n$, if exists, after a number of operations that is polynomial both in the size $n$ of the input vector ${\bf x}$ as well as the number of bits required to represent the smallest such relation \cite{ferguson1999analysis}. 
Furthermore, it does not suffer from the numerical instability of HJLS algorithm \cite{hastad1989polynomial}, \cite{ferguson1999analysis}.

Since its introduction, PSLQ algorithm has been successfully used in many applications. For instance, it lead to the discovery of new formulas for certain transcendental numbers, such as $\pi$ (the so-called Bailey-Borwein-Plouffe formula) and $\log(2)$, through which $n^{th}$ hexadecimal or binary digit of the number can be computed in a reasonable time, without computing any of the previous $n-1$ digits \cite{bailey1997rapid}. Yet another application of this algorithm is to determine whether a given number $\alpha$ is algebraic of some degree $n$ or less; namely to determine if there exists a polynomial of degree at most $n$, with integer coefficients, admitting $\alpha$ as one of its roots. Such a polynomial exists iff there exists an integer relation for the vector, $(1,\alpha,\alpha^2,\dots,\alpha^n)$. This algorithm can be successfully used to recover such a polynomial, if it exists; or to identify a bound, such that, there is no such polynomial with integer coefficients, whose Euclidean norm is less than this bound \cite{bailey2001parallel}.


The results of the present paper reveal an intriguing algorithmic connection between the high-dimensional regression problem, and the integer-relation detection, subset-sum and consequently, the $\gamma$-approximate shortest vector problems. 

\subsection{Summary of the Results}
We now summarize our main contributions.
\begin{itemize}
    \item[(a)] We first study the high-dimensional linear regression problem, where the algorithmic goal is to efficiently recover an unknown feature vector $\beta^*\in\R^p$, from its noisy linear measurements $Y=X\beta^*+W\in\R^n$, under the assumptions that entries of $\beta^*$ enjoy the so-called $Q-$rationality assumption, that is, entries of $\beta^*$ are rational numbers with a common denominator $Q\in\mathbb{Z}^+$;  $\|\beta^*\|_\infty\leq\widehat{R}$ for some $\widehat{R}>0$, $X$ has bounded density and finite expectation, and $W$ is either adversarial with $\|W\|_\infty\leq \sigma$ or iid mean-zero noise with variance $\sigma^2$. We propose an efficient algorithm, which, under some explicit parameter assumptions, provably recovers $\beta^*\in\R^p$ in time polynomial in $n,p,\sigma,\widehat{R}$, and $Q$. Moreover,  our analysis reveals that the single-sample recovery $(n=1)$ is indeed possible, provided that the parameters satisfy an explicitly stated condition. 
    We complement our analysis with the information-theoretic limits of the problem, and establish for large $Q$ and normal noise, we have near optimal noise tolerance, using results on the capacity of Gaussian channel with power constraint. A crucial step towards our main result is the extension of the Lagarias-Odlyzko algorithm \cite{lagarias1985solving}, \cite{FriezeSubset} to not necessarily binary, integer vectors $\beta^* \in \mathbb{Z}^p$, for measurement matrix $X \in \mathbb{Z}^{ n \times p}$ with iid entries not necessarily from the uniform distribution, and finally, for non-zero noise vector $W$. As in \cite{lagarias1985solving} and \cite{FriezeSubset}, the algorithm we construct depends crucially on building an appropriate lattice and applying the LLL algorithm on it. However, unlike the previous case, where the recovered vector is a multiple $\lambda \beta^*$ of a binary vector in which the corresponding multiplicity can be read off directly, we need an extra additional step:  We translate the observations $Y=X\beta^*+Z$ by $XZ$, then establish, using an analytic number theory argument, that the entries of $\beta^*+Z$ has greatest common divisor equal to one, with high probability. This argument extends an elegant result from probabilistic number theory (see for example, Theorem 332 in \cite{Hardy}) according to which $\lim_{m \rightarrow +\infty} \mathbb{P}_{P,Q \sim \mathrm{Unif}\{1,2,\ldots,m\}, P \independent  Q}\left[ \mathrm{gcd}\left(P,Q\right)=1\right]= \frac{6}{\pi^2},$ where  $P \independent  Q$ refers to $P,Q$ being independent random variables.  A key implication of this result for us is the fact that the limit above is strictly positive.  
    
    \item[(b)] Our next focus is the noiseless high-dimensional linear regression problem, where the algorithmic goal of the learner is to efficiently recover an unknown feature vector $\beta^*\in \R^p$, this time consisting of both rational as well as irrational entries; from its noiseless linear measurements, $Y=X\beta^*\in\R^n$. We study this problem for a $\beta^*\in\R^p$, whose mixed irrational-rational support enjoys the so-called mixed-support assumption, an assumption that has been concretely defined as the Assumption \ref{def:mixed_support} in the main body of this paper. We propose an efficient algorithm, which, under some explicitly stated parameter assumptions, provably recovers a $\beta^*\in\R^p$ enjoying the mixed-support assumption from its noiseless linear measurements $Y=X\beta^*\in\R^n$, in time polynomial in the problem parameters. More formally, we show that if the iid entries of $X\in\mathbb{R}^{n\times p}$ are drawn from a continuous distribution with bounded density and finite expected value (or from a discrete distribution with a large enough support), the entries of $\beta^*$ are supported on a set $\mathcal{S}$, where the irrational elements $\{a_1,\dots,a_{\mathcal{R}}\}$ of $\mathcal{S}$ are rationally independent and known to the learner; and the rational elements share a common denominator $Q\in\mathbb{Z}^+$; the proposed algorithm recovers a $\beta^*$ with high probability, after a number of operations, polynomial in the number $n$ of measurements, the dimension $p$ of the feature vector $\beta^*$, the number $\mathcal{R}$ of the irrational elements of the support $\mathcal{S}$, and in $\log Q$, where $Q\in\mathbb{Z}^+$ is the common denominator of the rational entries of the support $\mathcal{S}$ of $\beta^*$. The algorithm that we propose is obtained using a novel combination of the PSLQ integer relation detection \cite{ferguson1999analysis}, and LLL lattice basis reduction \cite{lenstra1982factoring} algorithms. Moreover, analogous to the previous setting, our analysis reveals also that the efficient recovery is indeed  possible, even when the learner has access to only one measurement ($n=1$); provided that the parameters satisfy an explicitly stated condition. In particular, our algorithms for the high-dimensional linear regression  problem address the extreme regime when the learner has access to only one measurement $(n=1)$, a regime where the sparsity-based methods, such as LASSO and Basis Pursuit are known to fail. In particular, in the aforementioned setting, our analysis reveals that there is no {\em statistical-computational gap}.

\item[(c)] We then apply our ideas to the phase retrieval problem, where a complex-valued vector $\beta^*$ is observed through the model, $Y=|\langle X,\beta^*\rangle|$. We study this problem for both discrete-valued and continuous-valued measurement matrix $X$ with random entries; and complex-valued $\beta^*$, whose entries are supported on a set known to the learner, with cardinality $\mathcal{R}$. We establish that, under a rational independence assumption on a certain predetermined set, generated from the support of $\beta^*$, one can recover $\beta^*$ after polynomial in $p$ and $\mathcal{R}$ many arithmetic operations on real numbers, with high probability. Moreover, the proposed method transfers almost immediately, to more generalized observation models, where a complex-valued $\beta^*$ is observed through $Y=f(|\langle X,\beta^*\rangle|)$, provided that $f$ belongs to a class of functions for which, for any real number $r$, the roots of the equation, $f(x)=r$ can be computed, in a sense to discussed. An example of such an $f$ is a polynomial with degree at most $4$, where there exists formulas for the roots of such polynomials. Towards the goal of devising an efficient algorithm to the phase retrieval problem; we show that the LLL algorithm run on an appropriate lattice yields an efficient algorithm for the random subset-sum problem with dependent inputs, therefore, extending Lagarias-Odlyzko algorithm \cite{lagarias1985solving} and generalizing Frieze's result \cite{FriezeSubset} in a different direction. In particular, we show that if $X_1,\dots,X_p$ are iid random variables, taking values from a large enough discrete support; and $Y=\sum_{i<j}X_iX_j \xi_{ij}$ with $\xi_{ij}\in\{0,1\}$, there exists an algorithm, which admits $Y,X_1,\dots,X_p$ as its inputs and recovers the binary variables $\xi_{ij}$ with high probability as $p\to \infty$.

\item[(d)] Our analysis, in addition to addressing high-dimensional linear regression and phase retrieval problems, also reveals an algorithmic connection between these problems, and the well-known, yet not directly related to the inference tasks, integer relation detection, randomized subset-sum, and approximate short vector problems. It is also intriguing that all of these three problems are related to an implicit lattice structure; and admit an efficient, lattice-based algorithm\footnote{Recall that the earlier HJLS algorithm introduced by Hastad et al. \cite{hastad1989polynomial} for solving integer relation algorithm is based on LLL lattice basis reduction algorithm.}.

 \end{itemize}

\subsection{Notation and  Definitions} 
\subsubsection{Notation}
Let $\Z$ denote the set of integers, and $\R$ denote the set of real numbers. Let $\mathbb{Z}^*$ denote $\mathbb{Z} \setminus\{0\}$. We use $\Z^+$ and $\mathbb{R}^+$ for the set of positive integers, and for the set of positive real numbers, respectively. For $k \in \mathbb{Z}^+$ we set $[k]:=\{1,2,\ldots,k\}$. For a vector $x \in \mathbb{R}^d$ we define $\mathrm{Diag}_{d \times d}\left(x\right) \in \mathbb{R}^{d \times d}$ to be the diagonal matrix with $\mathrm{Diag}_{d \times d}\left(x\right)_{ii}=x_i, \text{ for } i \in [d].$ For $1\leq p < \infty$ by $\mathcal{L}_p$ we refer to the standard $p$-norm notation for finite dimensional real vectors. Given $x\in\mathbb{R}^d$, $\|x\|$ denotes the Euclidean norm $(\sum_{i=1}^d |x_i|^2)^{1/2}$ of $x$, and $\|x\|_{\infty}$ denotes $\max_{1\leq i\leq d}|x_i|$. Given two vectors $x,y \in \mathbb{R}^d$ the Euclidean inner product is  $\inner{x}{y}:=\sum_{i=1}^d x_iy_i$. By  $\log: \mathbb{R}^+ \rightarrow \mathbb{R}$ we refer the logarithm with base 2. The integer lattice generated by a set of linearly independent $b_1,\ldots,b_k \in \mathbb{Z}^k$  is defined as $\{ \sum_{i=1}^k z_i b_i | z_1,z_2,\ldots,z_k \in \mathbb{Z}\}$, and will be denoted by $\Lambda$, where $\Lambda,\subseteq \mathbb{Z}^k$. The collection $b_1,\dots,b_k$ is called the lattice base. 
Given any real number $r$, $\fl{r}$ denotes the largest integer, which does not exceed $r$, and $\{r\}=r-\fl{r}$ denotes the fractional part of $r$.

We denote by $x^H$ the complex conjugate of an $x\in\mathbb{C}$.  Throughout the paper we use the standard asymptotic notation, $o,O,\Theta,\Omega$ for comparing the growth of two real-valued sequences $(a_n)_{n=1}^\infty$ and $(b_n)_{n=1}^\infty$. Finally, we say that a sequence of events $\{A_p\}_{p \in \mathbb{N}}$ holds with high probability (whp) as $p \rightarrow +\infty$ if $\lim_{p \rightarrow + \infty} \mathbb{P}\left(A_p\right)=1.$
\subsubsection{Definitions}
\begin{definition}
\label{def:rat_indep}
A set $S=\{a_1,\dots,a_{\mathcal{R}}\}\subset \mathbb{R}$ is called {\bf linearly independent over rationals} or in short {\bf rationally independent}, if for every $q_1,\dots,q_{\mathcal{R}}\in\Q$, $\sum_{k=1}^{\mathcal{R}} q_ka_k = 0$ implies, $q_k=0$, for every $k\in[\mathcal{R}]$. 
\end{definition}
As an example, the set, $S=\{\sqrt{2},\sqrt{3}\}$ is rationally independent, while the set, $S'=\{\sqrt{2},\sqrt{3}-\sqrt{2},\sqrt{3}\}$ is not. 
Clearly, if $S$ is rationally independent, so does any subset $S'\subset S$; and linear independence over rationals is equivalent to the linear independence over integers, which states that, for $q_1,\dots,q_{\mathcal{R}}\in\mathbb{Z}$ if $\sum_{k=1}^{\mathcal{R}} q_ia_i=0$, then $q_i=0$ for all $i$.  

\begin{definition}
\label{def:Q-rational}
Let $p,Q \in \mathbb{Z}^+$. We say that a vector $\beta^* \in \mathbb{R}^p$ satisfies the $Q$\textbf{-rationality assumption} if for all $i \in [p]$, $\beta^*_i=K_i/Q$, for some $K_i \in \mathbb{Z}$. Here, we do not necessarily assume that $K_i$ and $Q$ are coprime. 
\end{definition}
Our mixed-support assumption, on the entries of $\beta^*$ is as follows.
\begin{assumptions}
\label{def:mixed_support}
Let $Q,\mathcal{R},\widetilde{R}\in\mathbb{Z}^+$, and let $\mathcal{S}_0=\{a_1,\dots,a_{\mathcal{R}}\}
\subset \R$ be such that $\mathcal{S}_0\cup \{1\}$ is rationally independent. A vector $\beta^*\in \mathbb{R}^p$ satisfies the {\bf mixed-support assumption} if for all $i\in[p]$, 
\begin{itemize}
    \item[(i)] Either $\beta_i^* = K_i/Q$ for some $K_i\in \mathbb{Z}$, and $\beta_i^*\in [-\widetilde{R},\widetilde{R}]$ (namely, $\beta_i^*$ enjoys the aforementioned $Q$-rationality assumption).
    \item[(ii)] Or, $\beta_i^*\in \mathcal{S}_0=\{a_1,\dots,a_{\mathcal{R}}\}$.
\end{itemize}
\end{assumptions}
Note that, under mixed-support assumption, the numerator of the rational entries of $\beta^*$ are upper bounded by $\widetilde{R}$ in magnitude. The rational independence of $\{a_1,\dots,a_{\mathcal{R}},1\}$ is a slightly stronger assumption than $\{a_1,\dots,a_{\mathcal{R}}\}$ (to see this, note for instance that $\{2-\sqrt{2},\sqrt{2}\}$ is rationally independent, whereas $\{2-\sqrt{2},\sqrt{2},1\}$ is not rationally independent), and the reason for this will become clear, once we provide the details of the algorithm. Note that, rational independence assumption above implies that $a_i$'s are all irrational. 

Our results will reveal an algorithmic connection between the regression problem, and three well-known problems:
\begin{definition}
An instance of the {\bf integer relation detection problem} is as follows. Given a vector ${\bf b}=(b_1,\dots,b_n)\in\mathbb{R}^n$, find an ${\bf m}\in\mathbb{Z}^n\setminus\{{\bf 0}\}$, such that $\langle {\bf b},{\bf m}\rangle=\sum_{i=1}^n b_im_i=0$. In this case, ${\bf m}$ is said to be an {\bf integer relation} for the vector, ${\bf b}$.
\end{definition}
The problem of finding an integer relation for a given ${\bf b}\in\R^n$, as mentioned in the introduction, has been studied since Euclid, and algorithms are available; of which the two most well-known are the celebrated HJLS \cite{hastad1989polynomial} and PSLQ \cite{ferguson1999analysis} algorithms. For any given ${\bf b}\in\R^n$, the PSLQ algorithm with input ${\bf b}$ finds an integer relation ${\bf m'}\in \mathbb{Z}^n$ for ${\bf b}$ in time polynomial in $n$ and $\log \|{\bf m}\|$, where ${\bf m}$ is a relation for ${\bf b}$ with smallest $\|{\bf m}\|$ \cite{ferguson1999analysis}. More concretely, we have the following theorem:
\begin{theorem}\label{thm:pslq_main_thm}{\rm \cite{ferguson1999analysis}}
Let ${\bf m}\in\mathbb{Z}^n\setminus \{{\bf 0}\}$ be an integer relation for ${\bf b}\in\mathbb{R}^n$ with the smallest norm $\|{\bf m}\|$. Then, PSLQ algorithm with input ${\bf b}$ outputs an integer relation for ${\bf b}$ after at most $O(n^3+n^2\log\|{\bf m}\|)$ arithmetic operations on real numbers.
\end{theorem}
Note that, the output of the algorithm is not necessarily ${\bf m}$, but can be any other integer relation for ${\bf b}$.  In the sequel, we assume that this algorithm is at our disposal, and we simply refer to it as Integer Relation Algorithm (IRA).


The following two algorithmic problems are extensively studied in theoretical computer science and cryptography.

\begin{definition}
An instance of the {\bf subset-sum problem} is as follows. Suppose $Y,X_1,\dots,X_n$ are integers, such that $Y=\sum_{i\in S^*}X_i$ for some $S^*\subset [n]$. Determine $S^*$ from the knowledge $(Y,X_1,\dots,X_n)$.  Equivalently, given integers $X_1,\dots,X_n$, and $Y=\sum_{i=1}^n X_ie_i$, where
 $e\in\{0,1\}^n$ is a hidden vector, find $e$.

\end{definition}

While this problem is NP-hard in worst case, assuming $X$ consists of iid samples of an appropriate distribution ({\em e.g.} uniform over a large enough set of integers) Lagarias and Odlyzko \cite{lagarias1985solving} and Frieze \cite{FriezeSubset} developed a polynomial-time algorithm,  for recovering the hidden subset $S^*$. Their algortihm is based on the seminal Lenstra-Lenstra-Lov\'asz (LLL) lattice basis reduction algorithm \cite{lenstra1982factoring}. 

We now introduce our last definition, which is the  computational problem of finding short vectors of an integer lattice.


\begin{definition}
\label{def:gamma-approx-short}
An instance of {\bf $\gamma$-approximate shortest vector problem} is as follows. Given a lattice base $b_1,\dots,b_p\in \Z^p$ for an integer lattice, $\Lambda \subseteq \mathbb{Z}^p$, and  a constant $\gamma>0$; find a vector $\widehat{x}\in \Lambda$, such that
$$
\|\widehat{x}\|\leq \gamma \min_{x\in \Lambda, x\neq 0}\|x\|.
$$
\end{definition}
\noindent LLL lattice basis reduction algorithm solves this problem for $\gamma=2^{\frac{p}{2}}$, in a time polynomial in $p$ and $\log\max_{1\leq i\leq p} \|b_i\|_\infty$.

Some of our results with irrational-valued supports will operate under the so-called joint continuity assumption.
\begin{definition}\label{def:joint-cont}
A random vector $X\in \R^p$ is called {\bf jointly continuous}, if there exists a measurable function $f:\mathbb{R}^p\to [0,\infty)$, called the joint density of $X$, such that for every Borel set $\mathcal{B}\subseteq\mathbb{R}^p$,
$$
\mathbb{P}(X\in \mathcal{B})=\int_{\mathcal{B}}f(x_1,\dots,x_p)\; d\lambda(x_1,\cdots  ,x_p).
$$
\end{definition}
Note that, joint continuity is a more general notion than being iid (and even independent): If $X\in\R^p$ is a random vector with independent coordinates $X_i$, having density $\varphi_i(\cdot)$, then $X$ is also jointly continuous with $f(x_1,\dots,x_p)=\prod_{\ell=1}^p\varphi_i(x_\ell)$.

\subsubsection{Computational Model}
Throughout the paper, the performance guarantees of our algorithms will be provided under different computational models, depending on the nature of the support $\mathcal{S}$ of $\beta^*$. More concretely,
\begin{itemize}
    \item[(a)] For $\beta^*\in\R^p$ with rational-only elements, we adopt the floating point computational model, and measure the complexity in terms of bit operations, where each arithmetic bit operation is of $O(1)$ cost. This applies to both integer-valued and continuous-valued measurement matrices $X\in \R^{n\times p}$, where for the latter, one has to truncate the measurement matrix $X$, as well as the measurement $Y$ itself, since the formal algorithmic inputs under this measurement model must consist of finitely many bits.
    \item[(b)] For $\beta^*\in \R^p$ with irrational-only elements, we adopt the real-valued computational model, where we assume the existence of a computational engine operating over real-valued inputs. Each arithmetic operation on real inputs is assumed to be of $O(1)$ cost. An example of such a computational model is the so-called Blum-Shub-Smale machine \cite{blum2012complexity,blum1988theory} operating over real-valued inputs. This model allows one to focus more transparently on the complexity of internal steps of the algorithm, without worrying about the length of the input bit stream provided to the oracle.
    \item[(c)] The algorithm for recovering $\beta^*\in \R^p$ with a mixed-support, namely, a support consisting of both rational and irrational entries, will operate under a computational model, which can be perceived as a mixture of aforementioned two models. More concretely, any such algorithm proceed first by recovering its irrational entries of $\beta^*$, then  continue with recovering its rational entries. The irrational entries of $\beta^*$ are to be recovered using a computational engine operating over real-valued inputs. This will be followed by recovering the rational entries of $\beta^*$, in which case we switch the model of computation to floating-point arithmetic, as in $(a)$ above.
    \end{itemize}
Equipped with this, whenever we say that an algorithm runs in polynomial time (in the relevant parameters of the problem), we mean that the algorithm runs in polynomial time under the adopted computational model, which will be transparent from the context. In particular, the polynomial runtime under the floating-point arithmetic model requires that the input bitstream supplied to the algorithm has a length that is at most polynomial in the parameters of the problem; whereas there is no such requirement under real-valued computational model, as the computational engine is assumed to operate over real-valued inputs, that are potentially of infinitely many bits. 

\section{Main Results}
\subsection{Noisy High-Dimensional Linear Regression with $Q-$rational $\beta^*$}

\subsubsection{Extended Lagarias-Odlyzko algorithm}\label{sec:ELO-here}

Let $n,p,R \in \mathbb{Z}^+$. Given $X \in \mathbb{Z}^{n \times p},\beta^* \in \left(\mathbb{Z} \cap [-R,R]\right)^p$ and $ W \in \mathbb{Z}^n$, set $Y=X\beta^*+W$. From the knowledge of $Y,X$ the goal is to infer exactly $\beta^*$. For this task we propose the following algorithm which is an extension of the algorithm in \cite{lagarias1985solving} and \cite{FriezeSubset}. For realistic purposes the values of $R,\|W\|_{\infty}$ is not assumed to be known exactly. As a result, the following algorithm, besides $Y,X$, receives as an input a number $\hat{R} \in \mathbb{Z}^+$ which is an estimated upper bound in absolute value for the entries of $\beta^*$ and a number $\hat{W} \in \mathbb{Z}^+$ which is an estimated upper bound in absolute value for the entries of $W$.
\begin{algorithm}
\caption{Extended Lagarias-Odlyzko (ELO) Algorithm}
\label{algo:ELO}
\LinesNumbered
\KwIn{$(Y,X,\hat{R},\hat{W})$, $Y \in \mathbb{Z}^n, X \in \mathbb{Z}^{n \times p}$, $\hat{R}, \hat{W} \in \mathbb{Z}^+$.}
\KwOut{$\hat{\beta^*}$ an estimate of $\beta^*$} 
 Generate a random vector $Z \in \{\hat{R}+1,\hat{R}+2,\ldots,2\hat{R}+\log p\}^p$ with iid entries uniform in $\{\hat{R}+1,\hat{R}+2,\ldots,2\hat{R}+\log p\}$\\
\nl Set $Y_1=Y+XZ$.\\
\nl For each $i=1,2,\ldots,n$, if $|(Y_1)_i|<3$ set $(Y_2)_i=3$ and otherwise set $(Y_2)_i=(Y_1)_i$.\\
\nl Set $m=2^{n+\ceil{\frac{p}{2}}+3}p\left(\hat{R}\ceil{\sqrt{p}}+\hat{W}\ceil{\sqrt{n}}\right)$. \\
\nl Output $\hat{z} \in \mathbb{R}^{2n+p}$ from running the LLL basis reduction algorithm on the lattice generated by the columns of the following $(2n+p)\times (2n+p)$ integer-valued matrix,\begin{equation}
 A_m:= \left[ {\begin{array}{ccc}
    m X  & -m\mathrm{Diag}_{n \times n}\left(Y_2\right) & m I_{n \times n} \\
     I_{p \times p} & 0_{p \times n} & 0_{p \times n}\\
    0_{n \times p} & 0_{n \times n} & I_{n \times n}
  \end{array} } \right]
  \end{equation}\\
\nl Compute $g=\mathrm{gcd}\left(\hat{z}_{n+1},\hat{z}_{n+2},\ldots,\hat{z}_{n+p}\right),$ using the Euclid's algorithm.\\
\nl If $g \not = 0$, output $\hat{\beta^*}=\frac{1}{g}(\hat{z}_{n+1},\hat{z}_{n+2},\ldots,\hat{z}_{n+p})^t-Z.$ Otherwise, output $\hat{\beta^*}=0_{p \times 1}$.
\end{algorithm}

We explain here informally the steps of the (ELO) algorithm and briefly sketch the motivation behind each one of them. In the first and second steps the algorithm translates $Y$ by $XZ$ where $Z$ is a random vector with iid elements chosen uniformly from $\{\hat{R}+1,\hat{R}+2,\ldots,2\hat{R}+\log p\}$. In that way $\beta^*$ is translated implicitly to $\beta=\beta^*+Z$ because $Y_1=Y+XZ=X(\beta^*+Z)+W$. We will establish using a number theoretic argument, that $\mathrm{gcd}\left(\beta\right)=1$ whp as $p \rightarrow +\infty$ with respect to the randomness of $Z$, even though this is not necessarily the case for the original $\beta^*$. This is an essential requirement for our technique to exactly recover $\beta^*$ and steps six and seven to be meaningful. In the third step the algorithm gets rid of the significantly small observations. This minor but necessary modification affects the observations in a negligible way.

The fourth and fifth steps of the algorithm provide a basis for a specific lattice in $2n+p$ dimensions. The lattice is built with the knowledge of the input and $Y_2$, the modified version of $Y_1$. The algorithm in step five calls the LLL basis reduction algorithm using the columns of $A_m$ as initial basis for the lattice. The fact that $Y$ has been modified to be non-zero on every coordinate is essential here so that $A_m$ is full-rank and the LLL basis reduction algorithm, defined in \cite{lenstra1982factoring}, can be applied. This application of the LLL basis reduction algorithm is similar to the one used in \cite{FriezeSubset} with one important modification. In order to deal here with multiple equations and non-zero noise, we use $2n+p$ dimensions instead of $1+p$ in \cite{FriezeSubset}. Following though a similar strategy as in \cite{FriezeSubset}, it can be established that the $n+1$ to $n+p$ coordinates of the output of the algorithm, $\hat{z} \in \mathbb{Z}^{2n+p}$, correspond to a vector which is a non-zero integer multiple of $\beta$, say $\lambda \beta$ for $\lambda \in \mathbb{Z}^*$, w.h.p. as $p \rightarrow +\infty$.

The proof of the above result is an important part in the analysis of the algorithm and it is heavily based on the fact that the matrix $A_m$, which generates the lattice, has its first $n$ rows multiplied by the ``large enough" and appropriately chosen integer $m$ which is defined in step four. It can be shown that this property of $A_m$ implies that any vector $z$ in the lattice with ``small enough" $\mathcal{L}_2$ norm necessarily satisfies $\left(z_{n+1},z_{n+2},\ldots,z_{n+p}\right)=\lambda \beta$ for some $\lambda \in \mathbb{Z}^*$ whp as $p \rightarrow +\infty$. In particular, using that $\hat{z}$ is guaranteed to satisfy $\|\hat{z}\|_2 \leq 2^{\frac{2n+p}{2}} \|z\|_2$ for all non-zero $z$ in the lattice, it can be derived that $\hat{z}$ has a ``small enough" $\mathcal{L}_2$ norm and therefore indeed satisfies the desired property whp as $p \rightarrow +\infty$.  Assuming now the validity of the $\mathrm{gcd}\left(\beta\right)=1$ property, step six finds in polynomial time this unknown integer $\lambda$ that corresponds to $\hat{z}$, because $\mathrm{gcd}\left(\hat{z}_{n+1},\hat{z}_{n+2},\ldots,\hat{z}_{n+p}\right)=\mathrm{gcd}\left(\lambda \beta\right)=\lambda$. Finally step seven scales out $\lambda$ from every coordinate and then subtracts the known random vector $Z$, to output exactly $\beta^*$.  

Of course the above is based on an informal reasoning. Formally we establish the following result.
\begin{theorem}\label{extention}
Suppose
\begin{itemize}
\item[(1)] $X \in \mathbb{Z}^{n \times p}$ is a matrix with iid entries generated according to a distribution $\mathcal{D}$ on $\mathbb{Z}$ which for some $N \in \mathbb{Z}^+$ and constants $C,c>0$, assigns at most $\frac{c}{2^N}$ probability on each element of $\mathbb{Z}$ and satisfies $\mathbb{E}[|V|] \leq C2^N$, for $ V\overset{d}{=}\mathcal{D} $;
\item[(2)]  $\beta^* \in \left(\mathbb{Z} \cap [-R,R] \right)^p$, $W \in \mathbb{Z}^n$;
\item[(3)] $Y=X\beta^*+W$.
\end{itemize}  
Suppose furthermore that $  \hat{R} \geq R$ and 
\begin{equation}\label{eq:limit}
N \geq \frac{1}{2n}(2n+p)\left[2n+p+10\log\left( \hat{R} \sqrt{p}+\left(\|W\|_{\infty}+1\right) \sqrt{n}\right) \right]+6 \log \left(\left(1+c\right)np \right).
\end{equation} For any $\hat{W} \geq \|W\|_{\infty}$  the algorithm ELO with input $(Y,X,\hat{R},\hat{W})$  outputs \textbf{exactly} $\beta^*$ w.p. $1-O \left(\frac{1}{np} \right)$ (whp as $p \rightarrow +\infty$) and terminates in time at most polynomial in $n,p,N, \log \hat{R}$ and $\log \hat{W}$.
\end{theorem}
The constants $C$ and $c$ are hidden under $O(1/np)$. We defer the proof to Section \ref{ProofExt}.
\begin{remark}
In the statement of Theorem \ref{extention} the only parameters that are assumed to grow to infinity are $p$ and whichever other parameters among $n,R,\|W\|_{\infty},N$ are implied to grow to infinity because of (\ref{eq:limit}). Note in particular that $n$ can remain bounded, including the case $n=1$, if $N$ grows fast enough. Similarly,  $\widehat{R}$, and $\|W\|_\infty$ may remain bounded.
\end{remark}

\begin{remark}
It can be easily checked that the assumptions of Theorem \ref{extention} are satisfied for $n=1$, $N=(1+\epsilon)\frac{p^2}{2}$, $R=1$, $\mathcal{D}=\mathrm{Unif}\{1,2,3,\ldots,2^{(1+\epsilon)\frac{p^2}{2}}\}$ and $W=0$. Under these assumptions, the Theorem's implication is a generalization of the result from \cite{lagarias1985solving} and \cite{FriezeSubset} to the case $\beta^* \in \{-1,0,1\}^p$.
\end{remark}

\subsubsection{Applications to High-Dimensional Linear Regression}
\subsubsection{The Model}
 The high-dimensional linear regression model we are considering is as follows.
\begin{assumptions}\label{ass}Let $n,p,Q \in \mathbb{Z}^+$ and $R,\sigma,c>0$. Suppose
\begin{itemize}
\item[(1)] measurement matrix $X \in \mathbb{R}^{n \times p}$ with iid entries generated according to a continuous distribution $\mathcal{C}$ which has density $f$ with $\|f\|_{\infty} \leq c$ and satisfies $\mathbb{E}[|V|]<+\infty$, where $ V\overset{d}{=}\mathcal{C} $;
\item[(2)] ground truth vector $\beta^*$ satisfies $\beta^* \in [-R,R]^p$ and the $Q$-rationality assumption;
\item[(3)] $Y=X\beta^*+W$ for some noise vector $W \in \mathbb{R}^n$. It is assumed that either $\|W\|_{\infty} \leq \sigma$ or $W$ has iid entries with mean zero and variance at most $\sigma^2$, depending on the context.
\end{itemize}
\end{assumptions}\textbf{Objective:} Based on the knowledge of $Y$ and $X$ the goal is to efficiently recover $\beta^*$.

\subsubsection{The Lattice-Based Regression (LBR) Algorithm}
As mentioned in the Introduction, we propose an algorithm to solve the regression problem, which we call the Lattice-Based Regression (LBR) algorithm. The exact knowledge of $Q,R,\|W\|_{\infty}$ is not assumed. Instead the algorithm receives as an input, additional to $Y$ and $X$, $\hat{Q} \in \mathbb{Z}^+$ which is an estimated multiple of $Q$, $\hat{R} \in \mathbb{Z}^+$ which is an estimated upper bound in absolute value for the entries of $\beta^*$ and  $\hat{W} \in \mathbb{R}^+$ which is an estimated upper bound on $\|W\|_\infty$. Furthermore an integer number $N \in \mathbb{Z}^+$ is given to the algorithm as an input, which, as we will explain, corresponds to a truncation in the data in the first step of the algorithm. Given $x \in \mathbb{R}$ and $N \in \mathbb{Z}^+$ let $x_N=\mathrm{sign}(x)\frac{\floor{2^N|x|}}{2^N}$, which corresponds to the operation of keeping the first $N$ bits after zero of a real number $x$.

\begin{algorithm}
\caption{Lattice Based Regression (LBR) Algorithm}
\label{algo:LBR}
\setcounter{AlgoLine}{1}
\KwIn{$(Y,X,N,\hat{Q},\hat{R},\hat{W})$, $Y \in \mathbb{R}^n, X \in \mathbb{R}^{n \times p}$ and $N ,\hat{Q},\hat{R},\hat{W} \in \mathbb{Z}^+$.}
\KwOut{$\hat{\beta^*}$ an estimate of $\beta^*$}
\nl Set  $Y_N=\left( (Y_{i})_N\right)_{i \in [n]}$ and $X_N=\left( (X_{ij})_N\right)_{i \in [n],j \in [p]}$.\\
\nl Set $(\hat{\beta_1})^*$ to be the output of the ELO algorithm with input: \begin{equation*} \left(2^{N}\hat{Q}Y_N,2^{N}X_{N},\hat{Q}\hat{R},2\hat{Q}\left(2^N\hat{W}+\hat{R}p\right)\right). \end{equation*}\\
\nl Output $\hat{\beta^*}=\frac{1}{\hat{Q}}(\hat{\beta_1})^*$.
\end{algorithm}

We now explain informally the steps of the Algorithm 2 (LBR) below. In the first step, the algorithm truncates each entry of $Y$ and $X$ by keeping only its first $N$ bits after zero, for some $N \in \mathbb{Z}^+$. This in particular allows to perform finite-precision operations and to call the ELO algorithm in the next step which is designed for integer input. In the second step, the algorithm naturally scales up the truncated data to integer values, that is it scales $Y_N$ by $2^N\hat{Q}$ and $X_N$ by $2^N$. The reason for the additional multiplication of the observation vector $Y$ by $\hat{Q}$ is necessary to make sure the ground truth vector $\beta^*$ can be treated as integer-valued. To see this notice that $Y=X\beta^*+W$ and $Y_N,X_N$ being ``close" to $Y,X$ imply $$2^N\hat{Q}Y_N=2^NX_N(\hat{Q}\beta^*)+\text{ ``extra noise terms"}+2^N\hat{Q}W.$$ Therefore, assuming the control of the magnitude of the extra noise terms, by using the $Q$-rationality assumption and that $\hat{Q}$ is estimated to be a multiple of $Q$,  the new ground truth vector becomes $\hat{Q}\beta^*$ which is integer-valued. The final step of the algorithm consist of rescaling now the output of Step 2, to an output which is estimated to be the original $\beta^*$. In the next subsection, we turn this discussion into a provable recovery guarantee.
\subsubsection{Recovery Guarantees for the LBR algorithm}
We state now our first main result, explicitly stating the assumptions on the parameters, under which the LBR algorithm recovers \textbf{exactly} $\beta^*$ from bounded but \textbf{adversarial noise} $W$.
\begin{subtheorem}{theorem}
\begin{theorem}\label{main}
Under Assumption \ref{ass} and assuming $\|W\|_\infty\leq \sigma$ for some $\sigma \geq 0$, the following holds. Suppose  $\hat{Q}$ is a multiple of $Q$, $\hat{R} \geq R$ and
\begin{equation}\label{eq:limit1}
N>\frac{1}{2}\left(2n+p\right)\left(2n+p+10\log \hat{Q}+ 10\log \left(2^{N}\sigma +\hat{R}p \right) +20 \log (3\left(1+c\right)np) \right).
\end{equation}
For any $\hat{W} \geq \sigma$,  the LBR algorithm with input $(Y,X,N,\hat{Q},\hat{R},\hat{W})$ terminates with $\hat{\beta^*}=\beta^*$ w.p. $1-O\left(\frac{1}{np}\right)$ (whp as $p \rightarrow + \infty$) and in time polynomial in $n,p,N,\log \hat{R},\log \hat{W}$ and $\log \hat{Q}$.
\end{theorem}

 Applying  Theorem \ref{main} we establish the following result handling \textbf{random noise} $W$.
\begin{theorem}\label{mainiid}
Under Assumption \ref{ass} and assuming $W \in \mathbb{R}^n$ is a vector with iid entries generating according to an, independent from $X$, distribution $\mathcal{W}$ on $\mathbb{R}$ with mean zero and variance at most $\sigma^2$ for some $\sigma \geq 0$ the following holds. Suppose that $\hat{Q}$ is a multiple of $Q$, $\hat{R} \geq R$,  and
\begin{equation}\label{eq:limit2}
N>\frac{1}{2}\left(2n+p\right)\left(2n+p+10\log \hat{Q}+ 10\log \left(2^{N}\sqrt{np}\sigma +\hat{R}p \right)  +20 \log (3\left(1+c\right)np) \right).
\end{equation}
For any $\hat{W} \geq \sqrt{np}\sigma$  the LBR algorithm with input $(Y,X,N,\hat{Q},\hat{R},\hat{W})$ terminates with $\hat{\beta^*}=\beta^*$ w.p. $1-O\left(\frac{1}{np}\right)$ (whp as $p \rightarrow + \infty$) and in time polynomial in $n,p,N,\log \hat{R},\log \hat{W}$ and $\log \hat{Q}$.
\end{theorem}
\end{subtheorem}

Proofs of Theorems \ref{main} and \ref{mainiid} are deferred to Section \ref{ProofMain}.

\subsubsection{Noise tolerance of the LBR algorithm}

The assumptions (\ref{eq:limit}) and (\ref{eq:limit2}) might make it hard to build an intuition for the choice of the truncation level $N$. For this reason, in this subsection we simplify it and state a Proposition explicitly mentioning the optimal truncation level and hence characterizing the optimal level of noise that the LBR algorithm can tolerate with $n$ samples.

First note that in the statements of Theorem \ref{main} and Theorem \ref{mainiid} the only parameters that are assumed to grow are $p$ and, as  an implication $N$, due to (\ref{eq:limit}) and (\ref{eq:limit2}). Therefore, importantly, $n$ does not necessarily grow to infinity. That means that Theorem \ref{main} and Theorem \ref{mainiid} imply non-trivial guarantees for \textit{arbitrary sample size} $n$. The proposition below shows that if $\sigma$ is at most exponential in $-(1+\epsilon)\left[\frac{(p+2n)^2}{2n}+(2+\frac{p}{n}) \log \left(RQ\right)\right]$ for some $\epsilon>0$, then for appropriately chosen truncation level $N$ the LBR algorithm recovers exactly the vector $\beta^*$ with $n$ samples. In particular, with one sample ($n=1$) LBR algorithm tolerates noise level up to exponential in $-(1+\epsilon)\left[p^2/2+(2+p) \log (QR)\right]$ for some $\epsilon>0$. On the other hand, if $n=\Theta(p)$ and $\log \left(RQ\right)=o(p)$, the LBR algorithm tolerates noise level up to exponential in $-O(p)$.

\begin{proposition}\label{cor2}
Under Assumption \ref{ass} and assuming $W \in \mathbb{R}^n$ is a vector with iid entries generating according to an, independent from $X$, distribution $\mathcal{W}$ on $\mathbb{R}$ with mean zero and variance at most $\sigma^2$ for some $\sigma \geq 0$, the following holds.

Suppose for some $\epsilon>0$, $p \geq \frac{300}{\epsilon} \log \left(\frac{300}{(1+c) \epsilon} \right)$, and  $\sigma \leq 2^{-(1+\epsilon)\left[\frac{(p+2n)^2}{2n}+(2+\frac{p}{n}) \log \left(RQ\right)\right]}$. Then the LBR algorithm with \begin{itemize} 
\item input $Y,X$, $\hat{Q}=Q,$ $\hat{R}= R$ and $\hat{W}_{\infty}=1$ and
\item truncation level $N$ satisfying $ \log \left(\frac{1}{\sigma}\right) \geq N  \geq (1+\epsilon)\left[\frac{(p+2n)^2}{2n}+(2+\frac{p}{n}) \log \left(RQ\right)\right],$ 

\end{itemize} terminates with $\hat{\beta^*}=\beta^*$ w.p. $1-O\left(\frac{1}{np}\right)$ (whp as $p \rightarrow + \infty$) and in time polynomial in $n,p,N,\log \hat{R},\log \hat{W}$ and $\log \hat{Q}$.
\end{proposition}
The assumptions $\hat{Q}=Q$, $\hat{R}=R$, and $\hat{W}_{\infty}=1$ are imposed to make the noise level that can be tolerated by $n$ samples more transparent. The proof of Proposition \ref{cor2} is deferred to Section \ref{Rest}.

It is worth noticing that in the noisy case $(\sigma>0)$ the above Proposition requires the truncation level $N$ to be upper bounded by $\log (\frac{1}{\sigma})$, which implies the seemingly counter-intuitive conclusion that revealing more bits of the data after some point can ``hurt" the performance of the recovery mechanism. Note that this is actually justified because of the presence of adversarial noise of magnitute $\sigma$. In particular, handling an arbitrary noise of absolute value at most of the order $\sigma$ implies that the only bits of each observation that are certainly unaffected by the noise are the first $\log \left(\frac{1}{\sigma}\right)$ bits.  Any bit in a later position could have potentially changed because of the noise. This correct middle ground for the truncation level $N$ appears to be necessary also in the analysis of the synthetic experiments with the LBR algorithm; see Section 3 in  \cite{zadik2018high} for corresponding synthetic experiments.
\subsubsection{Information Theoretic Bounds} 

In this subsection, we discuss the maximum noise that can be tolerated information-theoretically in recovering a $\beta^* \in [-R,R]^p$ satisfying the $Q$-rationality assumption. We establish that under Gaussian white noise, any successful recovery mechanism can tolerate noise level at most exponentially small in $-\left[p \log \left(QR \right)/n\right]$.

\begin{proposition}\label{InfTh}
Suppose that $X \in \mathbb{R}^{n \times p}$ is a vector with iid entries following a continuous distribution $\mathcal{D}$ with $\mathbb{E}[|V|]<+\infty$, where $ V\overset{d}{=}\mathcal{D}$, $\beta^* \in [-R,R]^p$ satisfies the $Q$-rationality assumption,  $W \in \mathbb{R}^n$ has iid $N(0,\sigma^2)$ entries and $Y=X\beta^*+W$. Suppose furthermore that $\sigma > R(np)^3 \left(2^{\frac{2p \log \left(2QR+1\right)}{n}}-1\right)^{-\frac{1}{2}}$. Then there is \textbf{no} mechanism which, whp as $p \rightarrow + \infty$, recovers \textbf{exactly} $\beta^*$ with knowledge of $Y,X,Q,R,\sigma$. That is, for any function $\hat{\beta^*}=\hat{\beta^*}\left(Y,X,Q,R,\sigma\right)$ we have
\begin{equation*}
\limsup_{p \rightarrow + \infty} \mathbb{P}\left(\hat{\beta^*}=\beta^* \right)<1. 
\end{equation*}
\end{proposition}The proof of Proposition \ref{InfTh} is deferred to Section \ref{Rest}.

\subsubsection*{Sharp Optimality of the LBR Algorithm}
Using Propositions \ref{cor2} and \ref{InfTh} the following \textbf{sharp} result is established. 
\begin{proposition}\label{optimal}
Under Assumptions \ref{ass} where $W \in \mathbb{R}^n$ is a vector with iid $N(0,\sigma^2)$ entries the following holds. Suppose that $n=o\left(\frac{p}{\log p}\right)$ and $ RQ=2^{\omega(p)}$. Then for $\sigma_0:=2^{-\frac{p \log \left(RQ\right)}{n}}$ and $\epsilon>0$:
\begin{itemize}
\item if $ \sigma>\sigma_0^{1-\epsilon}$,then the w.h.p. exact recovery of $\beta^*$ from the knowledge of $Y,X,Q,R,\sigma$ is impossible.
\item if $\sigma<\sigma_0^{1+\epsilon}$, then the w.h.p. exact recovery of $\beta^*$ from the knowledge of $Y,X,Q,R,\sigma$ is possible by the LBR algorithm.
\end{itemize}
\end{proposition}

The proof of Proposition \ref{optimal} is deferred to Section \ref{Rest}.

\subsection{Noiseless High-Dimensional Linear Regression with Irrational $\beta^*$}
We now consider the noiseless setting, where the learner has access to the noiseless linear measurements $Y=X\beta^*\in\mathbb{R}^n$ of $\beta^*$. We establish that, under the mixed support assumption (Assumption \ref{def:mixed_support}) on the entries of $\beta^*$, efficient recovery is possible for a large class of distributions, whose iid entries constitute the measurement matrix $X$. Furthermore, we show that, the efficient recovery is possible, even when the learner has access to only one measurement ($n=1$), under explicitly stated conditions on the distributions of the entries of $X$. The algorithms that we propose are obtained by using a novel combination of the LLL lattice basis reduction algorithm \cite{lenstra1982factoring} employed in previous subsections, together with PSLQ integer relation  \cite{ferguson1999analysis} algorithm. 

To demonstrate our techniques, we first start in this section with the case where $\beta^*\in\R^p$ consists only of irrational entries, and the measurement matrix consists of either integer or real-valued random entries. The next section will address the recovery problem in the case when $\beta^*\in\R^p$ enjoys the mixed-support assumption.

\subsubsection{Integer-Valued Measurement Matrix $X$}
The setup we consider is as follows. The learner has access to $n$ noiseless linear measurements $Y=X\beta^*\in\mathbb{R}^n$, of an irrational-valued feature vector $\beta^*$ consisting of entries $\beta_i^* \in \mathcal{S}=\{a_1,\dots,a_{\mathcal{R}}\}$. The support $\mathcal{S}$ which is known to the learner consists of rationally independent elements; and $X\in\mathbb{Z}^{n\times p}$ consists of iid entries. We propose the JIRSS algorithm to solve this problem. The informal details of this algorithm are as follows. 

Note that, for any fixed $i$, using $Y_i=\langle X_i,\beta^*\rangle$, where $X_i$ is the $i^{th}$ row of $X$, one can observe that, $Y_i$ can be written as an integral combination of the elements of $\mathcal{S}$, namely, we will establish that, $Y_i=\sum_{j=1}^{\mathcal{R}} \theta^*_{ij}a_j$, for every $i\in[n]$, where integers $\theta_{ij}^*$ are defined in the JIRSS algorithm. The algorithm starts by calling IRA\footnote{IRA stands for the integer relation algorithm, see the comment following Theorem \ref{thm:pslq_main_thm}.} with input $(Y_i,a_1,\dots,a_{\mathcal{R}})$, recovers a relation, and performs rescaling. Due to rational independence, it is not hard to establish that any integer relation for the vector consisting of $Y_i$, and the elements of $\mathcal{S}$ is an integer multiple of a fixed vector, and the rescaling, as we will show, takes out this constant and reveals the fixed vector. Due to the structure of the relation, this multiple can be obtained almost immediately, unlike the corresponding result for rational-valued $\beta^*$, where $Y$ had to be translated by setting $Y_1=Y+XZ$ in order to ensure that the greatest common divisor of its entries is $1$, see line 2 in Algorithm \ref{algo:ELO}, and the associated Theorem \ref{extention}. 
The remainder of the algorithm relies on an observation that, the task of obtaining the underlying values of $\beta^*$ from the coefficients of the integer relation can be achieved by solving randomized subset-sum problems, where the subset membership is defined by the corresponding values of the feature vector $\beta_i^*$. More concretely, if $\xi^{(k)}\in\{0,1\}^p$ is a binary vector, whose $i^{th}$ entry is $1$, if and only if $\beta_i^*=a_k$ for $k=1,2,\dots,\mathcal{R}$; then the coefficient $\theta_{ik}^*$, which is the coefficient in the relation between $Y_i$ and the elements of $\mathcal{S}$ corresponding to $a_k$, can be represented as $\theta_{ik}^*=\langle X_i,\xi^{(k)}\rangle$. This, indeed, is an instance of the subset-sum problem, with the hidden vector observed through multiple channels. Using a slight modification of the  algorithm of Frieze \cite{frieze}, and running LLL on an appropriate lattice whose approximate shortest vectors are integer multiples of $\xi^{(k)}$, we will establish that one can recover $\xi^{(k)}$, and hence, the underlying $\beta^*$.

\begin{algorithm}
\caption{Joint Integer Relation and Subset Sum Algorithm (JIRSS)}
\label{algo:JIRSS}
\LinesNumbered
\KwIn{$(Y,\mathcal{S},X)$, $Y\in\R^n$, $X\in \Z^{n\times p}$, $\mathcal{S}=\{a_1,\dots,a_{\mathcal{R}}\}\subset\R$.}
\KwOut{$\widehat{\beta^*}$.}
\setcounter{AlgoLine}{0}
Run IRA with inputs $(Y_i,a_1,\dots,a_{\mathcal{R}})$, and the output $b^{(i)}=(b_0^{(i)},b_1^{(i)},\dots,b_{\mathcal{R}}^{(i)})$, for $i\in[n]$.  \\
\nl Set $\theta_i^* = (\theta_{i1}^*,\dots,\theta_{i{\mathcal{R}}}^*)=(-b_1^{(i)}/b_0^{(i)},\dots,-b_{\mathcal{R}}^{(i)}/b_0^{(i)})$.\\
\nl Set $m=p2^{\lceil \frac{p+n}{2}\rceil}$.\\
\nl Set $\Theta_j=\begin{bmatrix}\theta_{1j}^*&\theta_{2j}^* & \cdots & \theta_{nj}^*\end{bmatrix}^T$, for each $j\in[\mathcal{R}]$.\\
\nl For each $j=1,2,\dots,\mathcal{R}$, run LLL lattice basis reduction algorithm on the lattice generated by the columns of the following $(n+p)\times (n+p)$ integer-valued matrix,
$$
A_j = \begin{bmatrix}
m\mathrm{diag}_{n\times n}(\Theta_j) & -mX_{n\times p} \\ 0_{p\times n}& I_{p\times p},
\end{bmatrix},\quad j=1,2,\dots,\mathcal{R};
$$
with outputs, $\gamma^{(1)},\dots,\gamma^{(\mathcal{R})}\in\Z^{p+n}$.\\
\nl For each $j=1,2,\dots,\mathcal{R}$, compute $g_j={\rm gcd}(\gamma^{(j)}_1,\dots,\gamma^{(j)}_{n+p})$ using Euclid's algorithm. \\
\nl Set $e^{(j)}=\frac{1}{g_j}\gamma^{(j)}$ for each $j=1,2,\dots,\mathcal{R}$. \\
\nl For each $i=1,2,\dots,p$, set $\widehat{\beta^*_i}=a_j$ for the smallest $j\geq 1$ such that $e^{(j)}_i=1$. Output $\widehat{\beta^*}$.\\
\end{algorithm}


The analysis we pursued imply also that, the high-dimensional linear regression problem in the aforementioned setup is a simultaneous instance of an integer relation detection, and modified subset-sum problems, both of which admit polynomial-in-$p$ time algorithms. In particular, these algorithms turn out to be the fundamental building blocks of our algorithm for recovering $\beta^*$ exactly and efficiently. 

We establish the following formal performance guarantee for the JIRSS algorithm.

\begin{theorem}\label{thm:lllpslqmain}
Let $Y=X\beta^*$, under the following assumptions:
\begin{itemize}
    \item $X\in \Z^{n\times p}$ consisting of iid entries, drawn from a distribution $\mathcal{D}$ on $\Z$, where there exists constants $c,C>0$ such that, $\mathcal{D}$ assigns at most $c/2^N$ probability to each integer, and $\mathbb{E}[|V|]\leq C2^N$, where $V\overset{d}{=}\mathcal{D}$.
    \item $\beta^* \in \R^{p\times 1}$ such that, $\beta_i^*\in \mathcal{S}=\{a_1,\dots,a_{\mathcal{R}}\}$, where $\mathcal{S}$ is rationally independent, and is known to the learner.
\end{itemize}
Then, JIRSS algorithm with input $(Y,S,X)$ terminates with $\widehat{\beta^*}=\beta^*$ with high probability, as $p\to \infty$, in time at most polynomial in $p$ and $\mathcal{R}$, provided
$$
\lim_{p\to\infty}\left(n+p+n\log(n^2p)+\frac{n+p}{2}\log p+\frac{(n+p)^2}{2}-n\log c-nN\right)=-\infty.
$$
\end{theorem} 
The proof of Theorem \ref{thm:lllpslqmain} is deferred to Section \ref{sec:lllpslqmain-pf}. 

We now make the following remarks.  First, one can arrive at a threshold for $N$, in terms of $p$ and $n$, which reveals roughly the correct order at which $N$ should be growing, so that the parameter assumption of Theorem \ref{thm:lllpslqmain} is satisfied. For instance, if $N\geq \frac{1}{2n}(n+p)(n+p+\epsilon\log p)$,
then the limit is always $-\infty$. In particular, when $n$ is fixed, $N$ suffices to be of order at least $p^2/2n$, in order  to ensure the required limiting behaviour. Moreover; the limiting condition between the sample size $n$, discretization $N$, and the dimension $p$ is independent of the size $\mathcal{R}$ of the irrational-valued support $\mathcal{S}$ for the entries of $\beta^*$. Second, the run time is polynomial in both $p$, and $\mathcal{R}$. In particular, one can ensure that the overall process runs in time polynomial in $p$, if $\mathcal{R}$ is at most polynomial in $p$, that is, $\mathcal{R}=p^{O(1)}$.

Note that, with an appropriate choice of input parameters, the efficient recovery is possible even when the learner has access to only one measurement ($n=1$). Quantitatively, any discrete distribution for which, there exists a constant $c>0$, and a parameter $N\geq (1/2+\epsilon)p^2$ with $\epsilon>0$ being bounded away from zero, such that the probability mass of each point is bounded above by $c2^{-N}$, works. As a concrete example, one can consider the uniform distribution on $\{1,2,\dots,2^{cp^2}\}$ where $c=1/2+\epsilon>1/2$, a constant, which is the distribution was studied by Frieze \cite{FriezeSubset}.
\subsubsection{Continuous-Valued Measurement Matrix $X$}\label{subsec:irrational-only}
In this part, we focus our attention on the model, $Y=\langle X,\beta^*\rangle$ with $X\in\mathbb{R}^{1\times p}$ and $\beta^*\in \mathbb{R}^p$, such that, $\beta_i^*\in\mathcal{S}=\{a_1,\dots,a_{\mathcal{R}}\}$, a rationally-independent support known to the learner.  Note that, in this scenario, observing only measurement $(n=1)$ suffice. The reason for this will become clear soon, once the details of the associated algorithm is presented. 
 
We propose the following IHDR algorithm to address this problem.
 \begin{algorithm}[H]
\caption{Irrational High-Dimensional Regression (IHDR)}
\label{algo:IHDR}
\LinesNumbered
\KwIn{$(Y,\mathcal{S},X)$, $Y\in\R$, $X\in \R^{1\times p}$, $\mathcal{S}=\{a_1,\dots,a_{\mathcal{R}}\}\subset\mathbb{R}$.}
\KwOut{$\widehat{\beta^*}$.}
\setcounter{AlgoLine}{0}
Set $\mathcal{L}=\{X_ia_j : i\in[p],j\in [\mathcal{R}]\}$.\\
\nl Run the IRA, with inputs $(Y,\mathcal{L})$, denote the output by $(b_0,b_{ij}:i\in[p],j\in[\mathcal{R}])$. \\
\nl If $b_0=0$, set $\widehat{\beta^*}=0$, and halt.\\
\nl If $b_0\neq 0$, then set ${\bf c}=(c_{ij}:i\in[p],j\in[\mathcal{R}])=(-b_{ij}/b_0:i\in[p],j\in[\mathcal{R}])$.\\
\nl For each $i\in[p]$, set $\widehat{\beta_i^*}=a_{j(i)}$, where $j(i)=\min_{1\leq j\leq R}|b_{ij}|>0$. \\
\nl Output $\widehat{\beta^*}$.
\end{algorithm}
The main idea of the algorithm is as follows. Since $X_i\notin \mathbb{Z}$, there exists no integer relation for the vector consisting of the observation $Y$, and the elements of $\mathcal{S}$. There is, however, a relation between $Y$, and the elements of the set $\mathcal{L}=\{X_ia_j:i\in[p],j\in[\mathcal{R}]\}$, which is generated from $X$ and $\mathcal{S}$, by using $p\mathcal{R}$ (which is polynomial in $p$ and $\mathcal{R}$) arithmetic operations on real numbers. We will establish that provided $X$ is jointly continuous (see Definition \ref{def:joint-cont}); the set $\mathcal{L}$ is rationally independent with probability one, which implies that, any relation for the vector $(Y,\mathcal{L})$ is a multiple of a fixed vector. Then, running IRA with input $(Y,\mathcal{L})$ will allow us to recover this fixed vector, from which the entries of $\beta^*$ can be decoded. The algorithm that we propose works for any $X\in\mathbb{R}^p$, as long as $X$ is a jointly continuous random vector. 



The following result provides the performance guarantee for the IHDR algorithm.
\begin{theorem}
\label{thm:irrational-continuous}
Suppose, $Y=X\beta^*$ with,
\begin{itemize}
    \item $X\in\mathbb{R}^{1\times p}$, a jointly continuous random vector.
    \item For every $i\in[p]$, $\beta_i^*\in\mathcal{S}=\{a_1,\dots,a_{\mathcal{R}}\}$, where $\mathcal{S}$ is rationally independent set known to the learner.
\end{itemize}
Then, the IHDR algorithm with inputs $(Y,\mathcal{S},X)$ terminates with $\widehat{\beta^*}=\beta^*$ with probability $1$, after at most polynomial in $p$ and $\mathcal{R}$ number of aritmetic operations on real numbers.
\end{theorem}
The proof of Theorem \ref{thm:irrational-continuous} is deferred to Section \ref{sec:pf-irrational-continuous}.

Before we close  this section, we make the following remarks. The recovery guarantee holds for any jointly continuous random vector $X\in\mathbb{R}^p$, which is more general than mere iid inputs; and for a single measurement, $n=1$. Second, under the aforementioned assumptions, the noiseless linear regression problem is simply an instance of the integer relation detection problem.

\subsection{Noiseless High-Dimensional Linear Regression with Mixed $\beta^*$}
\subsubsection{Integer-Valued Measurement Matrix $X$}
Let $n,p,\widetilde{R}\in \Z^+$; and $Y=X\beta^* \in\mathbb{R}^n$ be $n$ noiseless linear measurements of a vector $\beta^* \in \mathbb{R}^p$, whose entries satisfy Assumption \ref{def:mixed_support}. 
We assume that only an upper bound $\widehat{R}\geq \widetilde{R}$ and a positive integer $\widehat{Q}$ that is divisible by $Q$ are known to learner. The algorithmic goal is to recover $\beta^*$ exactly, using the information, $Y$ and $X$. Under these assumptions, we propose the following algorithm.



\begin{algorithm}
\caption{MIRR (Mixed Irrational-Rational Regression) Algorithm}
\label{algo:MIRR}
\LinesNumbered
\KwIn{$(Y,X,\widehat{R},\widehat{Q},\mathcal{S})$, $Y\in\R^n$, $X\in \Z^{n\times p}$, $\mathcal{S}=\{a_1,\dots,a_{\mathcal{R}}\}\subset\R$, $\h{R},\h{Q}\in\Z^+$.}
\KwOut{$\widehat{\beta^*}$.}
\setcounter{AlgoLine}{0}
Run JIRSS algorithm, with inputs $(\h{Q}Y,\h{Q}X,\mathcal{S})$. Denote the corresponding output as $\widehat{\beta_1^*}$.\\
\nl For each $i\in[p]$ such that, $(\widehat{\beta_1^*})_i\in \mathcal{S}$, set $(\widehat{\beta^*})_i = (\widehat{\beta_1^*})_i$.\\
\nl Set $\widetilde{Y} = Y-X\widehat{\beta_1^*}$, construct $\widetilde{X}$, by erasing $i^{th}$ column of $X$, if $(\widehat{\beta_1^*})_i \neq 0$. Let the column $n(i)$ of $X$ correspond to the column $i$ of $\widetilde{X}$.\\
\nl If $\widetilde{Y}\notin \Z^n$, halt, and set $\widehat{\beta^*}=0$.\\
\nl If not, and $\widetilde{Y}\in\Z^n$, run ELO algorithm, with input $(\h{Q}Y,X,\h{Q}\h{R},0)$. Denote the output by $\widehat{\beta_2^*}$. \\
\nl Set $(\widehat{\beta^*})_{n(i)}=\frac{1}{Q}(\widehat{\beta_2^*})_i$. Return $\widehat{\beta^*}$. 
\end{algorithm}

The Mixed Irrational-Rational Regression Algorithm (MIRR) builds on JIRSS 
and ELO 
algorithms, described earlier. We next briefly and informally sketch the steps of MIRR algorithm. Note that, $Y=X\beta^*$ implies that, $\h{Q}Y = \h{Q}\beta^*$, which can be written as,
$$
\h{Q}Y_i = \sum_{j=1}^p\h{Q}X_{ij}\beta_j^* = \sum_{j=1}^{\mathcal{R}} \theta_{ij}^* a_j +\sum_{j:\beta_j^*\in \Q}X_{ij}(\h{Q}\beta_j^*) \quad \forall i\in[n],
$$
where $\theta_{ij}^* =Q\sum_{k:\beta_k^* =a_j}X_{ik}$. Observe that, $\sum_{j:\beta_j^*\in \Q}\h{Q}X_{ij}\beta_j^*\in \Z$, due to the $Q$-rationality assumption. In particular, for each $i$, $\h{Q}Y_i$ is an integral combination of the elements of $\mathcal{S}$, and $1$. The JIRSS step of the algorithm finds, for each $i$, an integral relation for the vector $(\h{Q}Y_i,a_1,\dots,a_{\mathcal{R}},1)$, and then using this relation, recovers the irrational-valued entries $\beta_i^*$, and we complete recovering the irrational entries of $\beta^*$. 

The second step of the algorithm is based on the following decomposition of $\beta^*$ into its rational and irrational entries: $\beta^*=\beta_I^*+\beta_R^*$, where, $(\beta_I^*)_i=\beta_i^*$ if $\beta_i^* \notin \Q$, and is $0$ otherwise; and $(\beta_R^*)_i=\beta_i^*$ if $\beta_i^* \in \Q$, and is $0$, otherwise (namely, $\beta_I^*$ stands for the irrational part of $\beta^*$, whereas $\beta_R^*$ stands for its rational part). With this, we notice $Y=X\beta_R^* +X\beta_I^*$, and establish that $\beta_I^*$ coincides with the output $\widehat{\beta_1^*}$ of the JIRSS algorithm, with high probability. Hence, $\widetilde{Y}=Y-X\widehat{\beta_1^*}$  with high probability obeys $\widetilde{Y} = \widetilde{X}\widetilde{\beta}$, where $\widetilde{X}$ is obtained by retaining the columns, corresponding to $0$ entries in $\widehat{\beta_1^*}$, and $\widetilde{\beta}$ is simply the vector obtained by erasing the entries that are $0$ in $\beta_R^*$. From here, the problem of recovery of $\beta_R^*$ is nothing but a regression problem with integer-valued  measurement matrix $X$, and $Q$-rational feature vector, $\widetilde{\beta}$. This has been discussed in Section \ref{sec:ELO-here}. The details of ELO and JIRSS algorithms can be found respectively in Theorem \ref{extention} and Theorem \ref{thm:lllpslqmain}.

Formally, we establish the following recovery guarantee.
\begin{theorem}
\label{thm:main-discrete-X}
Suppose $Y=X\beta^*$, where
\begin{itemize}
      \item $X\in \Z^{n\times p}$ consisting of iid entries, drawn from a distribution $\mathcal{D}$ on $\Z$, where there exists constants $c,C>0$ such that, $\mathcal{D}$ assigns at most $c/2^N$ probability to each integer, and $\mathbb{E}[|V|]\leq C2^N$, where $V\overset{d}{=}\mathcal{D}$;
    \item $\beta^* \in \R^{p\times 1}$ such that, entries of $\beta^*$ satisfy the mixed-support assumption (Assumption \ref{def:mixed_support}), and for $\beta_i^*\in\mathbb{Q}$, it holds that $|\beta_i^*|\leq \widetilde{R}$.
\end{itemize}
Suppose, the learner has access to $\widehat{R}$ where $\widehat{R}\geq \widetilde{R}$, and $\h{Q}$, a multiple of $Q$. Then, the MIRR algorithm with inputs $(Y,X,\widehat{R},\widehat{Q},\mathcal{S})$ recovers $\beta^*$ whp, in at most polynomial in $n,p,N,\mathcal{R},\log \h{R},\log \h{Q}$ number of operations, provided $N$ satisfies, 
\begin{equation}\label{eqn:main-discrete-rat-param}
    N \geq \frac{1}{2n}(2n+p)\left[2n+p+10\log\left( \hat{Q}\hat{R} \sqrt{p}+\sqrt{n}\right) \right]+6 \log \left(\left(1+c\right)np \right).
\end{equation}
\end{theorem}
The proof of Theorem \ref{thm:main-discrete-X} is deferred to Section \ref{sec:pf-main-discrete-X}.

We note that, for Equation \ref{eqn:main-discrete-rat-param} above, a lower bound on $N$ can be replaced with,
$$
    N \geq \frac{1}{2n}(2n+s)\left[2n+s+10\log\left( \hat{Q}\hat{R} \sqrt{p}+\sqrt{n}\right) \right]+6 \log \left(\left(1+c\right)ns\right),
$$
where $s=|\{i\in[p]:\beta_i^*\in \Q\}|$, the number of rational entries in $\beta^*$. We assume herein $s$ is not available to the learner; and use instead the (worst-case) bound of Equation \ref{eqn:main-discrete-rat-param}.

We pause to observe again that the single-sample recovery is indeed possible. Note that, provided $N$ is roughly greater than, $(1/2+\epsilon)p^2$, e.g., when $X$ is drawn from ${\rm Unif}\{1,2,\dots,2^{(1/2+\epsilon)p^2}\}$, one can indeed, with high probability, recover $\beta^*$ exactly and efficiently, even with a single observation, $n=1$.

\subsubsection{Continuous-Valued Measurement Matrix $X$}
In this section, we focus on the same noiseless regression problem, $Y=X\beta^*\in\mathbb{R}^n$, this time, under the assumption that the measurement matrix, $X$, consists of samples of a continuous distribution. As in previous section, we assume that the entries of $\beta^*$ obey the mixed-support assumption (Assumption \ref{def:mixed_support}), and for realistic purposes we assume the values of $\widetilde{R}$ and $Q$ are not known explicitly, but rather, an upper bound $\widehat{R}$ for $\widetilde{R}$ and a positive integer $\h{Q}$ divisible by $Q$ are known.


The algorithm that we develop is the MIRR-C algorithm, given below.
\begin{algorithm}[H]
\caption{MIRR-C (Mixed Irrational-Rational Regression, Continuous) Algorithm}
\label{algo:MIRR-C}
\LinesNumbered
\KwIn{$(Y,X,N,\widehat{R},\widehat{Q},S)$, $Y\in\R^n$, $X\in \R^{n\times p}$, $S=\{a_1,\dots,a_{\mathcal{R}}\}\subset\R$, $\h{R},\h{Q}\in\Z^+$.}
\KwOut{$\widehat{\beta^*}$.}
\setcounter{AlgoLine}{1}
For each $i\in[d]$, set $S_i$ to be the set, $S_i=\{X_{ij}a_k:j\in[p],k\in[\mathcal{R}]\}\cup\{X_{ij}:j\in[p]\}$.\\
\nl Run integer relation detection algorithm, with input $(\h{Q}Y_i,S_i)$. Denote the corresponding output by $b=(b_0,b_{jk},b_{ij}:j\in[p],k\in [\mathcal{R}])$, corresponding to $S_i$. \\
\nl Set $c_{jk}=-b_{jk}/b_0$, for every $j\in [p],k\in [\mathcal{R}]$. \\
\nl For every $j\in[p]$, set $\widehat{(\beta_1^*)_j}=a_k$, for the smallest $k\in[\mathcal{R}]$, such that $b_{jk}\neq 0$. \\
\nl Set $\widetilde{Y}=Y-X\widehat{\beta_1^*}$. If $\widetilde{Y}\notin\Q^n$, halt, and output $\widehat{\beta^*}=0_{p\times 1}$. \\
\nl Set $\widetilde{X}$ by retaining column $i$ of $X$, whenever $\widehat{(\beta_1^*)_i}=0$. Let $n(i)$ be the column $X$, corresponding to column $i$ of $\widetilde{X}$. \\
\nl Run LBR algorithm with inputs $(\widetilde{Y},\widetilde{X},N,\widehat{Q},\widehat{R},0)$. Denote its output by $\widehat{\beta_2^*}$.  \\
\nl For each $i$, set $\beta_i^* = \widehat{(\beta_1^*)_i}$, whenever $\widehat{(\beta_1^*)_i}\neq 0$. Then, set $\beta_{n(i)}^* = \widehat{(\beta_2^*)_i}$ for every $i$, with $\widehat{(\beta_2^*)_i}\neq 0$.\\
\nl Return $\widehat{\beta^*}$. 
\end{algorithm}

The algorithm is two-fold, and uses a mixture of the real-valued computation model for integer relation detection and recovering irrational part of the support, and a finite precision arithmetic model for recovering the $Q$-rational part. The algorithm receives as an input, $Y,X,\mathcal{S}$, as well as the parameters $\h{Q},\widetilde{R}$, and the truncation level, $N$, where we assume, for $x\in\R$, $x_N={\rm sign}(x)\frac{\lfloor 2^N |x|\rfloor}{2^N}$, namely, $x_N$ is the number obtained by keeping the first $N$-bits of $x$ after binary point, and discarding the rest. 

The rationale behind the algorithm is as follows. There is no integer relation, as is, between $Y_i$ and $a_1,\dots,a_{\mathcal{R}}$, as employed in Theorem \ref{thm:main-discrete-X}, since the entries of $X$ are not integer-valued. There is, however, a relation, between $Y_i$, and a certain set, $\mathcal{S}_i$, defined by
\begin{equation}\label{eq:S_i}
\mathcal{S}_i = \{X_{ij}a_k:j\in[p],k\in[\mathcal{R}]\}\cup \{X_{ij}:j\in[p]\},
\end{equation}


With this, we note that,
$$
\h{Q}Y_i = \sum_{j=1}^p X_{ij}(Q\beta_j^*)=\sum_{j=1}^p \sum_{k=1}^{\mathcal{R}} X_{ij}a_k(Qe_{jk}) + \sum_{j:\beta_j^*\in \Q}X_{ij}(Q\beta_j^*),
$$
where $e_{jk}\in\{0,1\}$, and is 1 if and only if $\beta_j^*=a_k$. This decomposition shows indeed that, $\h{Q}Y_i$ is an integral combination of the members of the set $\mathcal{S}_i$, defined in (\ref{eq:S_i}). We will establish that, with probability $1$, $\mathcal{S}_i$ is rationally independent; and consequently, any integer relation for the vector consisting of $\h{Q}Y_i$ and the elements of $\mathcal{S}_i$, are integer multiples of a certain fixed vector. These observations will allow the recovery of the irrational entries of $\beta_i^*$. This is the vector, $\widehat{\beta_1^*}$, defined in the internal steps of the MIRR-C algorithm. 

Given this, we turn our attention to recovering the entries of $\beta^*$, that are $Q$-rational. This is done by first, truncating the continuous-valued input matrix, and then running LBR algorithm, as in Theorem \ref{main}. The following theorem establishes the performance guarantee of the MIRR-C algorithm.
\begin{theorem}
\label{thm:main-cts-noiseless}
Suppose, $Y=X\beta^*$ where,
\begin{itemize}
    \item The measurement matrix $X \in \mathbb{R}^{n \times p}$ consists of iid entries, drawn from a continuous distribution $\mathcal{C}$ which has density $f$ with $\|f\|_{\infty} \leq c$, such that, $\mathbb{E}[|V|]<+\infty$, where $ V\overset{d}{=}\mathcal{C} $;
    \item The vector $\beta^*\in\mathbb{R}^p$ obeys mixed-support assumption (Assumption \ref{def:mixed_support}).
\end{itemize}
Suppose that, the learner has access to $\h{Q},\h{R},N\in \mathbb{Z}^+$, such that $\h{Q}$ is divisible by $Q$ and $\h{R}\geq \widetilde{R}$.
Then, MIRR-C algorithm with input $(Y,X,N,\h{R},\h{Q},S)$ terminates with $\h{\beta^*}=\beta^*$ with high probability, in time that is at most polynomial in $n,p,N,s,\log \h{R},\log \h{Q},\mathcal{R}$; provided the truncation level $N$ satisfies,
$$
N>\frac{1}{2}\left(2n+p\right)\left(2n+p+10\log \hat{Q}+ 10\log \left(\hat{R}p \right) +20 \log (3\left(1+c\right)np) \right).
$$
\end{theorem}
The proof of Theorem \ref{thm:main-cts-noiseless} is deferred to Section \ref{sec:main-cts-noiseless}.

As in the previous setting, it is possible to replace the $p$'s in the lower bound above with $s$'s, where $s=|\{i\in[p]:\beta_i^*\in \Q\}|$. For realistic purposes, we assume, however, $s$ is not available to the learner, and resort to a bound with $p$, as stated in the preamble of Theorem \ref{thm:main-cts-noiseless}.

A further and detailed look at the noiseless regression problem reveal also that, if $\beta^*\in\R^p$ enjoys the mixed support assumption (Assumption \ref{def:mixed_support}), then the problem of recovering $\beta^*$ can also be cast directly as an integer relation detection problem, and consequently, integer relation oracle alone can be used for recovery. For completeness, we concretize this observation in the following theorem:
\begin{theorem}\label{thm:cts-noiseless-ira-only}
Suppose $Y=X\beta^*\in\R$, where
\begin{itemize}
    \item The measurement vector $X\in\R^{1\times p}$ is a jointly continuous random vector.
    \item The feature vector $\beta^*\in\R^p$ obeys the mixed-support assumption (Assumption \ref{def:mixed_support}). 
\end{itemize}
Suppose that, the learner has access to $\widehat{Q},\widehat{R}\in\mathbb{Z}^+$, such that $\widehat{Q}$ is divisible by $Q$, and $\widehat{R}\geq \widetilde{R}$. Then, IRA with input $(\widehat{Q}Y,X_i a_j, X_i : i\in[p],j\in[\mathcal{R}])$ recovers $\beta^*$ with probability one, in time that is at most polynomial in $p,\mathcal{R},\log \widehat{Q}$ and $\log \widehat{R}$. 
\end{theorem}
Note that, Theorem \ref{thm:cts-noiseless-ira-only} is a single-sample recovery guarantee as promised, and the statement of the theorem holds as long as the measurement vector $X\in\R^p$ is jointly continuous. The proof of this result is deferred to Section \ref{sec:pf-cts-noiseless-ira-only}.

\subsection{Application: Phase Retrieval Problem}
In this section, we address the so-called phase retrieval problem, using the ideas outlined in earlier sections. The goal of the learner is to recover a vector $\beta^{*}\in\mathbb{C}^p$ efficiently, using the following magnitude-only observations,
$$
Y_i=|\langle X_i,\beta^*\rangle|, \quad i=1,2,\dots,n
$$
with a small number $n$ of measurements, where, for each $i$, $X_i$ consists of random samples of a known distribution $\mathcal{D}$. We address this problem, when the elements of $\beta^*$ are supported on a bounded cardinality set, under a certain rational independence assumption.

\subsubsection{Discrete-Valued $X$}

We establish the JIRSS-based phase retrieval algorithm (see below), which provably recovers $\beta^*\in\mathbb{C}^p$ in polynomial time. The algorithm is given below, and its performance guarantee is the subject of Theorem \ref{thm:phase_retrieval}.

\begin{algorithm}
\caption{JIRSS-based Phase Retrieval Algorithm}
\label{algo:JIRSS-Phase}
\LinesNumbered
\KwIn{$(Y,\mathcal{S},X)$, $Y\in\R$, $X\in \Z_+^{1\times p}$, $\mathcal{S}=\{a_1,\dots,a_{\mathcal{R}}\}\subset\mathbb{C}$.}
\KwOut{$\widehat{\beta^*}$.}
\setcounter{AlgoLine}{1}
Construct the set, $\mathcal{S}'=\{|a_i|^2:1\leq i\leq \mathcal{R}\}\cup \{a_\alpha^Ha_\beta + a_\alpha a_\beta^H : 1\leq \alpha<\beta\leq \mathcal{R}\}$.\\
\nl Run integer relation detection algorithm with inputs $(Y,S')$, output $b=(b_0,b_1,\dots,b_{\mathcal{R}},b_{\alpha\beta}:1\leq \alpha<\beta\leq \mathcal{R})$.\\
\nl Set $\theta^* = (\theta_1^*,\dots,\theta_{\mathcal{R}}^*,\theta_{\alpha\beta}^*:1\leq \alpha<\beta\leq \mathcal{R})=(-b_1/b_0,\dots,-b_{\mathcal{R}}/b_0,-b_{\alpha\beta}/b_0:1\leq \alpha<\beta\leq \mathcal{R})$.\\
\nl Set $m=p 2^{\lceil(c'+1/2)p\rceil} $, and $N=\binom{p}{2}$.\\
\nl Set $Y=(Y_1,\dots,Y_N)=(X_iX_j:1\leq i<j\leq p)$, using any ordering, e.g. lexicographical.\\
\nl For each $(\alpha,\beta):1\leq \alpha<\beta\leq \mathcal{R}$, run LLL lattice basis reduction algorithm on the lattice generated by the columns of the following $(N+1)\times (N+1)$ integer-valued matrix,
$$
A_{\alpha,\beta} = \begin{bmatrix}
m\theta_{\alpha\beta}^* & -mY \\ 0_{N\times 1}& I_{N\times N}
\end{bmatrix},
$$
with outputs, $\gamma^{(\alpha,\beta)}\in\Z^{N+1}$ for $1\leq \alpha<\beta\leq \mathcal{R}$.\\
\nl For each $1\leq \alpha<\beta\leq \mathcal{R}$, compute $g_{\alpha\beta}={\rm gcd}(\gamma^{(\alpha,\beta)}_0,\dots,\gamma^{(\alpha,\beta)}_N)$ using Euclid's algorithm. \\
\nl Set $e^{(\alpha,\beta)}=\frac{1}{g_{\alpha\beta}}\gamma^{(\alpha,\beta)}$ for each $1\leq \alpha<\beta\leq \mathcal{R}$. \\
\nl For each $1\leq \alpha<\beta<\delta\leq \mathcal{R}$, add all indices $i,j$ with $1\leq i<j\leq p$ to $S_{\alpha,\beta}$ whenever the corresponding $\gamma_k^{(\alpha,\beta)}=1$. For every $i\in S_{\alpha,\beta}\cap S_{\alpha,\delta}$ set $\widehat{\beta^*_i}=a_\alpha$. Break ties arbitrarily.\\
\end{algorithm}

\begin{theorem}
\label{thm:phase_retrieval}
Let $Y=|\langle X,\beta^*\rangle|$ with,
\begin{itemize}
    \item $X\in \Z_+^{1\times p}$, with iid entries drawn from a distribution $\mathcal{D}$ on $\Z_+$, where there exists constants $c,C>0$, such that $\mathcal{D}$ assigns at most $c/2^N$ probability to each element of $\Z_+$, and $\mathbb{E}[X_1]\leq C2^N$. \item $\beta^*_i \in \mathcal{S}=\{a_1,\dots,a_{\mathcal{R}}\}\subset \mathbb{C}$, such that the set, $\mathcal{S'}=\{|a_d|^2:d\in[\mathcal{R}]\}\cup \{a_i^H a_j+a_ia_j^H:1\leq i<j\leq \mathcal{R}\}$ is rationally independent. 
    \end{itemize}
Then, the JIRSS-based phase retrieval algorithm with input $(Y,\mathcal{S},X)$ terminates with $\widehat{\beta^*}=\beta^*$ with high probability after at most polynomial in $p$ and $\mathcal{R}$ number of arithmetic operations, provided $N\geq (1/8+\epsilon)p^4$, for some $\epsilon>0$.
\end{theorem}

The proof of Theorem \ref{thm:phase_retrieval} is deferred to Section \ref{sec:phase_retrieval}.

The high-level idea is similar to that of Theorem \ref{thm:lllpslqmain}. We first observe that $Y^2$ can be expressed via integer combinations of the elements of $\mathcal{S}'$, since
$$
Y^2= \sum_{i=1}^p X_i^2 |\beta^*_i|^2+\sum_{1\leq i<j\leq p}X_iX_j((\beta_i^*)^H\beta_j^*+\beta_i^*(\beta_j^*)^H).
$$
Using similar ideas, one can prove that, any integer relation for the vector consisting of $Y$ and the elements of $\mathcal{S'}$ is an integer multiple of a fixed vector. However, the corresponding coefficients are not exactly subset-sums of $X_i$'s, but are more involved. Nevertheless, a variation of the result of Frieze \cite{FriezeSubset} allows us to overcome with this obstacle, which, as a by-product, yields an extension of the Frieze's result on the randomly-generated subset-sum problem to random variables with some dependence. This auxiliary result may be of independent interest, and is isolated in Proposition \ref{lemma:extended-frieze}.
\begin{proposition}
\label{lemma:extended-frieze}
Let 
\begin{itemize}
    \item $X_1,\dots,X_p$ be iid random variables, drawn from a distribution $\mathcal{D}$ supported on $\Z_+$, where there exists constants $c,C>0$ such that, for each $k\in\mathbb{Z}^+$, $\mathbb{P}(X_1=k)\leq c2^{-N}$, and $\mathbb{E}[X_1]\leq C2^N$. 
    \item Let $Y=\sum_{1\leq i<j\leq p}Y_{ij} \xi_{ij}$ with $\xi_{ij}\in\{0,1\}$, where $Y_{ij}=X_iX_j$. 
\end{itemize}
Then, there exists an algorithm, which admits $Y$ and $\{X_i\}_{i=1}^p$ as its inputs, and recovers $\xi_{ij}$'s with high probability, as $p\to\infty$, in polynomial in $p$ many bit operations, provided that $N\geq (1/8+\epsilon)p^4$, for any $\epsilon>0$.
\end{proposition}
The proof of Proposition \ref{lemma:extended-frieze} is deferred to Section \ref{sec:extended-frieze}.

This algorithm is again based on running LLL algorithm on an appropriate lattice, in a similar way as in \cite{FriezeSubset}, such that approximate short vectors of this lattice are integer multiples of the binary vector $\xi$ we wish to recover. The lattice is similar to the ones used in earlier results, and  will become apparent from the proof. However, due to the presence of dependence, the details of the proof are more involved.  

\subsubsection{Continuous-Valued $X$}

We now continue with our study of the phase retrieval problem, this time with continuous-valued measurement matrix $X$, where the entries of hidden $\beta^*\in\mathbb{C}$ are supported on a finite cardinality subset $\mathcal{S}=\{a_1,\dots,a_{\mathcal{R}}\}\subset\mathbb{C}$ known to learner, such that the set, $\mathcal{S'}=\{|a_i|^2:i\in[\mathcal{R}]\}\cup \{a_i^H a_j+a_ia_j^H:1\leq i<j\leq \mathcal{R}\}$, which can be obtained from $\mathcal{S}$ after at most $O(\mathcal{R}^2)$ arithmetic operations on real numbers, is rationally independent. Using similar ideas as above,  one can observe, as a consequence of $Y=|\langle X,\beta^*\rangle|$, that $Y^2$ can be represented as a linear combination of the elements of a set $\mathcal{L}$, where $\mathcal{L}=\mathcal{S}_1\cup\mathcal{S}_2\cup\mathcal{S}_3$, with
\begin{itemize}
\item $\mathcal{S}_1 = \{X_i^2|a_k|^2:i\in[p],k\in[\mathcal{R}]\}$,
\item $\mathcal{S}_2=\{X_iX_j|a_k|^2:1\leq i<j\leq p,k\in[\mathcal{R}]\}$,
\item $\mathcal{S}_3=\{X_iX_j(a_k^Ha_\ell+a_ka_\ell^H):1\leq i<j\leq p,1\leq k<\ell\leq \mathcal{R}\}$.
\end{itemize}
Note that, $|\mathcal{S}_1|=p\mathcal{R}$, $|\mathcal{S}_2|=O(p^2\mathcal{R})$ and $|\mathcal{S}_3|=O(p^2\mathcal{R}^2)$, and therefore, $|\mathcal{L}|\leq O(p^2\mathcal{R}^2)$. Thus, $\mathcal{L}$ can be obtained from $\mathcal{S}$ and $\{X_i\}_{i=1}^p$ in at most polynomial in $p$ and $R$ many arithmetic operations on real numbers. The most crucial step is to show, $\mathcal{L}$ is rationally independent, with probability one, and this is achieved by combining several ideas from the proof of Theorems \ref{thm:irrational-continuous} and \ref{thm:phase_retrieval}. Having establish that $\mathcal{L}$ is rationally independent, we then run IRA to recover $\beta^*$, similar to earlier IHDR algorithm (Algorithm \ref{algo:IHDR}) proposed for recovering an irrational-valued $\beta^*$ supported on a known, bounded cardinality, rationally independent set, from its noiseless linear measurements $Y=X\beta^*$.
\begin{theorem}
\label{thm:phase-retrieval-cts}
Let $Y=|\langle X,\beta^*\rangle|$ with,
\begin{itemize}
    \item The measurement matrix $X\in \mathbb{R}^{1\times p}$ is a jointly continuous random vector. 
    \item $\beta_i^*\in \mathcal{S}=\{a_1,\dots,a_{\mathcal{R}}\}\in \mathbb{C}$, where $\mathcal{S}$ is known to learner, and $\mathcal{S'}=\{|a_i|^2:i\in[\mathcal{R}]\}\cup\{a_i^Ha_j+a_ia_j^H:1\leq i<j\leq \mathcal{R}\}$ is rationally independent.
\end{itemize}
Then, there exists an algorithm which admits $(Y,\mathcal{S},X)$ as its input, and terminates with $\widehat{\beta^*}=\beta^*$ with high probability, after at most polynomial in $p$ and $\mathcal{R}$ many arithmetic operations on real numbers.
\end{theorem}

Observe that, the assertion of the theorem is valid as long as $X\in \R^p$ is a jointly continuous random vector, which is a milder requirement than $X\in \R^p$ having iid coordinates; and the recovery guarantee is provided for a single measurement. 

The proof of Theorem \ref{thm:phase-retrieval-cts} is deferred to Section \ref{sec:phase-retrieval-cts}.

We note that, our approach is not limited to phase retrieval setup; and  transfers to more general measurement models. Consider an observation model, $Y=f(|\langle X,\beta^*\rangle|)$, where $f(\cdot)$ is an  polynomial with ${\rm deg}(f)=d$, and $\beta^*\in\mathbb{C}^p$. Note that, $f(t)=Y$ implies, $t$ is a root of the polynomial, $g(t)=f(t)-Y$, where ${\rm deg}(g)=d$. Denote by $H=\{t_1,\dots,t_d\}\subset \mathbb{C}$ the set of all roots of $g$. Assume that these roots can be found, e.g., by means of an explicit formula, as in the case for $d\leq 4$.  Now, for each $i$ with $t_i\in\mathbb{R}_{\geq 0}$, we solve for $(\beta^{(i)})^*\in\mathbb{C}^p$, by running phase retrieval solver with input $(t_i,\mathcal{S},X)$, that is, we recover $(\beta^{(i)})^*$, such that, $t_i=|\langle X,(\beta^{(i)})^*\rangle|$. It is clear that, the entire process runs in time that is at most polynomial in $p,R$, and $d$. 

Yet another class of functions $f(\cdot)$ for which our methods immediately extend, is the class of strictly monotone functions. Provided that $f^{-1}(Y)$ can be computed (e.g., again either exactly or by means of a formula), we can simply run the phase retrieval solver with inputs $(f^{-1}(Y),\mathcal{S},X)$ to recover $\beta^*$. Namely, let $\mathcal{C}$ be a class of functions $f:\mathbb{R}\to\mathbb{R}$, where $f\in\mathcal{C}$, if for every $r\in\mathbb{R}$, the set $\{t:f(t)=r\}$ is finite, and can be computed, in the aforementioned sense. Then, our approach extends immediately, when a feature vector $\beta^*$ is observed through the mechanism, $Y=f(|\langle X,\beta^*\rangle|)$, provided $\beta^*$ satisfies a similar rational independence assumption.

\section{Proofs}

\subsection{Proof of Theorem \ref{extention}}\label{ProofExt}
\begin{proof}
We first observe that directly from (\ref{eq:limit}), 
\begin{align*} 
N &\geq 10 \log \left( \sqrt{p}+\sqrt{n} \left(\|W\|_{\infty}+1\right) \right)\\
& \geq 5  \log \left( \sqrt{p}\sqrt{n} \left(\|W\|_{\infty}+1\right) \right), \text{ from the elementary }a+b  \geq \sqrt{ab}\\
& \geq 2 \log \left( pn \left(\|W\|_{\infty}+1\right) \right).
\end{align*}Therefore $2^N \geq \left(pn \left(1+\|W\|_{\infty}\right) \right)^2$ which easily implies $$\frac{\|W\|_{\infty}}{2^N} \leq \frac{1}{n^2p^2}=\delta,$$where we set for convenience $\delta=\delta_p:=\frac{1}{n^2p^2}.$

\begin{lemma}\label{lemmaY}
For all $i \in [n]$, $|(Y_2)_i| \geq \frac{3}{2} \delta2^N$, w.p.  at least $1-O \left( \frac{1}{np} \right).$
\end{lemma}
\begin{proof}First if $\delta 2^N < 2$, for all $i \in [n]$, $|(Y_2)_i| \geq 3 \geq \frac{3}{2} \delta2^N$, because of the second step of the algorithm. 

Assume now that $\delta 2^N \geq  2$. In that case first observe $Y_1:=Y+XZ=X(\beta^*+Z)+W$ and therefore from the definition of $Y_2$, $Y_2=X \left(\beta^*+Z\right)+W_1$ for some $W_1 \in \mathbb{Z}^n$ with $\|W_1\|_{\infty} \leq \|W\|_{\infty}+1$.
Letting $\beta=\beta^*+Z$ we obtain that for all $i \in [n]$, $Y_i=\inner{X^{(i)}}{\beta}+(W_1)_i$, where $X^{(i)}$ is the $i$-th row of $X$, and therefore $$(Y_2)_i \geq |\sum_{j=1}^p X_{ij} \beta_j|-\|W_1\|_{\infty} \geq  |\sum_{j=1}^pX_{ij} \beta_j|-\|W\|_{\infty}-1.$$Furthermore $\hat{R} \geq  R$ implies $\beta \in [1,3\hat{R}+\log p]^p$.

We claim that conditional on $\beta \in [1,3\hat{R}+p]^p$ for all $i=1,\ldots,n$, $|\sum_{j=1}^p X_{ij} \beta_j| \geq 3\delta 2^N$  w.p. at least $1-O \left( \frac{1}{np} \right)$ with respect to the randomness of $X$. Note that this last inequality alongside with  $\|W\|_{\infty} \leq \delta 2^N$ implies for all $i$, $|(Y_2)_i| \geq 2 \delta 2^N-1$. Hence since $\delta 2^N \geq 2$ we can conclude from the claim that for all $i$, $|(Y_2)_i| \geq \frac{3}{2} \delta 2^N$ w.p. at least $1-O \left( \frac{1}{np} \right)$ Therefore it suffices to prove the claim to establish Lemma \ref{lemmaY}.

In order to prove the claim, observe that for large enough $p$, \begin{align*}
\mathbb{P}\left( \bigcup_{i=1}^n \{ |\sum_{j=1}^p X_{ij} \beta_j|<3\delta 2^N \}\right) &\leq \sum_{i=1}^n\mathbb{P}\left(|\sum_{j=1}^p X_{ij} \beta_j|<3 \delta 2^N\right)\\
&=\sum_{i=1}^n \sum_{k \in \mathbb{Z} \cap [-3\delta 2^N,3\delta 2^N]} \mathbb{P}\left(\sum_{j=1}^p X_{ij}\beta_j=k\right)\\ 
&\leq n  (6 \delta 2^N+1) \frac{c}{2^N}\\
& \leq 7c n \delta = O\left(\frac{1}{np} \right),\end{align*} where we have used that given $\beta_{1} \not = 0$ for $i \in [p]$ and $k \in \mathbb{Z}$ the event $\{\sum_{j=1}^p X_{ij} \beta_j=k\}$ implies that the random variable $X_{i1}$ takes a specific value, conditional on the realization of the remaining elements $X_{i2},\ldots,X_{ip}$ involved in the equations. Therefore by our assumption on the iid distribution generating the entries of $X$, each of these events has probability at most $c/2^N$. Note that the choice of $\beta_1$, as opposed to choosing some $\beta_i$ with $i>1$, was arbitrary in the previous argument. The last inequality uses the assumption $\delta 2^N\geq 1$ and the final convergence step is justified from $\delta =O(\frac{1}{n^2 p})$ and that $c$ is a constant.
\end{proof}
Next we use a number-theoretic lemma, which is an extension of a standard result in analytic number theory according to which $$\lim_{m \rightarrow +\infty} \mathbb{P}_{P,Q \sim \mathrm{Unif}\{1,2,\ldots,m\}, P \independent  Q}\left[ \mathrm{gcd}\left(P,Q\right)=1\right]= \frac{6}{\pi^2},$$ where  $P \independent  Q$ refers to $P,Q$ being independent random variables. This result is not of clear origin in the literature, but possibly it is attributed to Chebyshev, as mentioned in \cite{ErdosPrime}.
\begin{lemma}\label{gcdNT}
Suppose $q_1,q_2,q \in \mathbb{Z}^+$ with $q \rightarrow + \infty$ and $\max\{q_1,q_2\}=o(q^2)$. Then  $$ |\{(a,b) \in \mathbb{Z}^2 \cap ( [q_1,q_1+q]\times [q_2,q_2+q] ) : \mathrm{gcd}(a,b)=1\}|=q^2\left(\frac{6}{\pi^2}+o_q(1)\right).$$In other words, if we choose independently one uniform integer in $[q_1,q_1+q]$ and another uniform integer in $[q_2,q_2+q]$ the probability that these integers are relatively prime approaches $\frac{6}{\pi^2}$, as $q \rightarrow +\infty$.
\end{lemma}

\begin{proof}
We call an integer $n \in \mathbb{Z}^+$ square-free if it is not divisible by the square of a positive integer number other than 1. The \textbf{Mobius function} $\mu: \mathbb{Z}^+ \rightarrow \{-1,0,1\}$ is defined to be \[ \mu(n)=\begin{cases} 
      \text{    } 1, & $n$ \text{ is square-free with an even number of prime factors}  \\
      -1, & $n$ \text{ is square-free with an odd number of prime factors} \\
     \text{  } 0, & \text{ otherwise} 
   \end{cases}
\]

From now on we ease the notation by always referring for this proof to positive integer variables. A standard property for the Mobius function (see Theorem 263 in \cite{Hardy}) states that for all $n \in \mathbb{Z}^+$, 
\[\sum_{1 \leq d \leq n, d | n} \mu(d)=\begin{cases} 
      1, & n=1  \\
       0, & \text{ otherwise} 
   \end{cases}
\]Therefore using the above identity and switching the order of summation we obtain
\begin{align*}
&|(a,b) \in  [q_1,q_1+q]\times [q_2,q_2+q], \mathrm{gcd}(a,b)=1|\\
=&\sum_{(a,b) \in  [q_1,q_1+q]\times [q_2,q_2+q]} \left( \sum_{1 \leq d \leq\mathrm{gcd}(a,b), d|\mathrm{gcd}(a,b)} \mu(d)\right)  \\
= &\sum_{1 \leq d \leq \max\{q_1,q_2\}+q} \left(\sum_{(a,b) \in  [q_1,q_1+q]\times [q_2,q_2+q], d| \mathrm{gcd}(a,b)}\mu(d) \right).
\end{align*}
Now introducing the change of variables $a=kd,b=ld$ for some $k,l\in \mathbb{Z}^+$ and observing that the number of integer numbers in an interval of length $x>0$ are $x+O(1)$, we obtain
\begin{align*}
&\sum_{1 \leq d \leq \max\{q_1,q_2\}+q} \left(\sum_{ \frac{q_1}{d} \leq k \leq \frac{q_1+q}{d}, \frac{q_2}{d} \leq l \leq \frac{q_2+q}{d}}\mu(d) \right)\\
=&\sum_{1 \leq d \leq \max\{q_1,q_2\}+q} \left[\left(\frac{q}{d}+O(1)\right)^2\mu(d)\right]\\
=&\sum_{1 \leq d \leq \max\{q_1,q_2\}+q}\left[ \left(\frac{q}{d}\right)^2\mu(d)+O\left(\frac{q}{d}\right)\mu(d)+O(1)\mu(d)\right]
\end{align*}Now using $|\mu(d)| \leq 1$ for all $d \in \mathbb{Z}^+$, for $n \in \mathbb{Z}^+$, $$\sum_{d=1}^{n} \frac{1}{d}=O(\log n)$$and that by Theorem 287 in \cite{Hardy} for $n \in \mathbb{Z}^+$, $$\sum_{d=1}^{n} \frac{\mu(d)}{d^2}=\frac{1}{\zeta(2)}+o_n(1)=\frac{6}{\pi^2}+o_n(1)$$we conclude that the last quantity equals
\begin{align*}
q^2\left(\frac{6}{\pi^2}+\frac{1}{q}O(\log (\max\{q_1,q_2\}+q))+\frac{\max\{q_1,q_2\}+q}{q^2}+o_q(1)\right).
\end{align*}Recalling the assumption $q_1,q_2=o(q^2)$ the proof is complete.
\end{proof}
\begin{claim}\label{gcd}
The greatest common divisor of the coordinates of $\beta:=\beta^*+Z$ equals to 1, w.p. $1-\exp \left(-\Theta\left(p \right) \right)$ with respect to the randomness of $Z$.
\end{claim}

\begin{proof}
Each coordinate of $\beta$ is a uniform and independent choice of a positive integer from an interval of length $2\hat{R}+\log p$ with starting point in $[\hat{R}-R+1,\hat{R}+R+1]$, depending on the value of $\beta^*_i \in [-R,R]$. Note though that Lemma \ref{gcdNT} applies for arbitrary $q_1,q_2 \in [\hat{R}-R+1,\hat{R}+R+1]$ and $q=2 \hat{R}+\log p$ since $q_1,q_2=o(q^2)$ and $q \rightarrow +\infty$. from this we conclude that the probability any two specific coordinates of $\beta$ have greatest common divisor 1 approaches $\frac{6}{\pi^2}$, as $p\rightarrow +\infty$. But the probability the greatest common divisor of all the coordinates is not one implies that the greatest common divisor of the $2i-1$ and $2i$ coordinate is not one, for every $i=1,2,\ldots\floor{\frac{p}{2}}$. Hence using the independence among the values of the coordinates, we conclude that the greatest common divisor of the coordinates of $\beta$ is not one with probability at most $$\left(1-\frac{6}{\pi^2}+o_p(1)\right)^{\floor{\frac{p}{2}}} =\exp \left(-\Theta\left(p \right) \right).$$ 
\end{proof}

Given a vector $z \in \mathbb{R}^{2n+p}$, define $z_{n+1:p}:=(z_{n+1},\ldots,z_{n+p})^t$. 
\begin{claim}\label{multiple} The outcome of Step 5 of the algorithm, $\hat{z}$, satisfies
\begin{itemize}
\item $\|\hat{z}\|_2<m$ 
\item  $\hat{z}_{n+1:n+p}= q \beta$, for some $q \in \mathbb{Z}^*$, w.p. $1-O\left(\frac{1}{np}\right)$.
\end{itemize}
\end{claim}
\begin{proof}
Call $\mathcal{L}_m$the lattice generated by the columns of the $(2n+p)\times (2n+p)$ integer-valued matrix $A_m$ defined in the algorithm; that is $\mathcal{L}_m:=\{A_mz | z \in \mathbb{Z}^{2n+p}\}$. Notice that as $Y_2$ is nonzero at every coordinate, the lattice $\mathcal{L}_m$ is full-dimensional and the columns of $A_m$ define a basis for $\mathcal{L}_m$. Finally, an important vector in $\mathcal{L}_m$ for our proof is $z_0 \in \mathcal{L}_m$ which is defined for $1_n \in \mathbb{Z}^n$ the all-ones vector as \begin{equation} z_0:=A_m \left[ {\begin{array}{c}
    \beta \\
     1_{n} \\
    W_1
  \end{array} } \right]=\left[ {\begin{array}{c}
    0_{n \times 1} \\
     \beta \\
    W_1
  \end{array} } \right] \in \mathcal{L}_m. \end{equation}  
  
  Consider the following optimization problem on $\mathcal{L}_m$, known as the shortest vector problem, $$\begin{array}{clc}(\mathcal{S}_2)  & \min  &\|z\|_2 \\ &\text{s.t.}&  z \in \mathcal{L}_m, 
\end{array}$$ If $z^*$ is the optimal solution of $(\mathcal{S}_2)$ we obtain $$\|z^*\|_2 \leq \|z_0\|_2=\sqrt{\|\beta\|_2^2+\|W_1\|_2^2} \leq \|\beta\|_{\infty}\sqrt{p}+\|W_1\|_{\infty}\sqrt{n}.$$ and therefore given our assumptions on $\beta,W$ $$\|z^*\|_{2} \leq \left(3\hat{R}+\log p\right) \sqrt{p}+\left(\|W\|_{\infty}+1\right) \sqrt{n}.$$Using that $\hat{R} \geq 1$ and a crude bound this implies $$\|z^*\|_{2} \leq 4p \left(\hat{R} \sqrt{p}+\left(\|W\|_{\infty}+1\right)\sqrt{n} \right).$$ The LLL guarantee and the above observation imply that
   \begin{equation}\label{eq:LLLguar}
\|\hat{z}\|_2 \leq 2^{\frac{2n+p}{2}}\|z^*\|_2 \leq 2^{\frac{2n+p}{2}+2}p\left(\hat{R} \sqrt{p}+\left(\|W\|_{\infty}+1\right) \sqrt{n}\right):=m_0.
\end{equation}
Now recall that $ \hat{W}_{\infty} \geq \max\{\|W\|_{\infty},1\}$. Since $m \geq 2^{n+\frac{p}{2}+3}p\left(\hat{R} \sqrt{p}+\hat{W}_{\infty} \sqrt{n}\right)$, we obtain $m>m_0$ and hence $\|\hat{z}\|_2 <m$. This establishes the first part of the Claim.

For the second part, given (\ref{eq:LLLguar}) and that $\hat{z}$ is non-zero it suffices to establish that under the conditions of our Theorem there is no non-zero vector in $\mathcal{L}_m \setminus \{  z \in \mathcal{L}_m|  z_{n+1:n+p}= q \beta, q \in \mathbb{Z}^* \}$ with $L_2$ norm less than $m_0$, w.p. $1-O\left(\frac{1}{np}\right)$. By construction of the lattice for any $z \in \mathcal{L}_m$ there exists an $x \in \mathbb{Z}^{2n+p}$ such that $z=A_mx$. We decompose $x=(x_1,x_2,x_3)^t$ where $x_1 \in \mathbb{Z}^p, x_2,x_3 \in \mathbb{Z}^n$. It must be true \begin{equation*} z=\left[ {\begin{array}{c}
    m \left(Xx_1-\mathrm{Diag}_{n \times n}(Y)x_2+x_3 \right)\\
     x_1 \\
    x_3
  \end{array} } \right]. \end{equation*}Note that $x_1= z_{n+1:n+p}$. We use this decomposition of every $z \in \mathcal{L}_m$ to establish our result. 

We first establish that for any lattice vector $z \in \mathcal{L}_m$ the condition $\|z\|_2 \leq m_0$ implies necessarily \begin{equation}\label{neces}Xx_1-\mathrm{Diag}_{n \times n}(Y)x_2+x_3 = 0.\end{equation}and in particular $z=(0,x_1,x_3)$. If not, as it is an integer-valued vector, $\|Xx_1-\mathrm{Diag}_{n \times n}(Y)x_2+x_3||_2 \geq 1$ and therefore 
  \begin{equation*} 
m \leq m \|Xx_1-\mathrm{Diag}_{n \times n}(Y)x_2+x_3\|_2 \leq  \|z\|_2 \leq m_0 ,
  \end{equation*}a contradiction as $m>m_0$. Hence, necessarily equation (\ref{neces}) and $z=(0,x_1,x_3)$ hold. 

Now we claim that it suffices to show that there is no non-zero vector in $\mathcal{L}_m \setminus \{  z \in \mathcal{L}_m|  z_{n+1:n+p}= q \beta, q \in \mathbb{Z} \}$ with $L_2$ norm less than $m_0$,  w.p. $1-O\left(\frac{1}{np}\right)$. Note that in this claim the coefficient $q$ is allowed to take the zero value as well. The reason it suffices to prove this weaker statement is that any non-zero $z \in \mathcal{L}_m$ with $\|z\|_2 \leq m_0$ necessarily satisies that $z_{n+1:n+p} \not = 0$  w.p. $1-O\left(\frac{1}{np}\right)$ and therefore the case $q=0$ is not possible w.p. $1-O\left(\frac{1}{np}\right)$. To see this, we use the decomposition and recall that $x_1=z_{n+1:n+p}$. Therefore it suffices to establish that there is no triplet $x=(0,x_2,x_3)^t \in \mathbb{Z}^{2n+p}$ with $x_2,x_3 \in \mathbb{Z}^n$ for which the vector $z=A_mx \in \mathcal{L}_m$ is non-zero and $\|z\|_2 \leq m_0$,  w.p. $1-O\left(\frac{1}{np}\right)$. To prove this, we consider such a triplet $x=(0,x_2,x_3)$ and will upper bound the probability of its existence. From equation (\ref{neces}) it necessarily holds $\mathrm{Diag}_{n \times n}(Y)x_2=x_3$, or equivalently \begin{equation}\label{newequation} \text{ for all } i \in [n], Y_i (x_2)_i=(x_3)_i. \end{equation} From Lemma \ref{lemmaY} and (\ref{newequation}) we obtain that \begin{equation}\label{newequation1} \text{ for all } i \in [n], \frac{3}{2}\delta 2^N |(x_2)_i| \leq |(x_3)_i|\end{equation}w.p. $1-O\left(\frac{1}{np}\right)$. Since $z$ is assumed to be non-zero and $z=A_mx=(0,0,x_3)$ there exists $i \in [n]$ with $(x_3)_i \not = 0$. Using (\ref{newequation}) we obtain $(x_2)_i \not =0 $ as well. Therefore for this value of $i$ it must be simultaneously true that $|(x_2)_i| \geq 1$ and $|(x_3)_i| \leq m_0$. Plugging these inequalities to (\ref{newequation1}) for this value of $i$, we conclude that it necessarily holds that $$\frac{3}{2}\delta 2^N \leq m_0$$ Using the definition of $\delta$, $\delta= \frac{1}{n^2p^2}$, we conclude that it must hold $\frac{1}{n^2p^2} 2^N \leq m_0$, or $$N \leq 2\log (np)+\log m_0.$$ Plugging in the value of $m_0$ we conclude that for sufficiently large $p$, $$ N \leq 2\log (np)+ \frac{2n+p}{2}+\log p+\log \left(\hat{R} \sqrt{p}+\left(\|W\|_{\infty}+1\right) \sqrt{n} \right).$$ This can be checked to contradict directly our hypothesis (\ref{eq:limit}) and the proof of the claim is complete.

Therefore using the decomposition of every $z \in \mathcal{L}_m$, equation (\ref{neces}) and the claim in the last paragraph it suffices to establish that w.p. $1-O\left(\frac{1}{np}\right)$ there is no triplet $(x_1,x_2,x_3)$ with
\begin{itemize}
\item[(a)]  $x_1 \in \mathbb{Z}^p, x_2,x_3 \in \mathbb{Z}^n$;

\item[(b)] $ \|x_1\|_2^2+\|x_3\|_2^2 \leq m_0$;

\item[(c)]  $Xx_1-\mathrm{Diag}_{n \times n}(Y)x_2-x_3 = 0$;

\item[(d)]  $\forall q \in \mathbb{Z}: x_1 \not = q \beta$.
\end{itemize}

We first claim that any such triplet $(x_1,x_2,x_3)$ satifies w.p. $1-O\left(\frac{1}{np}\right)$  \begin{equation*}\|x_2\|_{\infty} = O(\frac{ m_0n^2p^3}{\delta}).\end{equation*} To see this let $i=1,2,\ldots,n$ and denote by $X^{(i)}$ the $i$-th row of $X$. We have because of (c),
\begin{align*}
0=(Xx_1-\mathrm{Diag}_{n \times n}(Y)x_2-x_3)_i=\inner{X^{(i)}}{x_1}-Y_i(x_2)_i-(x_3)_i,
\end{align*}  
and therefore by triangle inequality 
\begin{equation}\label{eq:firstbound}
|Y_i (x_2)_i|=|\inner{X^{(i)}}{x_1}-(x_3)_i| \leq |\inner{X^{(i)}}{x_1}|+|(x_3)_i|.
\end{equation}  But observe that for all $i \in [n], \|X^{(i)}\|_{\infty} \leq \|X\|_{\infty} \leq (np)^22^N$ w.p. $1-O\left(\frac{1}{np}\right)$. Indeed using a union bound, Markov's inequality and our assumption on the distribution $\mathcal{D}$ of the entries of $X$, $$\mathbb{P}\left(\|X\|_{\infty} > (np)^22^N\right) \leq np \mathbb{P}\left(|X_{11}| >(np)^22^N\right) \leq \frac{1}{2^Nnp}\mathbb{E}[ |X_{11}|]  \leq \frac{C}{np}=O \left(\frac{1}{np}\right),$$ which establishes the result. Using this, Lemma \ref{lemmaY} and (\ref{eq:firstbound}) we conclude that for all $i \in [n]$ w.p. $1-O\left(\frac{1}{np}\right)$ $$|(x_2)_i|\frac{3}{2}\delta 2^N \leq (2^Np(np)^2+1)m_0$$ which in particular implies $$|(x_2)_i| \leq O(\frac{ m_0n^2p^3}{\delta}),$$ w.p. $1-O\left(\frac{1}{np}\right)$.

Now we claim that for any such triplet $(x_1,x_2,x_3)$ it also holds
\begin{equation}\label{firststep}
\mathbb{P}\left(Xx_1-\mathrm{Diag}_{n \times n}(Y)x_2-x_3 =0\right) \leq \frac{c^n}{2^{nN}}.\end{equation} 
To see this note that for any $i \in [n]$ if $X^{(i)}$ is the $i$-th row of $X$ because $Y=X\beta+W$ it holds $Y_i=\inner{X^{(i)}}{\beta}+W_i$. In particular, $Xx_1-\mathrm{Diag}_{n \times n}(Y)x_2-x_3 =0$ implies for all $i \in [n]$, 
\begin{align*}
&\inner{X^{(i)}}{x_1}-Y_i(x_2)_i=(x_3)_i\\
 \text{ or }&\inner{X^{(i)}}{x_1}-\left(\inner{X^{(i)}}{\beta}+W_i\right)(x_2)_i=(x_3)_i\\
 \text{ or } &\inner{X^{(i)}}{x_1-(x_2)_i \beta}=(x_3)_i-(x_2)_iW_i
\end{align*}  Hence using independence between rows of $X$, \begin{align}\label{constraint}
\mathbb{P}\left(Xx_1-\mathrm{Diag}_{n \times n}(Y)x_2-x_3 =0 \right) =\prod_{i=1}^{n}\mathbb{P}\left(\inner{X^{(i)}}{x_1-(x_2)_i \beta}=(x_3)_i-(x_2)_iW_i\right)
\end{align} But because of (d) for all $i$, $x_1-(x_2)_i \beta \not = 0$. In particular, $\inner{X^{(i)}}{x_1-(x_2)_i \beta}=(x_3)_i-(x_2)_iW_i$ constraints at least one of the entries of $X^{(i)}$ to get a specific value with respect to the rest of the elements of the row which has probability at most $\frac{c}{2^N}$ by the independence assumption on the entries of $X$. This observation with (\ref{constraint}) implies (\ref{firststep}).

 Now, we establish that indeed there are no such triplets, w.p. $1-O\left(\frac{1}{np}\right)$. Recall the standard fact that for any $r>0$ there are at most $O(r^n)$ vectors in $\mathbb{Z}^n$ with $L_{\infty}$-norm at most $r$. Using this, (\ref{firststep}) and a union bound over all the integer vectors $(x_1,x_2,x_3)$ with $\|x_1\|_2^2+\|x_3\|_2^2 \leq m_0$, $\|x_2\|_{\infty} = O(\frac{m_0n^2p^3}{\delta})$ we conclude that the probability that there exist a triplet $(x_1,x_2,x_3)$ satisfying (a), (b), (c), (d) is at most of the order $$(\frac{m_0n^2p^3}{\delta})^n m_0^{n+p} \left[ \frac{c^n}{2^{nN}}\right].$$ Plugging in the value of $m_0$ we conclude that the probability is at most of the order
\begin{align*}
\frac{2^{\frac{1}{2}(2n+p)^2+n \log (cn^2p^3)+n \log (\frac{1}{\delta})+(2+\log p)(2n+p)}\left[ \hat{R} \sqrt{p}+\left(\|W\|_{\infty}+1\right) \sqrt{n}\right]^{2n+p}}{2^{nN}}.\\
\end{align*} 
Now recalling that $\delta=\frac{1}{n^2p^2}$ we obtain $\log (\frac{1}{\delta}) = 2\log (np)$ and therefore the last bound becomes at most of the order 
\begin{equation*}
\frac{2^{\frac{1}{2}(2n+p)^2+5n \log (cnp)+(2+\log p)(2n+p)}\left[ \hat{R} \sqrt{p}+\left(\|W\|_{\infty}+1\right) \sqrt{n}\right]^{2n+p}}{2^{nN}}.
\end{equation*}  We claim that the last quantity is $O\left(\frac{1}{np}\right)$ because of our assumption (\ref{eq:limit}). Indeed the logarithm of the above quantity equals $$\frac{1}{2}(2n+p)\left(2n+p+4+2 \log p+2\log\left( \hat{R} \sqrt{p}+\left(\|W\|_{\infty}+1\right) \sqrt{n}\right) \right) +5n \log (cnp)-nN. $$ Using that $\hat{R} \geq 1$ this is upper bounded by $$ \frac{1}{2}(2n+p)\left(2n+p+10\log\left( R \sqrt{p}+\left(\|W\|_{\infty}+1\right)\sqrt{n}\right) \right)+5n \log (cnp)-nN $$ which by our assumption (\ref{eq:limit}) is indeed less than $-n \log (np) <-\log (np),$ implying the desired bound. This completes the proof of claim \ref{multiple}.
\end{proof}

Now we prove Theorem \ref{extention}. 
First with respect to time complexity, it suffices to analyze Step 5 and Step 6. For step 5 we have from \cite{lenstra1982factoring} that it runs in time polynomial in $n,p,\log \|A_m\|_{\infty}$ which indeed is polynomial in $n,p, N$ and $\log \hat{R}, \log \hat{W}$. For step 6, recall that the Euclid algorithm to compute the greatest common divisor of $p$ numbers with norm bounded by $\|\hat{z}\|_{\infty}$ takes time which is polynomial in $p,\log \|\hat{z}\|_{\infty}$. But from Claim \ref{multiple} we have that $\|\hat{z}\|_{\infty}<m$ and therefore the time complexity is polynomial in $p, \log m$ and therefore again polynomial in $n,p, N$ and $\log \hat{R}, \log \hat{W}$.

Finally we prove that the ELO algorithm outputs exactly $\beta^*$ w.p. $1-O\left(\frac{1}{np}\right)$. We obtain from Claim \ref{multiple} that $\hat{z}_{n+1:n+p}=q\beta$ for $\beta=\beta^*+Z$ and some $q \in \mathbb{Z}^*$ w.p. $1-O\left(\frac{1}{np}\right)$. We claim that the $g$ computed in Step 6 is this non-zero integer $q$ w.h.p. To see it notice that from Claim \ref{gcd} $\mathrm{gcd}(\beta)=1$ w.p. $1-\exp(-\Theta(p))=1-O\left(\frac{1}{np}\right)$ and therefore the $g$ computed in Step 6 satisfies w.p. $1-O\left(\frac{1}{np}\right),$ $$g=\mathrm{gcd}(\hat{z}_{n+1:n+p})=\mathrm{gcd}(q \beta)=q \mathrm{gcd}(\beta)=q.$$
Hence we obtain w.p. $1-O\left(\frac{1}{np}\right)$. $$\hat{z}_{n+1:n+p}=g\beta=g \left(\beta^*+Z \right)$$ or w.p. $1-O\left(\frac{1}{np}\right)$ $$\beta^*=\frac{1}{g}\hat{z}_{n+1:n+p}-Z,$$ which implies based on Step 7 and the fact that $g=q \not =0$ that indeed the output of the algorithm is $\beta^*$ w.p. $1-O\left(\frac{1}{np}\right)$. The proof of Theorem \ref{extention} is complete.
\end{proof}

\subsection{Proofs of Theorems \ref{main} and \ref{mainiid}}\label{ProofMain}
\subsection{Proof of Theorem \ref{main}}
\begin{proof}
We first analyze the algorithm with respect to time complexity. It suffices to analyze step 2 as step 1 runs clearly in polynomial time $N,n,p$. Step 2 runs the ELO algorithm. From Theorem \ref{extention} we obtain that the ELO algorithm terminates in polynomial time in $n,p,N, \log \left(\hat{Q}\hat{R}\right),\log \left(2\hat{Q} \left(2^N\hat{W}+\hat{R}p\right) \right)$. As the last quantity is indeed polynomial in $n,p,N,\log \hat{R},\log \hat{Q},\log \hat{W}$, we are done.

Now we prove that $\hat{\beta^*}=\beta^*$, w.p. $1-O\left(\frac{1}{np}\right)$. Notice that it suffices to show that the output of Step 3 of the LBR algorithm is exactly $\hat{Q}\beta^*$, as then step 4 gives $\hat{\beta^*}=\frac{Q\beta^*}{Q}=\beta^*$ w.p. $1-O\left(\frac{1}{np}\right)$.

We first establish that
\begin{equation}\label{eq:first}
2^{N}\hat{Q}Y_N=2^NX_{N} \hat{Q}\beta^*+W_0
\end{equation}  for some $W_0 \in \mathbb{Z}^n$ with $\|W_{0}\|_{\infty} + 1 \leq 2\hat{Q}\left(2^N\sigma+Rp\right)$. We have $Y=X\beta^*+W$, with $\|W\|_{\infty} \leq \sigma$. From the way $Y_N$ is defined,  $\|Y-Y_N \|_{\infty} \leq 2^{-N}$. Hence for $W'=W+Y_N-Y$ which satisfies $\|W'\|_{\infty} \leq 2^{-N}+\sigma$ we obtain  $$Y_N=X\beta^*+W'.$$ Similarly since  $\|X-X_N \|_{\infty} \leq 2^{-N}$ and $\|\beta^*\|_{\infty} \leq R$ we obtain $\|\left(X-X_N\right)\beta^* \|_{\infty} \leq 2^{-N}Rp$, and therefore for $W''=W'+\left(X-X_N\right)\beta^*$ which satisfies $\|W''\|_{\infty} \leq 2^{-N}+\sigma+2^{-N}rp$ we obtain, $$Y_N=X_N\beta^*+W''$$ or equivalently 
\begin{equation*}
2^NY_N=2^NX_N\beta^*+W''',
\end{equation*}
where $W''':=2^NW''$ which satisfies $\|W'''\|_{\infty} \leq 1+2^N\sigma+Rp$. Multiplying with $\hat{Q}$ we obtain 
\begin{equation*}
2^{N}\hat{Q}Y_N=2^NX_N \left(\hat{Q}\beta^*\right)+W_0,
\end{equation*}where $W_{0}:=\hat{Q}W'''$ which satisfies $\|W_{0}\|_{\infty} \leq \hat{Q}\left(1+2^N\sigma+Rp\right) \leq 2\hat{Q}\left(2^N\sigma+Rp\right)-1$. This establishes equation (\ref{eq:first}).

We now apply Theorem \ref{extention} for $Y$ our vector $\hat{Q}2^NY_N$, $X$ our vector $2^NX_N$, $\beta^*$ our vector $\hat{Q}\beta^*$, $W$ our vector $W_0$, $R$ our $\hat{Q}R$, $\hat{R}$ our $\hat{Q}\hat{R}$, $\hat{W}$ our quantity $2\hat{Q}\left(2^N\sigma+Rp\right)$ and finally $N$ our truncation level $N$.

  We fist check the assumption (1), (2), (3) of Theorem \ref{extention}. We start with assumption (1). From the definition of $X_N$ we have that $2^NX_{N} \in \mathbb{Z}^{n \times p}$  and that for all $i \in [n],j \in [p]$, $$|(2^NX_N)_{ij}| \leq 2^N|X_{ij}|.$$ Therefore for $C=\mathbb{E}[|X_{1,1}|]<\infty$ and arbitrary $i \in [n],j \in [p]$, $$\mathbb{E}[|(2^NX_N)_{ij}|] \leq 2^N\mathbb{E}[|X_{ij}|]=C2^N,$$as we wanted. Furthermore, if $f$ is the density function of the distribution $\mathcal{D}$ of the entries of $X$, recall $\|f\|_{\infty} \leq c$, by our hypothesis. Now observe for arbitrary $i \in [n],j \in [p]$, $$\mathbb{P}\left((2^NX_N)_{ij}=k\right)=\mathbb{P}\left(\frac{k}{2^N}\leq X_{ij} \leq \frac{k+1}{2^N}\right) =\int_{\frac{k}{2^N}}^{\frac{k+1}{2^N}} f(u)du \leq \|f\|_{\infty}\int_{\frac{k}{2^N}}^{\frac{k+1}{2^N}} du \leq \frac{c}{2^N}.$$ This completes the proof that $2^NX_N$ satisfies assumption (1) of Theorem \ref{extention}. For assumption (2), notice that $\hat{Q}\beta^*$ is integer valued, as $\hat{Q}$ is assumed to be a mutliple of $Q$ and $\beta^*$ satisfies $Q$-rationality. Furthermore clearly $$\|\hat{Q}\beta^*\|_{\infty} \leq \hat{Q}R.$$ For the noise level we have by (\ref{eq:first}) $W_0=2^{N}\hat{Q}Y_N-2^NX_{N} \hat{Q}\beta^*$ and therefore $W_0 \in \mathbb{Z}^n$ as all the quantities $ 2^{N}\hat{Q}Y_N$, $2^NX_{N}$ and  $\hat{Q}\beta^*$ are integer-valued. Finally, Assumption (3) follows exactly from equation (\ref{eq:first}). 
  
 Now we check the parameters assumptions of Theorem \ref{extention}. We clearly have $$\hat{Q}R \leq \hat{Q}\hat{R}$$ and $$\|W\|_{\infty} \leq 2\hat{Q}\left(2^N\sigma+Rp\right)=\hat{W}.$$  The last step consists of establishing the relation (\ref{eq:limit}) of Theorem \ref{main}. Plugging in our parameter choice it suffices to prove \begin{equation*}N>\frac{(2n+p)}{2}\left(2n+p+10\log \left( \hat{Q}\hat{R}\sqrt{p}+2\hat{Q}\left(2^{N}\sigma +Rp \right) \sqrt{n}\right) \right)+6n \log (\left(1+c\right)np).\end{equation*} Using that $\hat{Q}R\sqrt{p} \leq \hat{Q}\left(2^{N}\sigma +\hat{R}p \right) \sqrt{n}$ and $R \leq \hat{R}$ it suffices to show after elementary algebraic manipulations that \begin{equation*}N>\frac{(2n+p)}{2}\left(2n+p+10 \log 3+10\log \hat{Q}+  10\log \left( 2^{N}\sigma +\hat{R}p \right) +5 \log n\right) +6n \log (\left(1+c\right)np).\end{equation*}Using now that by elementary considerations $$\frac{(2n+p)}{2}\left(10 \log 3 +5 \log n\right) +4n \log (\left(1+c\right)np) < \frac{(2n+p)}{2}[20 \log (3\left(1+c\right) np)] \text{ for all } n \in \mathbb{Z}^+,$$ it suffices to show\begin{equation*}N>\frac{(2n+p)}{2}\left(2n+p+10\log \hat{Q}+  10\log \left( 2^{N}\sigma +\hat{R}p \right)+20  \log (3\left(1+c\right) np) \right), \end{equation*} which is exactly assumption (\ref{eq:limit1}). 

Hence, the proof that we can apply Theorem \ref{extention} is complete. Applying it we conclude that w.p. $1-O\left(\frac{1}{np}\right)$ the output of LBR algorithm at step 3 is $\hat{Q}\beta^*$, as we wanted. 
\end{proof}

\subsection{Proof of Theorem \ref{mainiid}}
By using a standard union bound and Markov inequality we have $$\mathbb{P}\left(\|W\|_{\infty} \leq  \sqrt{np}\sigma \right) \geq 1-\sum_{i=1}^n \mathbb{P}\left( |W_i| > \sqrt{np}\sigma  \right) \geq 1-n\frac{\mathbb{E}\left[W_1^2\right]}{np\sigma^2} \geq 1- \frac{1}{p}.$$Therefore, conditional on the high probability event $\|W\|_{\infty} \leq  \sqrt{np}\sigma$, we can apply Theorem \ref{main} with $\sqrt{np}\sigma$ instead of $\sigma$ and conclude the result.

\subsection{Proof of Theorem \ref{thm:lllpslqmain}}\label{sec:lllpslqmain-pf}
\begin{proof}
We begin by proving that for any fixed $i$, all integer relations for the vector $(Y_i,a_1,\dots,a_{\mathcal{R}})$ are contained in a one-dimensional discrete set. Rational independence results along this spirit are essential to several of our results.
\begin{lemma}\label{lemma:hyperplanerelation}
Let $Y=\langle X,\beta^*\rangle$ where $X\in\mathbb{Z}^p$ and $\beta^*\in \mathbb{R}^p$, such that, $\beta_i\in \mathcal{S}=\{a_1,\dots,a_{\mathcal{R}}\}$ a rationally independent set. Let $t=(t_0,t_1,\dots,t_{\mathcal{R}})$ be an integer relation for the vector, $(Y,a_1,\dots,a_{\mathcal{R}})$. Then, $t\in H$, where
$$
H=\{k(-1,\theta_1^*,\dots,\theta_{\mathcal{R}}^*):k\in\Z\setminus\{0\}\},
$$
with $\theta_i^* = \sum_{j:\beta^*_j =a_i} X_j$.
\end{lemma}
\begin{proof}
Observe that,  $Y$ is an integer combination of $\mathcal{S}$, concretely, $Y=\sum_{i=1}^{\mathcal{R}} \theta_i^* a_i$. Now, let $(t_0,\dots,t_{\mathcal{R}})$ be an arbitrary relation for the vector $(Y,a_1,\dots,a_{\mathcal{R}})$, namely, $t_0Y+ \sum_{i=1}^{\mathcal{R}} t_iY_i=0$. Observe that, $t_0\neq 0$, as otherwise we would have obtained that there is a non-trivial integer relation among $a_1,\dots,a_{\mathcal{R}}$, which is a contradiction. Now, with this, and $Y=\sum_{i=1}^{\mathcal{R}} \theta_i^* a_i $, we have,  $\sum_{i=1}^{\mathcal{R}}(t_i+t_0\theta_i^*)a_i=0$. Since the set $\mathcal{S}$ is rationally independent, this means the only integral relation among its elements is the trivial one, hence, $t_i+t_0\theta_i^*=0$, for every $i$. This means,
$$
t=(t_0,\dots,t_{\mathcal{R}})=-t_0(-1,\theta_1^*,\dots,\theta_{\mathcal{R}}^*) \in H,
$$
as claimed.\qedhere

\end{proof}
We now return to the proof of Theorem \ref{thm:lllpslqmain}. For each $i\in[n]$, define the coefficients $\theta^*_{ij}=\sum_{k:\beta_k^*=a_j}X_{ik}$, such that:
$$
Y_i = \sum_{k=1}^p X_{ik}\beta_k^* = \sum_{j=1}^{\mathcal{R}} \theta_{ij}^* a_j . 
$$

Lemma \ref{lemma:hyperplanerelation} implies that, for each one of the $n$ runs of the integer relation algorithm with inputs $(Y_i,a_1,\dots,a_{\mathcal{R}})$ with $i\in[n]$, the output of integer relation detection algorithm is guaranteed to be a multiple of the relation, $(-1,\theta_{i1}^*,\dots,\theta_{i\mathcal{R}}^*)$, and the corresponding multiplicity can be read off from the first component. Now, recalling Theorem \ref{thm:lllpslqmain}, we have that the termination time for a single run of the integer relation detection algorithm is $O(\mathcal{R}^3 +\mathcal{R}^2\log\|{\bf m}\|)$, where $\|{\bf m}\|$ is the norm of the smallest non-trivial integer relation ${\bf m}=(-1,\theta_{i1}^*,\dots,\theta_{i\mathcal{R}}^*)$. In order to upper bound this quantity, we will first upper bound $\theta_{ik}^*$. Let $X_i\in \Z^{1\times p}$ be the $i^{th}$ row of $X$. Notice that, for each $k\in [\mathcal{R}]$, $\theta_{ik}^* =\sum_{j:\beta_j^* = a_k}X_{ij} \implies |\theta_{ik}^*|\leq p\|X\|_\infty$. Now, by union bound and Markov's inequality,
$$
\mathbb{P}(\|X\|_\infty > p^2n 2^N) \leq pn \mathbb{P}(|X_{11}|>p^2n 2^N)\leq pn\frac{\mathbb{E}[|X_{11}|]}{p^2n2^N}\leq pn\frac{C2^N}{p^2n 2^N}=O(\frac1p).
$$
In particular, for all $i\in[n]$ and $j\in [\mathcal{R}]$, $|\theta_{ij}^*|\leq p^3n2^N$ with probability at least $1-O(1/p)$. Therefore, the norm of the smallest relation obeys, 
$$
\|{\bf m}\| \leq \sqrt{1+\sum_{k=1}^{\mathcal{R}} (\theta_{ik}^*)^2 }  \leq O(p^3(\mathcal{R}n)^{1/2}2^N),
$$
hence, with probability at least $1-O(1/p)$, the vector ${\bf m}$ is such that, $\log \|{\bf m}\|$ is at most polynomial in $N,p,\mathcal{R}$, and $n$. Note also that, we make $n$ calls to the integer relation oracle, each taking polynomial in $p,\mathcal{R},N,n$ many calls, and therefore,  we conclude that, the overall run time of integer relation detection step is polynomial in $p,\mathcal{R},N$ and $n$, with probability at least $1-O(1/p)$.

Now, we study the second half of the algorithm, where we make calls to the LLL lattice basis reduction oracle. Fix a $j\in[\mathcal{R}]$. We now show how to recover the entries of $\beta^*$, with value $a_j$. 
Suppose, $e^{(j)}$ is a binary vector, with $e^{(j)}_k = 1$ if and only if $\beta^*_k=a_j$, and $0$ otherwise, for $1\leq k\leq p$.  Observe that, $\Theta_j =  Xe^{(j)}$, hence, this problem is essentially a subset-sum problem, where we have access to $n$ linear measurements of the hidden binary vector $e^{(j)}$, through the mechanism, $\langle X_i,e^{(j)}\rangle$, where $X_i$ is the $i^{th}$ row of $X$. 
Consider the $(n+p)$-dimensional integer lattice $\Lambda_j$, generated by the columns of the matrix $A^j$, which we recall
$$
A^j = \begin{bmatrix} m\mathrm{Diag}_{n\times n}(\Theta_j) & -mX_{n\times p}\\ 0_{p\times n}& I_{p\times p} \end{bmatrix},
$$
with $m=p2^{\lceil \frac{n+p}{2}\rceil}$.
Observe that, $$A_j\begin{bmatrix}1_{n\times 1}\\e^{(j)}_{p\times 1}\end{bmatrix}=\begin{bmatrix}0_{n\times 1}\\ e^{(j)}_{p\times 1}\end{bmatrix}\implies \begin{bmatrix}0_{n\times 1}\\ e^{(j)}_{p\times 1}\end{bmatrix}\in\Lambda_j.$$
Hence, we have $\min_{z\in\Lambda_m,z\neq 0}\|z\|\leq \sqrt{p}$. Therefore, LLL lattice basis reduction algorithm, when called on $A_m$, returns a vector $\h{x}$ such that, $\|\h{x}\|\leq 
 \sqrt{p}2^{\frac{n+p}{2}}\triangleq m_0$. Note that $m_0<m$.

We now claim that, essentially all 'short' vectors of this lattice satisfy that their $(n+1):(n+p)$ coordinates are multiples of the hidden vector. More concretely, 
\begin{lemma}
\label{lemma:irrational-discrete-lll}
Define the set $\mathcal{F}^j$ via,
$$
\mathcal{F}^j=\{x\in \Lambda_j: \|x\|\leq m_0,x_{n+1:n+p}\neq ke^{(j)},\forall k\}.
$$
Then, $\mathbb{P}(\mathcal{F}\neq \varnothing)=o(1)$, where the probability is taken with respect to the randomness in $X$.
\end{lemma}

\begin{proof}
Note that, $x\in \mathcal{F}^j$ implies existence of $x_1\in\Z^{n\times 1}$ and $x_2\in \Z^{p\times 1}$ such that, 
$$
x=A_j\begin{bmatrix}x_1 \\ x_2\end{bmatrix}=\begin{bmatrix}m(\mathrm{diag}(\Theta_j)x_1 -Xx_2) \\x_2 \end{bmatrix}.
$$
Since $\|x\|\leq m_0$, we immediately obtain that $|(x_2)_i|\leq m_0$, for each $i\in [p]$.

Next, notice that, for every $i\in [n]$, $\theta_{ij}^*(x_1)_i = \langle X_i,x_2\rangle$. Indeed, if this is not the case for some $i_0\in [n]$, then, $x_{i_0}$ is a non-zero integer, divisible by $m$, and therefore, $m_0\geq \|x\|\geq |x_{i_0}|\geq m>m_0$, a contradiction. 
For any fixed $i$, and fixed $x_1$ and $x_2$, the probability that $\theta_{ij}^* (x_1)_i =\langle X_i,x_2\rangle$ is equal to,
$
\mathbb{P}\left(\sum_{j=1}^p X_{ij}e^{(j)}_i (x_1)_i=\sum_{j=1}^p X_{ij}(x_2)_j\right),
$ which is upper bounded by $c/2^N$. To see the last deduction, note that since $x\in\mathcal{F}^j$, there is an index $n_0\in\{n+1,\dots,n+p\}$ such that, $(x_2)_{n_0-n}\neq e_{n_0}^{(j)}(x_1)_i$ (note that $x_2\in\Z^p$ with coordinates $1,\dots,p$), and therefore, the probability above can be expressed as an event, involving $X_{n_0}$, which by conditioning on the rest is found to be upper bounded by $c/2^N$. Now, using independence, the probability that $\theta_{ij}^* (x_1)_i =\langle X_i,x_2\rangle$ for every $i$ is $c^n 2^{-nN}$.

We now show that, with probability at least $1-O(1/p^2)$, $|\theta_{ij}^*|\geq \frac{2^N}{cnp^2}$, for every $i\in[n]$. To see this, observe that, using union bound:
\begin{align*}
\mathbb{P}\left(\bigcup_{i=1}^n \{|\theta^*_{ij}|<\frac{2^N}{cnp^2}\}\right)\leq n \mathbb{P}(|\theta_{ij}^* <\frac{2^N}{cnp^2}|) = n\sum_{k\in \mathbb{Z}\cap [-\frac{2^N}{cnp^2} ,\frac{2^N}{cnp^2}]}\mathbb{P}\left(\theta_{ij}^*=k\right)
&\leq \frac{2n2^N}{cnp^2}\frac{c}{2^N}=O(\frac{1}{p^2}),
\end{align*}
since the probability that $\theta_{ij}^*$ takes a specific value is at most $c/2^N$. We now claim, $\|X\|_\infty \leq np^22^N$ with probability at least $1-O(1/p)$, where $\|X\|_\infty =\max_{i,j}|X_{ij}|$. To see this, observe that using union bound and Markov's inequality, we have:
$$
\mathbb{P}(\|X\|_\infty>np^22^N)\leq pn\mathbb{P}(|X_{11}|>np^22^N) = pn\frac{\mathbb{E}[|X_{11}|]}{np^22^N}=O(\frac{1}{p}),
$$
since $\mathbb{E}[|X_{11}|]\leq C2^N$ for some constant $C>0$. Now, recalling $\theta_{ij}^*(x_1)_i =\inner{X_i}{x_2}$ with high probability, $|(x_2)_i|\leq m_0$ for each $i\in[p]$; together with the (whp) bounds on $\|X\|_\infty$ and $|\theta_{ij}^*|$, we have with high probability:
$$
\frac{2^N}{cnp^2}|(x_1)_i|\leq |\theta_{ij}^*(x_1)_i|\leq m_0p^3n2^N \Rightarrow |(x_1)_i|\leq O(m_0p^5 n^2),
$$
for all $i$. Therefore, 
$$
\left|\left\{x_1\in\Z^{n\times 1},x_2\in \Z^{p\times 1}:A_j\begin{bmatrix}x_1\\x_2\end{bmatrix}\in \mathcal{F}\right\}\right|\leq O((2m_0p^5n^2+1)^n)O((2m_0+1)^p).
$$
Hence, 
$$
\mathbb{P}(\mathcal{F}\neq \varnothing) \leq \exp_2\left(n+p+n\log(n^2p)+\frac{n+p}{2}\log p+\frac{(n+p)^2}{2}-n\log c-nN\right)
$$
which is $o(1)$, due to the choice of the parameters, where $\exp_2(\gamma)=2^\gamma$.
\qedhere

\end{proof}
With the Lemma \ref{lemma:irrational-discrete-lll} at our disposal, we continue with the proof of Theorem \ref{thm:lllpslqmain}. Note that, due to the lemma, LLL algorithm, when run on the columns of the matrix $A_j$, is guaranteed to output a multiple of the vector $e^{(j)}$, indicating the indices $i$ such that $\beta_i^* =a_j$. $\mathcal{R}-$executions of this algorithm allow us to recover $\beta^*$ completely. The runtime of the LLL algorithm is polynomial in $n,p$ and $\log \|A_m\|_\infty$, and therefore, is polynomial in $n,N,p,\mathcal{R}$, keeping in mind that we make $\mathcal{R}$ calls to LLL oracle.

This, together with the previous discussion on the runtime of the integer relation detection step, establishes that, the overall runtime of the procedure is at most polynomial in $p,n,N,\mathcal{R}$. \qedhere

\end{proof}

\subsection{Proof of Theorem \ref{thm:irrational-continuous}}\label{sec:pf-irrational-continuous}
\begin{proof}
Define the set $\mathcal{L}$ by,
$$
\mathcal{L}=\{X_ia_j : 1\leq i\leq p,1\leq j \leq \mathcal{R}\}.
$$
\begin{lemma}
\label{lemma:uniqueness}
Let $X_i$'s be jointly continuous random variables, drawn from an arbitrary, continuous distribution $\mathcal{D}$; and $a_1,\dots,a_{\mathcal{R}}$ be rationally independent real numbers. Then,
$$
\mathbb{P}(\mathcal{L} \text{ is rationally independent})=1.
$$
\end{lemma}
\begin{proof}
Note that, the event $\mathcal{L}$ is rationally dependent is precisely equivalent to existence of a collection ${\bf v}\in \mathbb{Q}^{p\mathcal{R}}\setminus\{{\bf 0}\}$, indexed by $v_i^{(j)}$ with $i\in[p],j\in[\mathcal{R}]$, such that,
$$
\sum_{j=1}^{\mathcal{R}} a_j \left(\sum_{i=1}^p v_i^{(j)}X_i\right)=0\iff 0=\sum_{i=1}^p X_i\underbrace{\left(\sum_{j=1}^{\mathcal{R}} a_jv_i^{(j)}\right)}_{\triangleq \gamma_i^{{\bf v}}}=\sum_{i=1}^p X_i \gamma_i^{\bf v} = \langle X,\Gamma_{{\bf v}}\rangle.
$$
where $\Gamma_{{\bf v}}=\begin{bmatrix}\gamma_1^{\bf v} & \gamma_2^{\bf v} & \cdots & \gamma_p^{\bf v}\end{bmatrix}$.
Since ${\bf v}$ is not identically $0$, we conclude that, there exists a $v_{i_0}^{(j)}$ which is non-zero. In particular, using the fact that $a_1,\dots,a_{\mathcal{R}}$ are rationally independent, we have $\gamma_{i_0}^{\bf v}\neq 0$, hence, $\Gamma_{{\bf v}}\neq 0$. For each fixed  ${\bf v}\in\mathbb{Q}^{p\mathcal{R}}\setminus {{\bf 0}}$, we have $\mathbb{P}(\langle X,\Gamma_{{\bf v}}\rangle =0)=0$, as $X$ is a jointly continuous random vector, and $\Gamma_{{\bf v}}\neq {\bf 0}$. Thus, we establish, via countable additivity that
$$
\mathbb{P}(\mathcal{L} \text{ is rationally independent})\leq \sum_{{\bf v}\in \mathbb{Q}^{p\mathcal{R}}\setminus \{{\bf 0}\}} \mathbb{P}(\langle X,\Gamma_{{\bf v}}\rangle =0)=0.
$$
\qedhere

\end{proof}
With this, we now continue with the proof of Theorem \ref{thm:irrational-continuous}. Observe that,
$X_i\beta_i^* =\sum_{k=1}^{\mathcal{R}} \xi^{(i)}_k (X_ia_k)$, where $\xi^{(i)}\in\{0,1\}^{\mathcal{R}}$ is a binary vector, with all but one of its components equal to $0$, and the only component that is one is precisely the $k\in [\mathcal{R}]$ such that $\beta_i^*=a_k$. With this, we deduce that, $
Y=\sum_{i=1}^{p}\sum_{j=1}^{\mathcal{R}} X_ia_je_{ij}$,
where, $e_{ij}\in\{0,1\}$, and $e_{ij}=1$ if and only if $\beta_i^* = a_j$. In particular, there exists a non-trivial integral relation between $Y$ and the elements of $\mathcal{L}$. Now, let $(t_0,t_{ij}:i\in[p],j\in[\mathcal{R}])$ be any such relation, with $t_0Y + \sum_{i=1}^p \sum_{j=1}^{\mathcal{R}} X_ia_jt_{ij}=0$. Clearly, $t_0\neq 0$ with probability $1$, otherwise would contradict with the rational independence of $\mathcal{L}$, per Lemma \ref{lemma:uniqueness}. We have $\sum_{i=1}^p \sum_{j=1}^{\mathcal{R}} X_ia_j (t_0e_{ij}+t_{ij})=0$. From here, we proceed exactly in the same way as in the proof of Lemma \ref{lemma:hyperplanerelation}, and establish that $t_{ij}=-t_0e_{ij}$ for every $i\in[p],j\in[\mathcal{R}]$, as a consequence of Lemma \ref{lemma:uniqueness}. Namely, all such relations are again contained in a one-dimensional discrete set, spanned by the true relation. Now, letting the error event to be  $\mathcal{E}=\{\widehat{\beta^*}\neq \beta^*\}$, and  definining $\mathcal{E'}=\{\mathcal{L} \text{ is rationally independent}\}$, we have,  $\mathbb{P}(\mathcal{E})=\mathbb{P}(\mathcal{E}|\mathcal{E'})\mathbb{P}(\mathcal{E'})+\mathbb{P}(\mathcal{E}|(\mathcal{E'})^c)\mathbb{P}((\mathcal{E'})^c)=0$, since $\mathbb{P}(\mathcal{E}|\mathcal{E'})=0$ and $\mathbb{P}(\mathcal{E'})=0$. Hence, with probability $1$, the internal output ${\bf c}$ of the IHDR algorithm coincide with the true relation, from which, one can immediately obtain the value of $\beta_i^*$. 

Note that, the set $\mathcal{L}$ is generated in $O(p\mathcal{R})$ time, which is polynomial in both $p$ and $\mathcal{R}$. Finally, the norm $\|{\bf m}\|$ of the smallest relation, ${\bf m}$, between $Y$ and the elements of $\mathcal{L}$ is at most $O(\sqrt{p\mathcal{R}})$, hence, the integer relation detection algorithm runs in time $O((p\mathcal{R})^3+(p\mathcal{R}^2)\log(p\mathcal{R}))$, which is polynomial in $p$ and $\mathcal{R}$. With this, we conclude the proof. \qedhere
\end{proof}
\subsection{Remark}
One might simply ask, why this scheme cannot be used to carry out recovery for integer valued $X$. The reason is simple, note that in this case, $\mathcal{L}$ is no longer rationally independent, for instance, for any realization $X_1,\dots,X_p$; and $a_1,\dots,a_{\mathcal{R}}$, one may simply consider $-X_2(X_1a_1)+X_1(X_2a_1)+0\times\text{(the rest)}=0$, hence $(-X_2,X_1,0,\dots,0)$ is a non-trivial relation for the set $\mathcal{L}$. 

\subsection{Proof of Theorem \ref{thm:main-discrete-X}} \label{sec:pf-main-discrete-X}
We first multiply everywhere by $\h{Q}$, and arrive at,
$$\h{Q}Y_i = \sum_{j=1}^p\h{Q}X_{ij}\beta_j^*  = \sum_{j=1}^{\mathcal{R}} \theta_{ij}^* a_j +\sum_{j:\beta_j^*\in \Q}\h{Q}X_{ij}\beta_j^*\quad \forall i\in[n],
\vspace{-.1in}
$$
with $\theta_{ij}^* =\h{Q}\sum_{k:\beta_k^* =a_j}X_{ik}$. Note also that $\sum_{j:\beta_j^*\in \Q}\h{Q}X_{ij}\beta_j^* \in \Z$, as $Q$ divides $\h{Q}$, and whenever $\beta_j^*$ is rational, it is also $Q$-rational.

Next, one can transfer the proof of Lemma \ref{lemma:hyperplanerelation} to obtain that, the only integer relations for the vector, $(\h{Q}Y_i,a_1,\dots,a_{\mathcal{R}},1)$ are those, contained in one-dimensional discrete set $\mathcal{H}$, where
$$
\mathcal{H} = \{k(-1,\theta_{i1}^*,\dots,\theta_{i\mathcal{R}}^*, \sum_{j:\beta_j^*\in \Q}\h{Q}X_{ij}\beta_j^* ):k\in \Z^*\}.
$$
Thus, the integer relations found by the JIRSS step are, indeed, multiples of the true relation, for each $i$. The fact that integer relation algorithm runs in time polynomial in $n,\mathcal{R},N,\log\|W\|_\infty$, and $\log \h{Q}$, follows immediately, by inspecting and adapting the corresponding lines, in the proof of the Theorem \ref{thm:lllpslqmain}.

Next, we decompose $\beta^*$ according to, $\beta^* =\beta_I^*+\beta_R^*$, where,
$$
(\beta_I^*)_i=\left\{\begin{array}{ll}\beta_i^*,& \text{ if } (\beta_I^*)_i\notin \Q \\ 0,& \text{ otherwise,} \end{array}\right.
$$
and,
$$
(\beta_R^*)_i=\left\{\begin{array}{ll}\beta_i^*,& \text{ if } (\beta_R^*)_i\in \Q \\ 0,& \text{ otherwise.} \end{array}\right.
$$
Note that, $Y=X\beta^*+W =X\beta_I^* + X\beta_R^* + W$. The output $\widehat{\beta_1^*}$ of the JIRSS step, per Theorem \ref{thm:lllpslqmain}, satisfies with high probability, $\widehat{\beta_1^*}=\beta_I^*$, under the given parameter specifications, in which case we enjoy the condition of Theorem \ref{thm:lllpslqmain}. Hence, $\widetilde{Y} = Y-X\widehat{\beta_1^*}$, satisfies, with high probability, $\widetilde{Y}=X\beta_R^* +W$, under given  parameter specifications. 

Now, let $\widetilde{s}$ be the number of $0$ entries in $\widehat{\beta_1^*}$. Notice that, with high probability, $\widetilde{s}=s$, where $s=|\{i\in[p]:\beta_i^*\in \Q\}|$. We now let $\widetilde{\beta}$ to be the vector, obtained by erasing $\beta_i^*$, if $\widehat{\beta_1^*}=0$. The corresponding matrix, with erased columns, is $\widetilde{X}\in\Z^{n\times \widetilde{s}}$. Now, the problem of recovering the remainder of $\widehat{\beta^*}$ is simply the problem of inferring $\widetilde{\beta}$, from $\widetilde{Y}=\widetilde{X}\widetilde{\beta}+W$, where $\widetilde{\beta}\in \Q^s$, consists of $Q$-rational numbers. The algorithm \ref{algo:ELO}, per Theorem  \ref{extention}, with high probability, recovers $\widetilde{\beta}$, provided that the parameters involved obey the condition (\ref{eqn:main-discrete-rat-param}).

The overall polynomial run-time guarantee follows from the corresponding run time guarantees of the Algorithm \ref{algo:JIRSS} per Theorem \ref{thm:lllpslqmain} and Algorithm \ref{algo:ELO} per Theorem \ref{extention}.

\subsection{Proof of Theorem \ref{thm:main-cts-noiseless}}\label{sec:main-cts-noiseless}
Fix an $i\in[n]$, and express $\h{Q}Y_i=\sum_{j=1}^p \h{Q}X_{ij}\beta_j^*$ in the following way.
$$
\h{Q}Y_i = \sum_{j=1}^p \sum_{k=1}^{\mathcal{R}} X_{ij}a_k(\h{Q}e_{jk}) + \sum_{j:\beta_j^* \in \Q}\h{Q}X_{ij}\beta_j^*,
$$
where $e_{jk}\in\{0,1\}$ and $e_{jk}=1$ if and only if $\beta_j^*=a_k$, and is $0$, otherwise. This is nothing but a decomposition of $\h{Q}Y_i$, as an integral combination of the elements of the set $S_i$, where
$$
S_i=\{X_{ij}a_k:j\in[p],k\in[\mathcal{R}]\}\cup \{X_{ij}:j\in[p]\}.
$$
Note that, $|S_i|=p+p\mathcal{R}=O(p\mathcal{R})$, hence it can be generated in polynomial in $p$ and $\mathcal{R}$ time. We now establish that, with probability $1$, the elements of $S_i$ are rationally independent. This is proved along the same lines as in Lemma \ref{lemma:uniqueness}. In order to see this, define the event, $E_i=\{S_i \text{ is rationally independent}\}$. Note that, $E_i^c$ implies existence of a collection of rationals, not simultaneously zero and indexed for convenience by $\{q_{jk},r_j\in \Q:j\in [p],k\in[\mathcal{R}]\}$, such that,
$$
\sum_{j=1}^p \sum_{k=1}^{\mathcal{R}} X_{ij}a_kq_{jk}+\sum_{j=1}^p X_{ij}r_j = \sum_{j=1}^p X_{ij}\left(\sum_{k=1}^{\mathcal{R}} q_{jk}a_k+r_j\right)=\sum_{j=1}^p X_{ij}\gamma_j=0,
$$
where $\gamma_j = \sum_{k=1}^{\mathcal{R}} q_{jk}a_k+r_j$. Now, using the fact that $a_1,\dots,a_{\mathcal{R}},1$ are rationally independent, and $q_{jk},r_k$ are not all $0$, we deduce that, there exists a $j_0$, such that, $\gamma_{j_0}\neq 0$. Now, using union bound,
$$
\mathbb{P}(E_i^c)\leq \sum_{q_{jk},r_j \in \Q}\mathbb{P}\left(\sum_{j=1}^p X_{ij}\gamma_j=0\right)=0,
$$
since, for a fixed $i_0$, the probability of the event, $\left\{X_{i_0}=-\frac{1}{\gamma_{i_0}}\sum_{j\neq i_0}X_{ij}\gamma_j\right\}$ is $0$, which is obtained first by conditioning on all  random variables except $X_{i_0}$, and recalling that $X_i$'s follow a continuous distribution. 

Using this fact, and proceeding in exact same way, as in the proof of Lemma \ref{lemma:hyperplanerelation}, we obtain that integer relations for the vector consisting of $\h{Q}Y_i$, and the elements of $S_i$ are those, contained in $\mathcal{H}$, where
$$
\mathcal{H}=\left\{k(-1,e_{jk},s_j):k\in \Z^*\right\},
$$
with $e_{jk}=\h{Q}e_{jk}$, and $s_j=\h{Q}\beta_j^*$, if $\beta_j^*\in \Q$, and is $0$, otherwise. In particular, with probability $1$, $\widehat{\beta_1^*}$ indeed clashes with the irrational entries of $\beta^*$. Now, the remaining step constructs a modified observation model, after taking away the irrational entries of $\beta^*$, and the fact that LBR algorithm recovers an output, which with high probability clashes with the rational part of $\beta^*$ is a consequence of Theorem \ref{main}, with $W=0$.

We will now conclude the proof by examining the overall runtime. First, note that, the set $S_i$ can be generated in $O(p\mathcal{R})$ time, which is polynomial in $p$ and $\mathcal{R}$. Next, we know that, the integer relation detection algorithm with $k$ inputs $(x_1,\dots,x_k)$ recovers a relation in a time $O(k^3+k^2\log \|{\bf m}\|)$, where ${\bf m}$ is a  relation for the vector, $(x_1,\dots,x_k)$ with the smallest $\|{\bf m}\|$. In our case, the input set is of cardinality $|S_i|+1=O(p\mathcal{R})$, and the smallest relation is $(-1,e_{jk},s_j:j\in[p],k\in[\mathcal{R}])$, which satisfies $\|m\|\leq \sqrt{\h{Q}^2p\mathcal{R}+p\h{Q}^2 \widetilde{R}^2}$, and thus, the overall runtime is $O(p^3\mathcal{R}^3 +p^2\mathcal{R}^2 \log (\h{Q}+\log p +\log(\mathcal{R}\widetilde{R})))$, that is at most polynomial in $p,\mathcal{R}$, and $\log \widehat{Q}$. Finally, the LLL algorithm runs in time polynomial in $n,s,N,\log \h{R},\log \h{Q}$, and therefore, the overall runtime is indeed polynomial in $n,p,N,s,\log \h{R},\log\h{Q}$, and $\mathcal{R}$, the cardinality of the irrational part of the support of $\beta^*$.

\subsection{Proof of Theorem \ref{thm:cts-noiseless-ira-only}}\label{sec:pf-cts-noiseless-ira-only}
\begin{proof}
Quite analogous to the previous proof, we first note that,
$$
\widehat{Q}Y = \sum_{j:\beta_j^*\in\mathbb{Q}} X_j(\widehat{Q}\beta_j^*) + \sum_{i=1}^p \sum_{j=1}^{\mathcal{R}} X_i a_j (\widehat{Q}e_{ij}),
$$
where $e_{ij}=1$ if and only if $\beta_i^* = a_j$, and is zero, otherwise. Observe that, for any $j$ with $\beta_j^*\in\mathbb{Q}$, the $Q$-rationality assumption, together with the fact that $\widehat{Q}$ is divisible by $Q$ yield that $\widehat{Q}\beta_j^*\in\mathbb{Z}$. In particular, there is a an integer relation for the vector,
$$
(\widehat{Q}Y, X_i a_j , X_i : i\in[p],j\in[\mathcal{R}]).
$$
Namely, $\widehat{Q}Y = \sum_{i=1}^p X_i \theta_i^* + \sum_{i=1}^p\sum_{j=1}^{\mathcal{R}} X_i a_j (\widehat{Q}e_{ij})$ with $\theta_i^*=0$ if $\beta_i^*\notin\mathbb{Q}$, and $\theta_i^* = \widehat{Q}\beta_i^*$ if $\beta_i^*\in\mathbb{Q}$. We now claim, similar to the previous results, that the set, $\{X_ia_j : i\in[p],j\in[\mathcal{R}]\} \cup \{X_i:i\in[p]\}$ is rationally independent, with probability one. To see this, let $(r_{ij},q_i:i\in[p],j\in[\mathcal{R}])\in \mathbb{Q}^{p\mathcal{R}+p}$ be a $p(\mathcal{R}+1)-$tuple of rationals, not all zero. We first establish,
$$
\mathbb{P}\left(\sum_{i=1}^p\sum_{j=1}^{\mathcal{R}} X_i a_j r_{ij}+\sum_{i=1}^p X_i q_i=0\right)=0.
$$
Let $\gamma_i = \sum_{j=1}^{\mathcal{R}} r_{ij}a_j + q_i$ for $1\leq i\leq p$. Note that, since $\{a_1,\dots,a_{\mathcal{R}},1\}$ is a rationally independent set by Assumption \ref{def:mixed_support}, we then get $\gamma_i = 0 \Rightarrow r_{ij},q_i = 0$ for any $j\in[\mathcal{R}]$. In particular, if $\gamma_i=0$ for every $i$, we then deduce immediately that $r_{ij},q_i=0$ for every $i\in[p]$ and $j\in[\mathcal{R}]$, which contradicts with the choice of this tuple. From here, we then get the vector $\Gamma=(\gamma_1,\dots,\gamma_p)^T \in\R^p$ is not identically zero. But now, $\mathbb{P}(\inner{X}{\Gamma}=0)=0$ immediately, due to joint continuity. Union bound over all such $p(\mathcal{R}+1)-$tuples of rationals then establish the  desired claim.

Now, assume in the remainder we condition on this event. Let $(m_0,m_{ij},n_i : i\in[p],j\in[\mathcal{R}])$ be a (non-zero) integer relation for the vector $(\widehat{Q}Y,X_i a_j, X_i : i\in[p],j\in[\mathcal{R}])$. We then have,
$$
0= m_0\widehat{Q}Y + \sum_{i=1}^p \sum_{j=1}^{\mathcal{R}} m_{ij}X_i a_j + \sum_{i=1}^p X_i n_i  = \sum_{i=1}^p X_i (m_0\theta_i^*+n_i) + \sum_{i=1}^p \sum_{j=1}^{\mathcal{R}} X_i a_j ( m_0\widehat{Q}e_{ij} + m_{ij}).
$$
Since the vector, $(X_i a_j,X_i : i\in[p],j\in[\mathcal{R}])$ is rationally independent by conditioning,  we then immediately get, $m_{ij}=-m_0 \widehat{Q}e_{ij}$ and $n_i = -m_0 \theta_i^*$. Namely, IRA indeed recovers $\beta^*\in\R^p$, since by inspecting the relation coefficients, the recover is immediate.

Finally, we study the runtime of the algorithm. Recall that, the smallest such relation is of form $(-1,\theta_i^*,\widehat{Q}e_{ij}:i\in[p],j\in[\mathcal{R}])$. Each coordinate of this vector is clearly upper bounded by $\widehat{Q}\widehat{R}$. From here, we get that the smallest relation has  norm which is at most $O(\widehat{Q}\widehat{R}(p\mathcal{R})^{1/2})$. Finally, using Theorem \ref{thm:pslq_main_thm}, we get that the overall runtime is at most $O(p^3\mathcal{R}^3 + p^2\mathcal{R}^2 \log(\widehat{Q}\widehat{R}(p\mathcal{R})^{1/2}))$, which is ${\rm poly}(p,\mathcal{R},\log \widehat{Q},\log\widehat{Q})$, as claimed. 
\end{proof}

\subsection{Proof of Theorem \ref{thm:phase_retrieval} }\label{sec:phase_retrieval}
\begin{proof}

We first note that, $Y^2 = \sum_{i=1}^p X_i^2 |\beta^*_i|^2 + \sum_{1\leq i<j\leq p}X_iX_j((\beta_i^*)^H\beta_j^*+\beta_i^*(\beta_j^*)^H)$. In particular, it is not hard to see that, $Y^2=\sum_{d=1}^{\mathcal{R}} \theta_d^* |a_d|^2 + \sum_{1\leq i<j\leq \mathcal{R}}\theta_{ij}^*(a_i^H a_j+a_ia_j^H)$, for some integers $\theta_1^*,\dots,\theta_{\mathcal{R}}^*\in \Z$, and $\theta_{ij}^* \in \Z$ for $1\leq i<j\leq \mathcal{R}$. In particular, if ${\bf t}=(t_0,t_d,t_{ij}:1\leq d\leq \mathcal{R},1\leq i<j\leq \mathcal{R})$ is an integer relation for the vector, consisting of $Y^2$, and the entries of $\mathcal{S'}=\{|a_d|^2:d\in[\mathcal{R}]\}\cup\{a_i^Ha_j+a_ia_j^H:1\leq i<j\leq \mathcal{R}\}$, then one may proceed, in a similar way, as in proof of Lemma \ref{lemma:hyperplanerelation} and establish that, ${\bf t}=-t_0(-1,\theta_d^*,\theta_{ij}^*:1\leq d\leq \mathcal{R},1\leq i<j\leq \mathcal{R})=-t_0{\bf t'}$, namely, all relations contained in a one-dimensional (discrete) set, spanned by the vector, ${\bf t'}=(-1,\theta_d^*,\theta_{ij}^*:1\leq d\leq \mathcal{R},1\leq i<j\leq \mathcal{R})\in\mathbb{Z}^{\mathcal{R}(\mathcal{R}+1)/2+1}$. In particular, the integer relation algorithm, with inputs $Y^2$ and the entries of $\mathcal{S'}$ will terminate with $(-1,\theta_d^*,\theta_{ij}^*:1\leq d\leq \mathcal{R},1\leq i<j\leq \mathcal{R})$.  

We now show, how to decode $\beta^*$, using ${\bf t'}$, and the solution of a certain subset-sum problem. Observe that, $X_iX_j$ contributes to $\theta^*_{k\ell}$, corresponding to $a_k^Ha_{\ell}+a_ka_{\ell}^H$, if and only if, $\{\beta_i^*,\beta_j^*\}=\{a_k,a_\ell\}$. With this, we observe that, for every $1\leq k<\ell \leq \mathcal{R}$, and $1\leq i<j\leq p$, there exists binary variables $\xi_{ij}^{(k,\ell)}\in\{0,1\}$, such that, the following holds: $\theta^*_{k\ell}=\sum_{1\leq i<j\leq p}X_iX_j\xi_{ij}^{(k,\ell)}$. Namely, the coefficients, $\theta_{k\ell}^*$ are subset-sums of $\{X_iX_j:1\leq i<j\leq p\}$. The binary variables, $\xi_{ij}^{(k,\ell)}$ can be recovered in polynomial in $p$ and $\mathcal{R}$ many (bit) operations, using LLL algorithm, and this is isolated as a separate result, in Proposition \ref{lemma:extended-frieze}. 
Taking this to be granted, namely,   $\xi_{ij}^{(k,\ell)}$ can be recovered, we now show how to use this information to decode $\beta^*$. For convenience of notation, assume $\xi_{ij}^{(k,\ell)}=\xi_{ji}^{(k,\ell)}$.

If $|\{d:\theta_d^*\neq 0\}|=1$, then $\beta_i^*=a_{d_0}$ for every $i\in[p]$, where $d_0$ is the unique index with $\theta_{d_0}^* \neq 0$. If, $|\{d:\theta_d^*\neq 0\}|=2$, then there is a pair $k,\ell$ of indices, such that, $\beta_i^*\in \{a_k,a_\ell\}$, for all $i\in[p]$. Now, we focus on $\theta_{k\ell}^* = \sum_{i<j}X_iX_j\xi_{ij}^{(k,\ell)}$. Note that, if $\beta_i^*=\beta_j^*$, then $\xi_{ij}^{(k,\ell)}=0$. Start from an arbitrary index $i_0$. Set $\beta_{i_0}^*=+$. Now, for any index, $i<i_0$, if $\xi_{i_0 i}^{(k,\ell)}=1$, then set $\beta_{i}^*=-$. Else, set it to $+$. At the end, it is either the case that,  all entries labeled with $+$ are $a_k$, and those labeled with $-$ are $a_\ell$; or vice versa. This can be verified, using the information $Y^2$, clearly, in polynomial time. Finally, suppose that, $|\{d:\theta_d^*\neq 0\}|\geq 3$. Fix an $i\in[p]$, and suppose the goal is to decode $\beta_i^*$. Find, in polynomial time, by brute force search, two pairs $(k,\ell)\neq (k',\ell')$ with $1\leq k<\ell\leq \mathcal{R}$, and $1\leq k'<\ell'\leq \mathcal{R}$; and indices $j$ and $j'$, such that, $\xi_{ij}^{(k,\ell)}= \xi_{ij'}^{(k,\ell')}=1$.
Note that, since $\beta^*$ consists at least of three different values from the set $\mathcal{S}$, such pairs and indices indeed exist, and that, $|\{k,\ell\}\cap \{k',\ell'\}|=1$. Observe  now that, $\{\beta_i^*,\beta_j^*\}=\{a_k,a_\ell\}$ and $\{\beta_i^*,\beta_{j'}^*\}=\{a_{k'},a_{\ell'}\}$. Hence, $\beta_i^*=a_{\{k,\ell\}\cap \{k',\ell'\}}$, and  found in polynomial in $p,\mathcal{R}$ many operations. Repeating this $p$ times, we finish decoding $\beta^*$.

We now show that the overall run-time of this protocol is polynomial. For this, it suffices to show, both LLL and  and integer relation detection steps run in polynomial time. The runtime of LLL step will be established in the proof of Proposition \ref{lemma:extended-frieze}. Now, we focus on the integer relation detection step. First, note that generating $\mathcal{S'}$ takes at most $O(\mathcal{R}^2)$ arithmetic operations on reals. Now, we focus on bounding the norm of ${\bf t'}$, which has the smallest $\|{\bf t'}\|$.
Observe that, 
$$
\mathbb{P}(\|X\|_\infty > p^22^N)\leq p\mathbb{P}(|X_1|>p^22^N)\leq p\frac{\mathbb{E}[|X_1|]}{p^22^N} \leq O\left(\frac{1}{p}\right),
$$
using union bound and Markov inequality. Thus, $\|X\|_\infty \leq p^2 2^N$, with probability at least $1-O(1/p)$. Next, let $\widetilde{X}$ be a vector of dimension $\binom{p}{2}$, consisting of elements, $X_iX_j$, for $1\leq i<j\leq p$. Now, 
$$
\mathbb{P}(\|Y\|_\infty > p^3 2^{2N})\leq p^2 \mathbb{P}(|X_1X_2|>p^3 2^{2N}) \leq p^2\frac{\mathbb{E}[|X_1X_2|]}{p^32^{2N}} \leq O\left(\frac1p\right),
$$
using union bound, Markov inequality, and independence of $X_1$ and $X_2$. Hence, with probability at least $1-O(1/p)$, $\|Y\|_\infty\leq p^32^{2N}$. Recalling that, the IRA terminates with (a multiple of) ${\bf t'}$ in time at most $O(T^3+T^2\log \|{\bf t'}\|)$, which is at most polynomial in $\mathcal{R},N,p$, where $T=O(\mathcal{R}^2)$ is the size of the input to the problem. Finally, the overall run time is polynomial in $p$ and $\mathcal{R}$, since the integer relation detection step runs in time polynomial in $p$ and $\mathcal{R}$, and we make polynomial-in-$\mathcal{R}$ many calls to the LLL lattice basis reduction oracle, each of which returns an answer in polynomial-in-$p$ time. The proof is completed. \qedhere


\end{proof}

\subsection{Proof of Proposition \ref{lemma:extended-frieze}}\label{sec:extended-frieze}
\begin{proof}

We will now study the modified subset-sum problem with dependent inputs, where $\theta=\sum_{1\leq i<j\leq p}X_iX_j\xi_{ij}$ with $\xi_{ij}\in\{0,1\}$. Let $L=\binom{p}{2}$, and $Y_1,\dots,Y_L$ be an enumeration of $\{X_iX_j:1\leq i<j\leq p\}$. Rewrite the problem as $\theta=\sum_{i=1}^L Y_i\xi_i$ with $\xi_i\in\{0,1\}$, where the algorithmic goal is to recover $\xi\in \{0,1\}^L$, using $\theta$ and $Y_1,\dots,Y_L$. Set $m=p^2 2^{\lceil p^2/4\rceil}$. We consider $L+1$ dimensional integer lattice, $\Lambda \subset \Z^{L+1}$, generated by the vectors $b_0,b_1,\dots,b_L$, defined as, $b_0=(m\theta,0_{1\times L})$, and for every $i=1,2,\dots,L$, $b_i=(-mY_i,e_i)$, where $e_i$ is the $i^{th}$ element of the standard basis of $L$-dimensional Euclidean space, $\R^L$. 
Namely, this is the lattice generated by the columns of the following matrix:
$$
\begin{bmatrix}
m\theta &-m\widetilde{Y}_{1\times L} \\ 0_{L\times 1} & I_{L\times L}
\end{bmatrix},
$$
with $\widetilde{Y}=(Y_1,\dots,Y_L)$. Observe that, $b_0+\sum_{i=1}^L \xi_i b_i = (0,\xi_1,\dots,\xi_L)=\Xi\in \Lambda$. In particular, we have the bound $\min_{x\in\Lambda, x\neq 0}\|x\| \leq \|\Xi\| < p$, regarding the norm of the shortest non-zero vector of this lattice. Therefore, running LLL algorithm on the lattice $\Lambda$ yields a vector, $\widehat{x}$ such that $\|\widehat{x}\| \leq p2^{L/2}\triangleq m_0$. We will now show that, the 'short' vectors of this lattice are essentially integer multiples of $\Xi$, which will therefore establish that $\Xi$ can be obtained from the LLL output, keeping in mind that $\Xi$ is a binary vector. 

Define the set, $E\subset \Lambda$ as,
$$
E=\left\{x\in\Lambda : \|x\|\leq m_0,x\neq k\Xi ,\forall k\in\Z\right\}.
$$
We claim that, $\mathbb{P}(E\neq\varnothing)=o(1)$, as $p\to\infty$. In particular, all 'short' vectors of this lattice, namely those whose norm is at most $m_0$, are multiples of $\Xi$. Since the LLL output is guaranteed to have norm at most $m_0$, this claim establishes that with high probability, LLL output recovers the desired vector, $\Xi$. 

We now prove the claim. Suppose $E\neq \varnothing$, and let $x=(x_0,\dots,x_L)\in E$. Then, there exists an integer $x_0'$, such that, $x=x_0'b_0 +\sum_{i=1}^L x_i b_i$. Now, observe that, if $\theta x_0'-\sum_{i=1}^L Y_ix_i \neq 0$, then $\left|\theta x_0'-\sum_{i=1}^L Y_ix_i\right|\geq 1$, since it is an integer. Therefore, $|x_0|\geq m$, hence, $\|x\|\geq |x_0|\geq m>m_0$, a contradiction. Therefore, we deduce $x_0'\theta = \sum_{i=1}^L Y_ix_i$. As in the proof of outlined by Frieze, the high level idea of the proof is to use union bound, over all tuples $(x_0',x_1,\dots,x_L)$ such that $x_0'b_0+\sum_{i=1}^L x_ib_i\in E$, together with the probability of a certain event, for every fixed tuple. To that end, we begin by bounding the cardinality of allowed such vectors. Note that, $x\in E$ implies $\|x\|\leq m_0$, and therefore, $|x_i|\leq m_0$ for every $i\in[L]$. 

We now turn out attention to bounding $|x_0'|$, which requires a separate analysis. Notice, without loss of generality, we may assume that $\theta\geq \frac{1}{2}\sum_{i=1}^L Y_i$. Indeed, if this is not the case, we consider an alternative problem, $\widehat{\theta}=\sum_{i=1}^L Y_i-\theta = \sum_{i=1}^L Y_i(1-\xi_i)$. Observe that, if $\theta <\frac{1}{2}\sum_{i=1}^N Y_i$, we would have, $\widehat{\theta}>\frac{1}{2}\sum_{i=1}^L Y_i$, and thus we can equivalently consider the problem of recovering $\xi'$, where $\h{\theta}=\sum_{i=1}^L Y_i\xi_i'$, with inputs $\h{\theta},Y_1,\dots,Y_L$, and $\xi_i' = 1-\xi_i$. Under this assumption, we obtain 
$$
x_0'\theta = \sum_{i=1}^L Y_i x_i \implies |x_0'|\cdot \theta \leq \sum_{i=1}^L Y_i|x_i|\leq m_0\sum_{i=1}^L Y_i \implies |x_0'|\leq 2m_0.
$$
Thus, $|x_0'|\leq 2m_0$. Hence, the set, 
$$
\mathcal{S}=\{(x_0',x_1,\dots,x_L):x_0'b_0+\sum_{i=1}^L x_ib_i \in E\}
$$ has cardinality at most $(4m_0+1)\cdot (2m_0+1)^L \leq O(2^{p^4/8+o(p^4)})$. 

Next, for any fixed $(x_0',x_1,\dots,x_L)$ such that the corresponding vector $x$ is contained in $E$, we will study the probability of the event, $\left\{x_0'\theta = \sum_{i=1}^L Y_ix_i \right\}$ which is the event that $\left\{\sum_{i=1}^L Y_i\xi_i x_0'=\sum_{i=1}^L Y_i x_i \right\}$. Since $x\in E$, we know that $x\neq k\Xi$ for every $k\in\Z$, therefore, there exists an $i_0$ such that $\xi_{i_0}x_0'\neq x_{i_0}'$. Note that, the smallest $i_0$, for which $\xi_{i_0}x_0'\neq x_{i_0}$ is not $0$, since $x_0=0$, as we have already proven, thanks to $\|x\|\leq m_0$. 

Let $Y_{i_0}=X_nX_m$ for some $n<m$. Now, for this fixed $n$, we introduce the following notation: We say $k\sim n$, if $Y_k$ is of form $X_nX_i$, where $i\in[p]$ is an index. That is,
$$
k\sim n\iff k\in\{k_0:\exists i\in[p], Y_{k_0}=X_nX_i,1\leq k_0\leq L\}.
$$
By definition, $i_0 \sim n$. Moreover, say $k\not\sim n$ if $Y_k=X_{i'}X_{i''}$, where $n\notin\{i',i''\}$. Also, denote by $X_{\sim i}$ the collection, $(X_1,\dots,X_{i-1},X_{i+1},\dots,X_p)$. 

Now, define 
$$
A=\sum_{1\leq k\leq L:k\sim n}\frac{Y_k}{X_n}(\xi_k x_0'-x_k)\quad\text{and}\quad B=\sum_{1\leq k\leq N:k\not\sim n}Y_k(\xi_k x_0'-x_k).
$$
Observe that, the source of randomness in $B$ is $X_{\sim n}$, and thus, $X_n$ and $B$ are independent. With this, we observe that,
$$
\left\{\sum_{i=1}^L Y_i\xi_i x_0'=\sum_{i=1}^L Y_i x_i\right\}\implies\left\{AX_n+B=0\right\}.
$$
Now, define the events:
$$
\mathcal{E} \triangleq \{A=0\}, \quad \mathcal{F}\triangleq \{AX_n+B=0\},\quad\text{and}\quad A_{v} = \{X_{\sim n}=v\}
$$
Focusing on $i_0$ with $Y_{i_0}=X_nX_m$, we have:
$$
\mathcal{E}=\left\{X_m(\xi_{i_0}-x_{i_0}) = -\sum_{k\sim n,k\neq i_0}\frac{Y_k}{X_n}(\xi_k x_0'-x_{i_0})\right\}.
$$
Due to the fact that $X_1,\dots,X_p$ are iid,  $\mathbb{P}(\mathcal{E})=O(2^{-N})$. Using these,
\begin{align*}
\mathbb{P}(\mathcal{F})&= \mathbb{P}(\mathcal{F}\cap \mathcal{E})+\mathbb{P}(\mathcal{F}\cap \mathcal{E}^c) \\    &\leq \sum_{v}\underbrace{\mathbb{P}(\mathcal{F}|\mathcal{E}^c\cap A_v)}_{\leq O(2^{-N})}\underbrace{\mathbb{P}(\mathcal{E}^c \cap A_v)}_{\leq \mathbb{P}(A_v)}  + \mathbb{P}(\mathcal{E})\\
&\leqslant O(2^{-N})\sum_{v}\mathbb{P}(A_v)+O(2^{-N}) = O(2^{-N}),
\end{align*}
using the fact that conditional on $\mathcal{E}^c\cap A_v$, the event $\mathcal{F}$ is the event that $X_n$ takes a unique value, which happens with probability $O(2^{-N})$.

Hence, for any fixed $(x_0',x_1,\dots,x_L)$, the event $\mathcal{F}$ holds with probability at most $O(2^{-N})$. Now, recalling that, $|\{(x_0',x_1,\dots,x_L):x_0'b_0 +\sum_{i=1}^L x_ib_i\in E\}|=O(2^{p^4/8+o(p^4)})$,
a union bound over all admissible $(L+1)-$tuples $(x_0',x_1,\dots,x_L)$ yields that
$$
\mathbb{P}(E\neq \varnothing) \leq O(2^{-(N-p^4/8)+o(p^4)}
)=o_p(1),
$$
since $N\geqslant (1/8+\epsilon)p^4$, by assumption of the theorem. \qedhere
\end{proof}

\subsection{Proof of Theorem \ref{thm:phase-retrieval-cts}}\label{sec:phase-retrieval-cts}
\begin{proof}
Note that, $Y^2=\sum_{i=1}^p X_i^2|\beta_i^*|^2 + \sum_{1\leq i<j\leq p}X_iX_j((\beta_i^*)^H \beta_j^* + \beta_i^* (\beta_j^*)^H)$ yields that, $Y^2$ is an integer combination of the elements of set $\mathcal{L}$, where $\mathcal{L}=\mathcal{S}_1\cup\mathcal{S}_2\cup\mathcal{S}_3$, such that $\mathcal{S}_1 = \{X_i^2|a_k|^2:i\in[p],k\in[\mathcal{R}]\}$, $\mathcal{S}_2=\{X_iX_j|a_k|^2:1\leq i<j\leq p,k\in[\mathcal{R}]\}$,
and $\mathcal{S}_3=\{X_iX_j(a_k^Ha_\ell+a_ka_\ell^H):1\leq i<j\leq p,1\leq k<\ell\leq \mathcal{R}\}$; namely,  $Y^2=\sum_{a\in \mathcal{L}}a\theta_a^* $ for some coefficients $\theta_a^*$, where for each $a\in \mathcal{L}$, $\theta_a^*\in\mathbb{Z}$. We now establish that $\mathcal{L}$ is rationally independent.
\begin{lemma}
\label{lemma:ell-rat-indep}
Let $\mathcal{L}$ be defined as above. Then, $\mathbb{P}(\mathcal{L}\text{ is rationally independent})=1$, where the probability is taken with respect to joint distribution of $\{X_i\}_{i=1}^p$.
\end{lemma}
\begin{proof} {(of Lemma \ref{lemma:ell-rat-indep})}
We begin by the following simple auxiliary result, that will turn out to be useful in the proof:
\begin{theorem}\rm{\cite{caron2005zero}}\label{thm:auxiliary}
Let $\ell$ be an arbitrary positive integer; and $P:\R^\ell\to \R$ be a polynomial. Then, either $P$ is identically $0$, or $\{x\in \R^\ell:P(x)=0\}$ has zero Lebesgue measure, namely, $P(x)$ is non-zero almost everywhere.
\end{theorem}
As a consequence of Theorem \ref{thm:auxiliary}, observe that if $X=(X_1,\dots,X_p)\in \R^p$ a jointly continuous random vector with density $f$; $P:\R^p\to \R$ is a polynomial where there exists $(x_1,\dots,x_p)\in\R^p$ with $P(x_1,\dots,x_p)\neq 0$; and  $S=\{(x_1,\dots,x_p)\in \R^p:P(x_1,\dots,x_p)=0\}$, then we have:
$$
\mathbb{P}(P(X)=0)=\int_{S\subset \R^p}f(x_1,\dots,x_p)d\lambda(x_1,\cdots,x_p)=0,
$$
where $\lambda$ is the ($p$-dimensional) Lebesgue measure; since the set $S$ has zero Lebesgue measure. In particular, if $P$ is a polynomial that is non-vanishing, then $P(X)$ is non-zero almost everywhere, for any jointly continuous $X\in\R^p$. 

Equipped with this, we now return to the proof. For an arbitrary subset $S\subseteq [p]$, denote $X_{\sim S}=\{X_i:i\in [p]\setminus S\}$ ; and when $S$ is a single element set, $S=\{i\}$, let us use $X_{\sim i}$, in place of $X_{\sim\{i\}}$. Now, note that, the event, $\mathcal{L}$ is not rationally independent implies the existence of a collection $\mathcal{Q}$ of rationals, not all zero:
$$
\mathcal{Q}=\{q_{ik}:i\in[p],k\in[\mathcal{R}]\}\cup \{r_{ijk}:1\leq i<j\leq p,k\in [\mathcal{R}]\}\cup \{t_{ijk\ell}:1\leq i<j\leq p,1\leq k<\ell\leq \mathcal{R}\}
$$
such that,
$$
\sum_{i=1}^p\sum_{k=1}^{\mathcal{R}} X_i^2|a_k|^2 q_{ik}+\sum_{1\leq i<j\leq p}\sum_{k=1}^{\mathcal{R}} X_iX_j|a_k|^2 r_{ijk} + \sum_{1\leq i<j\leq p}\sum_{1\leq k<\ell\leq \mathcal{R}}X_iX_j(a_k^Ha_\ell + a_ka_\ell^H)t_{ijk\ell}=0.
$$
Our strategy is to show that, for any fixed non-zero collection $\mathcal{Q}$, the probability of the event,
$$
E_{\mathcal{Q}}=\left\{\sum_{i=1}^p\sum_{k=1}^{\mathcal{R}} X_i^2|a_k|^2 q_{ik}+\sum_{1\leq i<j\leq p}\sum_{k=1}^{\mathcal{R}} X_iX_j|a_k|^2 r_{ijk} + \sum_{1\leq i<j\leq p}\sum_{1\leq k<\ell\leq \mathcal{R}}X_iX_j(a_k^Ha_\ell + a_ka_\ell^H)t_{ijk\ell}=0\right\}
$$
is zero. Once this is established, a union bound over all non-zero collection $\mathcal{Q}$ of rationals, of the above form yields,
$$
\mathbb{P}(\mathcal{L}\text{ is rationally independent}) \leq \sum_{\mathcal{Q}\in \mathbb{Q}^D\setminus {\bf 0}} \mathbb{P}(E_{\mathcal{Q}})=0,
$$
since we take a countable union of measure zero events, where $D=p\mathcal{R}+\binom{p}{2}\mathcal{R}+\binom{p}{2}\binom{\mathcal{R}}{2}=|\mathcal{L}|$. Now, we divide the proof into two sub-cases.
\begin{itemize}
    \item[(i)] Suppose, there is a pair $(i',k')\in[p]\times [\mathcal{R}]$ with $q_{i'k'}\neq 0$. Note that, rational independence of $\mathcal{S'}=\{|a_k|^2:k\in[\mathcal{R}]\}\cup\{a_k^Ha_\ell+a_k a_\ell^H:1\leq k<\ell\leq \mathcal{R}\}$ implies any subset of $\mathcal{S'}$ is also rationally independent, in particular, so does $\{|a_k|^2:k\in [\mathcal{R}]\}$; and also, $|a_i|^2>0$ for every $i$. Using this, the event, $E_{\mathcal{Q}}$ is equal to,
$\left\{X_{i'}^2 \gamma_{i'} + X_{i'}B_{i'} + C_{i'}=0\right\}$, where, $\gamma_{i'}=\sum_{k=1}^{\mathcal{R}} |a_j|^2q_{i'k}$ and both $B_{i'}$ and $C_{i'}$ are functions of $X_{\sim i'}$ only, hence, are independent of $X_{i'}$. Now, observe that $\gamma_{i'}=0$ if and only if $q_{i'k}=0$ for every $k\in[\mathcal{R}]$ as $\{|a_k|^2:k\in[\mathcal{R}]\}$ is rationally independent, but since $q_{i'k'}\neq 0$, it follows that $\gamma_{i'}\neq 0$. In particular, the event of interest is of form $\{P(X)=0\}$ for some polynomial $P:\R^p\to \R$, where $P\neq 0$. Thus, for this selection of $\mathcal{Q}$, we deduce using Theorem \ref{thm:auxiliary} that $\mathbb{P}(E_{\mathcal{Q}})=0$. 
\item[(ii)] Now, suppose that, $q_{ik}=0$, for every $i\in[p]$ and $k\in [\mathcal{R}]$. Now, for every $1\leq i<j\leq p$, define by $\xi_{ij}$ the number
$$
\xi_{ij}=\sum_{k=1}^{\mathcal{R}}|a_k|^2 r_{ijk}+\sum_{1\leq k<\ell \leq \mathcal{R}}(a_k^H a_\ell +a_k a_\ell^H)t_{ijk\ell}.
$$
In order not to deal with cases $i<j$ and $i>j$ separately, let us adopt the convention that $\xi_{ji}=\xi_{ij}$ for every $i\neq j$. Note that, since the set, $\mathcal{S'}=\{|a_k|^2:k\in[\mathcal{R}]\}\cup \{a_k^Ha_\ell +a_ka_\ell^H:1\leq k<\ell \leq \mathcal{R}\}$ is assumed to be rationally independent,  $\mathcal{Q}$ is a non-zero collection of rationals, and $q_{ij}=0$, for every $i\in[p],j\in[\mathcal{R}]$, we deduce that, there exists a pair, $(i_0,j_0)$ such that $\xi_{i_0j_0}\neq 0$. Keeping this in mind,
$$
E_{\mathcal{Q}}=\left\{\sum_{1\leq i<j\leq p}X_iX_j\xi_{ij}=0\right\}=\left\{A_{i_0}X_{i_0}+B_{i_0}=0\right\}\subset \{A_{i_0}=0\}\cup \{A_{i_0}X_{i_0}+B_{i_0}=0,A_{i_0}\neq 0\},
$$
where, $A_{i_0}=\sum_{j\neq i_0}X_j\xi_{i_0 j}$ and $B_{i_0}=\sum_{1\leq i<j\leq p, i,j\neq  i_0}X_iX_j\xi_{ij}$. Notice that, both $A_{i_0}$ and $B_{i_0}$ are polynomials in of $X_{\sim i_0}$, and are independent of $X_{i_0}$. Moreover, due to the fact that $\xi_{i_0,j_0}\neq 0$, we have that $A_{i_0}$ is not vanishing (this can be seen, for instance, by observing that taking $X_j =0$ for all $j\neq j_0$ and setting $X_{j_0}\neq 0$, the polynomial $A_{i_0}$ evaluates to $\xi_{i_0j_0}X_{j_0}$, which is clearly non-zero). Thus, using Theorem \ref{thm:auxiliary}, we have $\mathbb{P}(A_{i_0}=0)=0$. 
Next, we claim the probability of the event, $\{A_{i_0}X_{i_0}+B_{i_0}=0,A_{i_0}\neq 0\}$ is $0$. To see this, we proceed by conditioning on $X_{\sim i_0}$ such that $A_{i_0}\neq 0$, and observing that conditional on this, the event, $\{A_{i_0}X_{i_0}+B_{i_0}=0\}$, is simply the probability that a certain polynomial in $X_{i_0}$ is non-zero, which again by Theorem \ref{thm:auxiliary} happens with probability $0$.
\end{itemize}
Hence, in both cases, we have $\mathbb{P}(E_{\mathcal{Q}})=0$, and therefore, we are done with the proof of the lemma.
\end{proof}
Now, recall $D=|\mathcal{L}|$, and define ${\bf m}=(m_0,m_i:1\leq i\leq D)\in\mathbb{Z}^{D+1}\setminus \{{\bf 0}\}$ to be an integer relation, for the vector, $(Y^2,a:a\in \mathcal{L})$. Using the exact same steps, as in proof of Theorem \ref{thm:irrational-continuous}, we deduce, as a consequence of Lemma \ref{lemma:ell-rat-indep} that, with probability $1$, any integer relation for this vector must be a multiple of $(-1,\theta_a^*:a\in \mathcal{L})\in\mathbb{Z}^{D+1}$ where we have defined $\theta_a^*$ earlier as $Y^2=\sum_{a\in \mathcal{L}}a\theta_a^*$, that is, the expression of $Y^2$, as an integer combination of the elements of $\mathcal{L}$. Hence, for every $i$, $\beta_i^*= a_k$, where the coefficient $\theta^*_{X_i^2 |a_k|^2}$ in the integer relation corresponding to the product $X_i^2|a_k|^2$ is non-zero (with probability $1$, there must be a unique $k\in[\mathcal{R}]$ for which $\theta^*_{X_i^2|a_k|^2}\neq 0$). Finally, since the cardinality of $\mathcal{L}$ is at most $O(p^2\mathcal{R}^2)$, the overall runtime of the procedure is at most polynomial in $p$ and $\mathcal{R}$ (here, the  overall runtime includes generating the sets $\mathcal{S'}$ and $\mathcal{L}$, as well as running integer relation solver on a vector consisting of elements of $\mathcal{L}$ and $Y^2$).  \qedhere
\end{proof}

\subsection{Rest of the proofs}\label{Rest}
\subsubsection{Proof of Proposition \ref{cor2}}

\begin{proof}
If we show that we can apply Theorem \ref{mainiid}, the result follows. Since the model assumptions are identical we only need to check the parameter assumptions of Theorem \ref{mainiid}. First note that we assume $\hat{R}=R$, we clearly have for the noise $\sigma \leq W_{\infty}=1$ and finally $\hat{Q}=Q$.
Now for establishing \ref{eq:limit2}, we first notice that since $N \leq \log \left(\frac{1}{\sigma}\right)$ is equivalent to $2^N\sigma \leq 1$, we obtain $2^{N}\sigma\sqrt{np} +Rp \leq 2^{\log (np)+\log (Rp)}$. Therefore it suffices 
\begin{equation*}
N>\frac{(2n+p)^2}{2n}+22\frac{2n+p}{n}\log (3(1+c)np)+\frac{2n+p}{n} \log (RQ)
\end{equation*} Now since $p \geq \frac{300}{\epsilon} \log \left(\frac{300}{c \epsilon}\right)$ it holds \begin{equation}\label{eq:NEW}
22(2n+p)\log (3(1+c)np)<\frac{\epsilon}{2} \frac{(2n+p)^2}{2},\end{equation} for all $n \in \mathbb{Z}^+$.Indeed, this can be equivalently written as $$22<\frac{\epsilon}{4} \frac{2n+p}{\log (3(1+c)np)}. $$But $\frac{2n+p}{\log (3(1+c)np)}$ increases with respect to $n \in \mathbb{Z}^+$ and therefore it is minimized for $n=1$. In particular it suffices to have $$22<\frac{\epsilon}{4} \frac{2+p}{\log (3(1+c)p)}, $$which can be checked to be true for $p \geq \frac{300}{\epsilon} \log \left(\frac{300}{(1+c) \epsilon}\right)$. Therefore using (\ref{eq:NEW}) it suffices
 \begin{equation*}
N>(1+\frac{\epsilon}{2})\frac{(2n+p)^2}{2n}+\frac{2n+p}{n} \log (RQ).
\end{equation*}But observe 
\begin{align*} 
N  &\geq (1+\epsilon)\left[\frac{p^2}{2n}+2n+2p+(2+\frac{p}{n}) \log \left(RQ\right)\right]\\
&=(1+\epsilon)\left[\frac{(2n+p)^2}{2n}+(\frac{2n+p}{n}) \log \left(RQ\right)\right]\\
&> (1+\frac{\epsilon}{2})\frac{(2n+p)^2}{2}+(2n+p) \log (RQ).
\end{align*} The proof of Proposition \ref{cor2} is complete.
\end{proof}

\subsubsection{Proof of Proposition \ref{InfTh}}
\begin{proof}
We first establish that $\|X\|_{\infty} \leq (np)^2$ whp as $p \rightarrow +\infty$. By a union bound and Markov inequality \begin{align*}
\mathbb{P}\left(\max_{i \in [n],j \in [p]} |X_{ij}| > (np)^2\right) \leq np \mathbb{P}\left(|X_{11}|>(np)^2\right) \leq \frac{1}{np}\mathbb{E}[|X_{11}|] =o(1).
\end{align*} Therefore with high probability $\|X\|_{\infty} \leq (np)^2$. Consider the set $T(R,Q)$ of all the vectors $\beta^* \in [-R,R]^p$ satisfying the $Q$-rationality assumption. The entries of these vectors are of the form $\frac{a}{Q}$ for some $a \in \mathbb{Z}$ with $|a| \leq RQ$. In particular $|T(R,Q)|=\left(2QR+1\right)^p$. Now because the entries of $X$ are continuously distributed, all $X\beta^*$ with $\beta^* \in T(R,Q)$ are distinct with probability 1. Furthermore by the above each one of them has $L_2$ norm satisfies $$\|X\beta^*\|_2^2 \leq np^2 \|X\|_{\infty}^2 \|\beta^*\|_{\infty}^2 \leq R^2n^5p^6<R^2(np)^6,$$w.h.p. as $p \rightarrow + \infty$.

Now we establish the proposition by contradiction. Suppose there exist a recovery mechanism that can recover w.h.p. any such vector $\beta^*$ after observing $Y=X\beta^*+W \in \mathbb{R}^n$, where $W$ has $n$ iid $N(0,\sigma^2)$ entries. In the language of information theory such a recovery guarantee implies that the Gaussian channel with power constraint $R^2(np)^6$ and noise variance $\sigma^2$ needs to have capacity at least $$\frac{\log |T(R,Q)|}{n}=\frac{p \log \left(2QR+1\right)}{n}.$$ On the other hand, the capacity of this Gaussian channel with power $\mathcal{P}$ and noise variance $\Sigma^2$ is known to be equal to $\frac{1}{2}\log \left(1+\frac{\mathcal{P}}{\Sigma^2}\right)$ (see for example Theorem 10.1.1 in \cite{Cover}).  In particular our Gaussian communication channel has capacity $$\frac{1}{2}\log \left(1+\frac{R^2(np)^6}{\sigma^2}\right).$$From this we conclude 
\begin{align*}
\frac{p \log \left(2QR+1\right)}{n} \leq \frac{1}{2}\log \left(1+\frac{R^2(np)^6}{\sigma^2}\right),
\end{align*} which implies
\begin{align*}
\sigma^2 \leq R^2(np)^6 \frac{1}{2^{\frac{2p \log \left(2QR+1\right)}{n}}-1},
\end{align*} or $$\sigma \leq R(np)^3 \left(2^{\frac{2p \log \left(2QR+1\right)}{n}}-1\right)^{-\frac{1}{2}},$$ which completes the proof of the Proposition.
\end{proof}

\subsubsection{Proof of Proposition \ref{optimal}}
\begin{proof}
Based on Proposition \ref{cor2} the amount of noise that can be tolerated is $2^{-(1+\epsilon)\left[\frac{p^2}{2n}+2n+2p+(2+\frac{p}{n}) \log \left(RQ\right)\right]}$, for an arbitrary $\epsilon>0$. Since $n=o(p)$ and $RQ=2^{\omega(p)}$ this simplifies asymptotically to $2^{-(1+\epsilon)\left[\frac{p}{n}\log \left(RQ\right)\right]}$, for an arbitrary $\epsilon>0$. 
Since $\sigma<\sigma_0^{1+\epsilon}$, we conclude that LBR algorithms is succesfully working in that regime.

For the first part it suffices to establish that under our assumptions for $p$ sufficiently large, $$ \sigma_0^{1-\epsilon}>R(np)^3 \left(2^{\frac{2p \log \left(2QR+1\right)}{n}}-1\right)^{-\frac{1}{2}}.$$ Since $n=o(\frac{p}{\log p})$ implies $n=o(p)$ we obtain that for $p$ sufficiently large, \begin{equation*}2^{\frac{2p \log \left(2QR+1\right)}{n}}-1 > 2^{2(1-\frac{1}{2}\epsilon)\frac{p \log \left(2QR+1\right)}{n}} \end{equation*} which equivalently gives \begin{equation*}\left(2^{\frac{2p \log \left(2QR+1\right)}{n}}-1\right)^{-\frac{1}{2}}<2^{-(1-\frac{1}{2}\epsilon)\frac{p \log \left(2QR+1\right)}{n}} \end{equation*} or $$R(np)^3\left(2^{\frac{2p \log \left(2QR+1\right)}{n}}-1\right)^{-\frac{1}{2}}<R(np)^32^{-(1-\frac{1}{2}\epsilon)\frac{p \log \left(2QR+1\right)}{n}}.$$ Therefore it suffices to show $$R(np)^32^{-(1-\frac{1}{2}\epsilon)\frac{p \log \left(2QR+1\right)}{n}} \leq \sigma_0^{1-\epsilon}=2^{-(1-\epsilon)\frac{p \log \left(QR\right)}{n}}$$or equivalently by taking logarithms and performing elementary algebraic manipulations, \begin{equation*} n\log R+3n \log (np) \leq \left(1-\frac{\epsilon}{2}\right)p \log (2+\frac{1}{RQ} )+\frac{\epsilon}{2} p\log RQ.\end{equation*} The condition $n=o(\frac{p}{\log p})$ implies for sufficiently large $p$, $n \log (np) \leq \frac{\epsilon}{4} p $ and $ n\log R \leq  \frac{\epsilon}{2} p \log {QR} $. Using both of these inequalities we conclude that for sufficiently large $p$,
\begin{align*}
n\log R+3n \log (np) & \leq \frac{\epsilon}{2} p \log {QR} \\
& \leq \left(1-\frac{\epsilon}{2}\right)p \log (2+\frac{1}{RQ} )+\frac{\epsilon}{2} p\log RQ.
\end{align*}This completes the proof.

\end{proof}
\subsubsection*{Acknowledgments}

The authors would like to gratefully acknowledge the work of Patricio Foncea and Andrew Zheng on performing the synthetic experiments for the ELO and LBR algorithms, which appeared in 2018 NeurIPS paper \cite{zadik2018high}.
\bibliographystyle{apalike}
\bibliography{bibliography}

\end{document}